\documentclass[11pt]{amsart}
\usepackage{tabularx,booktabs,tikz}
\usepackage{caption}
\usepackage{amsmath}
\usepackage{tikz-cd}

\usepackage{amsfonts}

\usepackage{amscd}
\usepackage{amsthm}
\usepackage{amssymb} \usepackage{latexsym}
\usepackage{euscript}
\usepackage{epsfig}
\usepackage{graphics}
\usepackage{array}
\usepackage{enumerate}
\usepackage{dsfont}
\usepackage{color}
\usepackage{wasysym}
\usepackage{hyperref}
\usepackage{pdfsync}
\numberwithin{equation}{section}

\newcommand{\bel}[1]{\begin{equation}\label{#1}}

\newcommand{\be}{\begin{equation}}

\newcommand{\ba}{\begin{eqnarray}}
\newcommand{\ea}{\end{eqnarray}}

\newcommand{\qe}{\end{equation}}

\newcommand{\R}{{\mathbb R}}
\newcommand{\rn}{{\mathbb{R}^{n,1}}}

\newcommand{\ds}{\mathrm{d}\mathbb{S}^n}

\newcommand{\de}{{\Delta}}

\newcommand{\Span}{\mathrm{Span}}

\newcommand{\Hmm}[1]{\leavevmode{\marginpar{\tiny%
$\hbox to 0mm{\hspace*{-0.5mm}$\leftarrow$\hss}%
\vcenter{\vrule depth 0.1mm height 0.1mm width \the\marginparwidth}%
\hbox to
0mm{\hss$\rightarrow$\hspace*{-0.5mm}}$\\\relax\raggedright #1}}}

\newtheorem{theorem}{Theorem}[section]

\newtheorem{lemma}[theorem]{Lemma}
\newtheorem{corollary}[theorem]{Corollary}
\newtheorem{definition}[theorem]{Definition}

\newtheorem{remark}[theorem]{Remark}

\newtheorem{prop}[theorem]{Proposition}
\newtheorem{problem}[theorem]{Problem}

\newtheorem{example}[theorem]{Example}

\newcommand{\tm}{\begin{theorem}}
\newcommand{\tmd}{\end{theorem}}
\newcommand{\co}{\begin{corollary}}
\newcommand{\cod}{\end{corollary}}
\newcommand{\prp}{\begin{prop}}
\newcommand{\prpd}{\end{prop}}
\newcommand{\pf}{\begin{proof}}
\newcommand{\pfd}{\end{proof}}
\newcommand{\rmk}{\begin{remark}}
\newcommand{\rmkd}{\end{remark}}
\newcommand{\ex}{\begin{example}}
\newcommand{\exd}{\end{example}}
\newcommand{\pr}{\begin{problem}}
\newcommand{\prd}{\end{problem}}
\newcommand{\lm}{\begin{lemma}}
\newcommand{\lmd}{\end{lemma}}
\newcommand{\dfn}{\begin{definition}}
\newcommand{\dfnd}{\end{definition}}

\begin{document}

\title[Conformal tetrahedra and global rigidity of sphere packings ]{Conformal tetrahedra and global rigidity of sphere packings in constant curvature background geometry}

\author{Zunwu He}
\address{Zunwu He: School of Mathematics, South China University of Technology, Guangzhou, 510641, P.R. China.}
\email{hzwmath789@scut.edu.cn}

\author{Guangming Hu}
\address{Guangming Hu: College of Science, Jiangsu Key Laboratory of Quantum Computing Science and Devices,  Nanjing University of Posts and Telecommunications, Nanjing, 210003, P.R. China.}
\email{20230210@njupt.edu.cn}

\author{Ziping Lei}
\address{Ziping Lei: Department of Mathematical Sciences, Tsinghua University, Beijing, 100084, P.R. China.}
\email{zplei@ruc.edu.cn}

\begin{abstract}
In this paper, we give the dual characterization of conformal tetrahedra in constant curvature spaces, which develops a geometric technique to obtain the global rigidity of (ideal) hyperbolic, Euclidean, spherical sphere packings on a tetrahedron. Moreover, we introduce ideal hyperbolic sphere packings and extend the variational principle to obtain the global rigidity of (ideal) hyperbolic sphere packings on 3-manifolds. We also 
provide a characterization of conformal $n(\ge3)$-simplices in constant curvature spaces, which is useful for studying higher-dimensional sphere packings.
\end{abstract}

\maketitle


\par
\maketitle

\bigskip


\section{Introduction}
\subsection{Background}
Circle packings establish a deep and strong connection between the geometry and topology of 3-manifolds. Thurston \cite{Thurston97} introduced
circle packings on surfaces to construct hyperbolic 3-manifolds and
3-orbifolds. Colin de Verdiére \cite{Colin91} introduced the variational principle to prove the existence and rigidity of Thurston’s circle packings on surfaces, which provides an important and fundamental tool to study circle packings. We refer the readers to the works of Rivin \cite{Rivin}, Leibon \cite{Leibon}, Bobenko-Springborn \cite{bo}, Luo \cite{Luo12, Luo13,Luo}, Huang-Liu \cite{HL17}, Ge-Xu \cite{Gexu}, Ge-Hua-Zhou \cite{Zhou21a,Zhou21b}, Hu-Lei-Sun-Zhou \cite{HLSZ}, Zhou \cite{Zhou25}, Hu-Lu-Tan-Zhong-Zhou \cite{HL25} and so on for further developments of variational principle. 

Inspired by Thurston’s work, Cooper and Rivin \cite{Cooper96} introduced the sphere packings on 3-manifolds, which are the 3-dimensional analogue of circle packings. They defined the combinatorial scalar curvature and proved the local rigidity of the hyperbolic and Euclidean sphere packings with respect to this combinatorial curvature on 3-manifolds. Moreover, they conjectured that sphere packings have the property of global rigidity. However, this conjecture had remained open for a long time, since the sphere packings are defined on a non-convex domain and variational principle fails. 

To solve this conjecture, the method of extension was introduced, which seems to have originated from the work of Bobenko, Pinkall and Springborn \cite{2010Discrete}. 
Moreover, Luo \cite{Luo11} further systematically developed their approach and established extension theory. Following Luo's pioneering work, Ge and Jiang \cite{GJ16CV, 2017On, GJ17JFA2, GJ17JFA3} obtained some important rigidity results of inversive distance circle packings and discrete conformal factors on surfaces. In 3-dimensional geometry, Luo and Yang \cite{Luo2018Volume} extended the dihedral angles of decorated (hyper)ideal hyperbolic polyhedron and proved the rigidity of hyperbolic conical metrics on 3-manifolds, which are obtained by isometrically gluing (hyper)ideal tetrahedra. Following their ideas and works, Xu \cite{Xu20} continuously extended the solid angles and established the variational principle to obtain the global rigidity of hyperbolic and Euclidean sphere packings on 3-manifolds, which solved Cooper and Rivin’s conjecture and generalized Thurston’s rigidity result \cite{Thurston97} of circle packings on surfaces. Moreover, Xu and Zheng \cite{Xu23} introduced the generalized hyperbolic sphere packings and obtained the global rigidity of this sphere packings on 3-manifolds with boundary. 


Historically, the methods for obtaining rigidity of sphere packings all depend on the variational principle and are non-constructive. Moreover, the variational principle fails on spherical sphere packings since the functional defined by Cooper and Rivin \cite{Cooper96} is non-convex; see Theorem \ref{maintheorem3}. 

In this paper, we introduce the dual tetrahedra and generalize Luo's work \cite{Luo97} on the characterization of hyperbolic $n$-simplices to ideal hyperbolic $n$-simplices (see Theorem \ref{admissibleidealangle}) to obtain the dual characterization of conformal tetrahedra in constant curvature spaces; see Theorem \ref{main1}, which implies the global rigidity of hyperbolic, Euclidean, spherical and ideal hyperbolic sphere packings on a tetrahedron as well as the computability of the rigidity map; see Corollary \ref{rigidityoftetrahedron}. The rigidity results of spherical and ideal hyperbolic sphere packings are new to the literature as well as our approach is completely different from previous methods based on the variational principle. We also construct a counterexample on $\mathbb{S}^3$, which implies that the  global rigidity of spherical sphere packings cannot be generalized to 3-manifolds; see Example \ref{norigidity}.

Moreover, we introduce the ideal hyperbolic sphere packings and extend variational principle (see Theorem \ref{yu3}, Theorem \ref{negac}, Corollary \ref{dd1}) to obtain the global rigidity of (ideal) hyperbolic sphere packings on 3-manifolds; see Theorem \ref{theo1} and Corollary \ref{copo1}. As a higher-dimensional generalization of conformal tetrahedra, we introduce the conformal $n$-simplices and provide a characterization of conformal $n$-simplices in constant curvature spaces, which is useful for studying higher-dimensional sphere packings; see Theorem \ref{highpolyhedron}.


In addition to global rigidity, the existence of sphere packings on 3-manifolds is also an important research problem. Glickenstein \cite{Glickenstein05, Glickenstein052} introduced the combinatorial Yamabe flow to deform the sphere packings to one with constant curvature. In the forthcoming paper \cite{HHL}, we study the convergence
of this combinatorial curvature flow to obtain the ideal hyperbolic sphere packings with prescribed combinatorial scalar curvatures on 3-manifolds. We refer the readers to the works of Chow-Luo \cite{chow-Luo}, Ge-Hua \cite{Ge20}, Ge-Jiang-Shen \cite{GJS22Adv}, Xu-Zheng \cite{XZ23}, Ge-Hua-Zhou \cite{GE24}, Hu-Lei-Sun-Zhou \cite{HLSZ}, Hu-Lu-Tan-Zhong-Zhou \cite{HL25} and so on for further developments of combinatorial curvature flows.


\subsection{Main results}

\subsubsection{Global rigidity of sphere packings in constant curvature background geometry}\label{s2}

Given a compact 3-dimensional manifold $M$ with a triangulation $\mathcal{T} = \{V, E, F, T\}$, where $V, E, F, T$ denote the
sets of vertices, edges, faces and tetrahedra, respectively.

For any combinatorial tetrahedron $t\in T$ with vertex set $V_t$, define 
\begin{align*}
Q^{\mathbb{H}}_3(r_t):=&(\sum_{i\in V_t}\coth r_i)^2-2\sum_{i\in V_t}\coth^2 r_i+4,\\
(Q^\mathbb{E}_3(r_t):=&(\sum_{i\in V_t}\frac{1}{r_i})^2
-2\sum_{i\in V_t}\frac{1}{r^2_i},\\
Q^\mathbb{S}_3(r_t):=&(\sum_{i\in V_t}\cot r_i)^2
-2\sum_{i\in V_t}\cot^2 r_i-4,\text{resp.})
\end{align*}
for all $r_t=(r_i)_{i\in V_t}\in\mathbb{R}_+^{V_t}(\mathbb{R}_+^{V_t},(0,\pi)^{V_t},\text{resp}.)$. Moreover, define $\widehat{\mathbb{R}_{+}}:=(0,\infty]$ and $\coth \infty:=1$.

Then we can define a space $\Omega^{\mathbb{H}}(\Omega^{\mathbb{E}},\Omega^{\mathbb{S}},\partial_{\infty}\Omega^{\mathbb{H}},\text{resp}.)$ called the \emph{hyperbolic $($Euclidean, spherical, ideal hyperbolic, resp.$)$ admissible space on compact triangulated 3-manifold} $(M,\mathcal{T})$, namely
\begin{align*} 
\Omega^{\mathbb{H}}:=\{&r=(r_i)_{i\in V}\in\mathbb{R}_+^{V}:Q^{\mathbb{H}}_3(r_t)>0,~\text{for all}~t\in T\}\\
(\Omega^{\mathbb{E}}:=\{&r=(r_i)_{i\in V}\in\mathbb{R}_+^{V}:Q^\mathbb{E}_3(r_t)>0,~\text{for all}~t\in T\},\\
\Omega^{\mathbb{S}}:=\{&r=(r_i)_{i\in V}\in(0,\pi)^{V}:Q^\mathbb{S}_3(r_t)>0,~r_i+r_j\in(0,\pi), \\
&\text{for all}~i\ne j\in V_t~\text{and all}~t\in T\},\\
\partial_{\infty}\Omega^{\mathbb{H}}:=\{&r=(r_i)_{i\in V}\in\widehat{\mathbb{R}_+^{V}}\setminus\mathbb{R}_+^V:Q^{\mathbb{H}}_3(r_t)>0,~\text{for all}~t\in T\},\text{resp.}).
\end{align*}

Given a parameter $r=(r_i)_{i\in V}\in\Omega^{\mathbb{H}}(\Omega^{\mathbb{E}},\Omega^{\mathbb{S}},\partial_{\infty}\Omega^{\mathbb{H}},\text{resp}.)$, then for any tetrahedron $t\in T$, by Section \ref{s1}, there exists a corresponding hyperbolic (Euclidean, spherical, ideal hyperbolic, resp.) sphere packing $\mathcal{P}(r_t)$ and a conformal (see Section \ref{s1} for definition) hyperbolic (Euclidean, spherical, ideal hyperbolic, resp.) tetrahedron $\Delta^3(r_t)$. 

Gluing these conformal hyperbolic (Euclidean, spherical, ideal hyperbolic, resp.) tetrahedra along their faces together, then we obtain a hyperbolic (Euclidean, spherical, ideal hyperbolic, resp.) metric on the compact 3-dimensional manifold $M$, possibly with conical singularities at the vertices and edges. Moreover, it also has at least one cusp at the vertices in ideal hyperbolic case. This metric is also called the \emph{hyperbolic $($Euclidean, spherical, ideal hyperbolic, resp.$)$ sphere packing metric parameterized by} $r\in\Omega^{\mathbb{H}}(\Omega^{\mathbb{E}},\Omega^{\mathbb{S}}, \partial_{\infty}\Omega^{\mathbb{H}}, \text{resp}.)$. Adopting a common abuse of notation, $r\in\Omega^{\mathbb{H}}(\Omega^{\mathbb{E}},$ $\Omega^{\mathbb{S}}, \partial_{\infty}\Omega^{\mathbb{H}},\text{resp}.)$ is also called a hyperbolic (Euclidean, spherical, ideal hyperbolic, resp.) sphere packing metric. 

Moreover, the set of the hyperbolic (Euclidean, spherical, ideal hyperbolic, resp.) sphere packing $\mathcal{P}(r)=\{\mathcal{P}(r_t)\}_{t\in T}$ is called the \emph{hyperbolic $($Euclidean, spherical, ideal hyperbolic, resp.$)$ sphere packing parameterized by} $r\in\Omega^{\mathbb{H}}(\Omega^{\mathbb{E}},\Omega^{\mathbb{S}},\partial_{\infty}\Omega^{\mathbb{H}}, \text{resp}.)$. 

To study the geometric properties of the sphere packing (metric) on 3-manifolds, Cooper and Rivin \cite{Cooper96} introduced the combinatorial scalar curvature, which is the 3-dimensional analogue of discrete Gaussian curvature. 

For any hyperbolic (Euclidean, spherical, ideal hyperbolic, resp.) sphere packing (metric) $r=(r_i)_{i\in V}\in\Omega^{\mathbb{H}}(\Omega^{\mathbb{E}},\Omega^{\mathbb{S}},\partial_{\infty}\Omega^{\mathbb{H}},\text{resp}.)$, the combinatorial scalar curvature $K_i$ at the vertex $i\in V$ is defined as
\begin{equation}\label{q96}
K_i:=\begin{cases}
 4\pi-\sum_{t\in T_i}\alpha_i^t, & \text{ if }i~ \text{is an interior vertex},  \\
  2\pi-\sum_{t\in T_i}\alpha_i^t,& \text{ if }i~ \text{is a boundary vertex}, 
\end{cases}
\end{equation}
where $\alpha_i^t$ is the solid angle at the vertex $i$ of the hyperbolic (Euclidean, spherical, ideal hyperbolic, resp.) tetrahedron $\Delta^3(r_t)$ and $T_i$ is the set of all tetrahedra in $T$ with vertex $i$. 

We can construct the \emph{combinatorial scalar curvature map}
\begin{equation}\label{cone}
		\begin{aligned}
			\mathbb{K}: \Omega^{\mathbb{H}}(\Omega^{\mathbb{E}},\Omega^{\mathbb{S}},\partial_{\infty}\Omega^{\mathbb{H}},\text{resp}.) & \longrightarrow \mathbb{R}^{V}  \\
			r=(r_i)_{i\in V} & \longmapsto K=(K_i)_{i\in V}.
		\end{aligned}
\end{equation}

The classical method (see e.g. \cite{Ge20}, \cite{GJS22Adv}, \cite{Xu20}) to prove the global rigidity of the hyperbolic and Euclidean sphere packings with respect to the combinatorial scalar curvatures on 3-manifolds is to extend the map $\mathbb{K}$ from the admissible spaces $\Omega^{\mathbb{H}},\Omega^{\mathbb{E}}$ to the open convex space $\mathbb{R}_+^V$ and construct the extended convex functional on hyperbolic and Euclidean sphere packings. 

However, in spherical case, we prove Theorem \ref{maintheorem3}, which implies that the functional defined by Cooper and Rivin \cite{Cooper96} is non-convex and the classical method fails.

In this paper, we systematically develop a  geometric technique instead of the extensional and variational method to obtain the global rigidity of hyperbolic, Euclidean, spherical and ideal hyperbolic sphere packings on a tetrahedron, which is constructive and based on geometric observations. 

Our approach relies crucially on the following dual characterization of conformal tetrahedra in constant curvature spaces, which implies the conformal equivalence of tetrahedra and dual tetrahedra (see Section \ref{s3} for definition).

\begin{theorem}\label{main1}
A non-degenerate hyperbolic $($Euclidean, spherical, ideal hyperbolic, resp.$)$ tetrahedron $\Delta^3$ is conformal if and only if the dual tetrahedron $(\Delta^{3})^*$ is conformal.  
\end{theorem}

\rmk
The dual characterization is trivial in 2-dimensional constant curvature spaces. In this paper, the conformality of a $($dual$)$ $n$-simplex is equivalent to its edge lengths satisfy a certain additivity condition, namely $($\ref{duedge}$)$ and $($\ref{duedg1}$)$ hold for all choices of four vertices. 

Moreover, the vertex of a dual $n$-simplex is determined by the $(n-1)$-dimensional face spanned by $n$ vertices of the $n$-simplex $($see Section \ref{s3}$)$, which implies that the conformal equivalence of $n$-simplices and dual $n$-simplices is only possible if $n+1=4$, namely the dual characterization $($Theorem \ref{main1}$)$ cannot be generalized to dimension $n\ge4$.  


\rmkd

As a consequence, we obtain another equivalent characterization of conformal tetrahedra in constant curvature spaces:

\begin{corollary}\label{sumangle}
A non-degenerate hyperbolic $($Euclidean, spherical, ideal hyperbolic, resp.$)$ tetrahedron $\Delta^3$ is conformal if and only if the three pairs of the sums of opposite dihedral angles of $\Delta^3$ are equal.   
\end{corollary}

\begin{remark}
After finishing the draft of this paper, we were reminded that the Euclidean version of Corollary \ref{sumangle} had been proved by Katsuura \cite{Katsuura2021Dihedral}, which depends on a complicated calculation. In this paper, we develop an original and independent technique to obtain Corollary \ref{sumangle}, which includes the known Euclidean result.
\end{remark}

\begin{remark}
In 3-dimensional constant curvature spaces, the solid angles of a tetrahedron are linear functions of its dihedral angles. However, the reverse statement is generally not true for reasons of degrees of freedom. 

In this paper, Theorem \ref{main1} or Corollary \ref{sumangle} implies that the reverse statement always holds on conformal tetrahedra, namely the dihedral angles are linear functions of the solid angles for all hyperbolic $($Euclidean, spherical, ideal hyperbolic, resp.$)$ conformal tetrahedra.
\end{remark}

The proof of ideal hyperbolic case of Theorem \ref{main1} depends on the following observation, which gives a characterization of (ideal) hyperbolic $n$-simplices and is the generalization of Luo's work \cite{Luo97}.

\begin{theorem}\label{admissibleidealangle}
Given a set of positive numbers
$$
\{\theta_{ij}:\theta_{ii}=\pi,\theta_{ij}=\theta_{ji}\in(0,\pi),1\le i, j\le n+1\},
$$
then there exists a $($ideal$)$ hyperbolic $n$-simplex $\Delta^n\subset\overline{\mathbb{H}^n}$ with $n+1$ vertices $v_1,\cdots,v_{n+1}$ such that the dihedral angle of the $ij$-face of $\Delta^n$ is $\theta_{ij}$ for all $ij$-faces if and only if the matrix $\Theta:=(\cos(\pi-\theta_{ij}))_{1\le i,j\le n+1}$ satisfies
\begin{enumerate}[1.]
\item the determinant is negative;
\item the $n\times n$ principal submatrix $\Theta(i,i)$ is positive definite for all vertex $v_i\in\mathbb{H}^n$;
\item the $n\times n$ principal submatrix $\Theta(i,i)$ is singular and all $(n-1)\times(n-1)$ principal submatrices of $\Theta(i,i)$ are positive definite for all ideal vertex $v_i\in\partial_{\infty}\mathbb{H}^n$;
\item the $(i,j)$-cofactors $(i\ne j)$ are positive.
\end{enumerate}
\end{theorem}

In 2-dimensional hyperbolic space $\mathbb{H}^2$, it is well-known that the ideal triangle is unique up to isometry. However, there is no 3-dimensional analogue of this fact, namely the \emph{ideal tetrahedron} is not unique up to isometry, where the ideal tetrahedron is an ideal hyperbolic tetrahedron (see Section \ref{s1} for definition) with four ideal vertices.  

In this paper, as a consequence of Theorem \ref{main1}, we have the following result, which gives the uniqueness on a class of ideal tetrahedra.

\begin{corollary}\label{ck1}
The conformal ideal tetrahedron is unique up to isometry.
\end{corollary}

As another consequence of Theorem \ref{main1}, we obtain the following result of global rigidity of hyperbolic, Euclidean, spherical and ideal hyperbolic sphere packings on a tetrahedron. Moreover, the rigidity map, namely the inverse of combinatorial scalar curvature map, can be expressed explicitly.

\co\label{rigidityoftetrahedron}
Given a triangulated tetrahedron $(\Delta^3,\mathcal{T}=(V,E,F,T))$ with $|T|=1$, then the hyperbolic $($Euclidean, spherical, ideal hyperbolic, resp.$)$ sphere packing $($metric$)$ on $(\Delta^3,\mathcal{T})$ is determined by the combinatorial scalar curvature map $\mathbb{K}: \Omega^{\mathbb{H}}(\Omega^{\mathbb{E}},\Omega^{\mathbb{S}},\partial_{\infty}\Omega^{\mathbb{H}},\text{resp}.)  \longrightarrow \mathbb{R}^{V}$$($up to scaling in Euclidean case$)$. 

\cod

\rmk
To the best of our knowledge, this is the first result of global rigidity of spherical and ideal hyperbolic sphere packings on a tetrahedron. There are also many local and global rigidity results of hyperbolic and Euclidean sphere packings on 3-manifolds, which depend on the extension method and variational principle, see e.g. \cite{Cooper96}, \cite{Glickenstein05}, \cite{Glickenstein052}, \cite{Xu20}, \cite{Ge21}, \cite{Xu23}. 
\rmkd

It is worth highlighting that the spherical result of Corollary \ref{rigidityoftetrahedron} cannot be generalized to 3-manifolds, see e.g. the following counterexample we have constructed.

\begin{example}[\textbf{No local and global rigidity of spherical sphere packings on $\mathbb{S}^3$}]\label{norigidity}
Given a Euclidean 4-simplex $\Delta^4$$($all edge lengths are equal$)$ inscribed in the 3-dimensional spherical space $\mathbb{S}^3$, then $\partial\Delta^4$ corresponds to a triangulation $\mathcal{T}=(V,E,F,T)$ on $\mathbb{S}^3$ and we can construct a spherical tetrahedron on $\mathbb{S}^3$ for all tetrahedron $t\in T$ by scaling. 

Gluing these spherical tetrahedra along their faces together, we obtain the round metric on $\mathbb{S}^3$ and this metric is parameterized by a constant spherical sphere packing $($metric$)$ $r\in \Omega^{\mathbb{S}}$, which satisfies $\mathbb{K}(r)=0$. 

Since $\mathbb{S}^3$ is also the boundary at infinity of Poincaré model of 4-dimensional hyperbolic space $\mathbb{H}^4$, then we can construct a conformal map on $\mathbb{S}^3$, which is induced from a non-elliptic isometry sufficiently close to the identity on $\mathbb{H}^4$ $($e.g. a translation map with a small enough translation distance$)$. 

Then we can obtain another spherical sphere packing $($metric$)$ $r^{\prime}\ne r\in \Omega^{\mathbb{S}}$ sufficiently close to $r$ such that $\mathbb{K}(r^{\prime})=0$ by the conformal map, which implies that there is no global and local rigidity of spherical sphere packings on $\mathbb{S}^3$. 
\end{example}

We introduce a ``compactification'' $\widehat{\Omega}^{\mathbb{H}}$ of the space $\Omega^{\mathbb{H}}$ called the \emph{extended hyperbolic admissible space} on compact triangulated 3-manifold $(M,\mathcal{T})$, namely 
$$
\widehat{\Omega}^{\mathbb{H}}:=\{ r=(r_i)_{i\in V}\in\widehat{\mathbb{R}_+^{V}}:Q^{\mathbb{H}}_3(r_t)>0,~\text{for all}~t\in T\}.
$$

It is not difficult to check that $\widehat{\Omega}^{\mathbb{H}}=\Omega^{\mathbb{H}}\bigsqcup\partial_{\infty}\Omega^{\mathbb{H}}$ and $\partial_{\infty}\Omega^{\mathbb{H}}$ is called the \emph{boundary at infinity} of the hyperbolic admissible space $\Omega^{\mathbb{H}}$. Moreover, the two combinatorial scalar curvature maps, which are defined on $\Omega^{\mathbb{H}}$ and $\partial_{\infty}\Omega^{\mathbb{H}}$, respectively, can be combined and written as 
$$
\mathbb{K}:\widehat{\Omega}^{\mathbb{H}}=\Omega^{\mathbb{H}}\bigsqcup\partial_{\infty}\Omega^{\mathbb{H}}\longrightarrow \mathbb{R}^{V}.
$$

In this paper, we give a cell decomposition of the extended hyperbolic admissible space $\widehat{\Omega}^{\mathbb{H}}$,namely
\begin{equation}
\widehat{\Omega}^{\mathbb{H}}=\Omega^{\mathbb{H}}\bigsqcup\partial_{\infty}\Omega^{\mathbb{H}}=\bigsqcup_{W\subset V}\Omega^{W},
\end{equation} 
where $\Omega^{V}=\Omega^{\mathbb{H}}$, $\partial_{\infty}\Omega^{\mathbb{H}}=\bigsqcup_{W\subsetneq V}\Omega^{W}$ and the subspace $\Omega^{W}$ is a $|W|$-dimensional cell of the extended hyperbolic admissible space $\widehat{\Omega}^{\mathbb{H}}$ for all subset $W\subset V$.

Moreover, we establish and extend the variational principle (see Theorem \ref{yu3}, Theorem \ref{negac}, Corollary \ref{dd1}) to obtain the following rigidity result of (ideal) hyperbolic sphere packings on each cell. 

\begin{theorem}\label{theo1}
 Given a compact triangulated 3-manifold $(M,\mathcal{T}=(V,E,F,T))$, then the $($ideal$)$ hyperbolic sphere packing $($metric$)$ on $\Omega^{W}$ is determined by the combinatorial scalar curvature map $\mathbb{K}:\Omega^{W}\longrightarrow \mathbb{R}^{V}$ for all subset $W\subset V$.  
 \end{theorem}

As a consequence, we  immediately obtain the following global rigidity of (ideal) hyperbolic sphere packings on 3-manifolds.

\begin{corollary}\label{copo1}
 Given a compact triangulated 3-manifold $(M,\mathcal{T}=(V,E,F,T))$, then the hyperbolic and ideal hyperbolic sphere packing $($metric$)$ on $(M,\mathcal{T})$ are determined by the combinatorial scalar curvature map $\mathbb{K}:\widehat{\Omega}^{\mathbb{H}}\longrightarrow \mathbb{R}^{V}$.  \end{corollary}

\subsubsection{ Characterization of conformal $n$-simplices in constant curvature spaces} 

In Section \ref{s1}, we have defined the conformal hyperbolic (Euclidean, spherical, ideal hyperbolic, resp.) tetrahedra. Similarly, we can also generalize the concept ``conformal'' to $n$-simplices (see Section \ref{s3} for definition).

 Taylor and Woodward \cite{Taylor05} provided a nice criterion of non-degenerate hyperbolic (Euclidean, spherical, resp.) $n$-simplices. Following their work, we obtain the following result, which is the higher-dimensional generalization of the characterization of conformal hyperbolic (Euclidean, spherical, ideal hyperbolic, resp.) tetrahedra. The result is useful for the understanding of higher-dimensional sphere packings in constant curvature background geometry.

\begin{theorem}\label{highpolyhedron}
Given a dimension $n\ge 3$ and a parameter $r=(r_i)_{1\le i\le n+1}\in\mathbb{R}_+^{n+1}(\mathbb{R}_+^{n+1},$ $(0,\pi)^{n+1}, \widehat{\mathbb{R}_+^{n+1}}\setminus\mathbb{R}_+^{n+1},\text{resp}.)$$(r_i+r_j(i\ne j)\in(0,\pi)$ in spherical case$)$, then the hyperbolic $($Euclidean, spherical, ideal hyperbolic, resp.$)$ $n$-simplex $\Delta^n\subset\mathbb{H}^n(\mathbb{E}^n, \mathbb{S}^n$, $\overline{\mathbb{H}^n},\text{resp}.)$ with edge length $\ell_{ij}=r_i+r_j$ for all geodesic edges $e_{ij}$ of $\Delta^n$ is conformal if and only if  
\begin{align*}
Q^{\mathbb{H}}_n(r):=&(\sum\limits_{i=1}^{n+1}\coth r_i)^2-(n-1)\sum\limits_{i=1}^{n+1}\coth^2 r_i+2(n-1)>0\\\nonumber
(Q^\mathbb{E}_n(r):=&(\sum\limits_{i=1}^{n+1}\dfrac{1}{r_i})^2-(n-1)\sum\limits_{i=1}^{n+1}\dfrac{1}{r^2_i}>0,\label{decar}\\\nonumber
Q^\mathbb{S}_n(r):=&(\sum\limits_{i=1}^{n+1}\cot r_i)^2-(n-1)\sum\limits_{i=1}^{n+1}\cot^2 r_i-2(n-1)>0,\\\nonumber
Q^{\mathbb{H}}_n(r)>&0, \text{resp}.).\nonumber
\end{align*}
\end{theorem}
\rmk\label{aa2}
Theorem \ref{highpolyhedron} also implies the higher-dimensional Descartes circle Theorem $($see e.g. \cite{LMW01}$)$: Given a dimension $n\ge 2$, then $n+2$ $(n-1)$-dimensional hyperbolic $($Euclidean, spherical, hyperbolic or horo, resp.$)$ spheres with radii $r_1,\cdots,r_{n+2}$ in $n$-dimensional hyperbolic $($Euclidean, spherical, compactified hyperbolic, resp.$)$ space $\mathbb{H}^{n}(\mathbb{E}^n, \mathbb{S}^n, \overline{\mathbb{H}^n},\text{resp}.)$ are mutually externally tangent if and only if
$$
Q^{\mathbb{H}}_{n+1}(r)(Q^{\mathbb{E}}_{n+1}(r),Q^{\mathbb{S}}_{n+1}(r),Q^{\mathbb{H}}_{n+1}(r),\text{resp}.)=0,
$$
where $r=(r_i)_{1\le i\le n+2}$.
\rmkd

\textbf{Organization.} This paper is organized as follows. In Section \ref{sd1}, we introduce some preliminaries on constant curvature geometry, conformal tetrahedra and sphere packings. In Section \ref{s3}, we introduce and study the (dual) $n$-simplices in constant curvature spaces. In Section \ref{s4}, we introduce the Gram matrix of (dual) $n$-simplices and prove Theorem \ref{admissibleidealangle}. In Section \ref{s5}, we study the admissible space of conformal tetrahedra in constant curvature spaces. In Section \ref{sdual}, we prove Theorem \ref{main1}, Corollary \ref{sumangle}, Corollary \ref{ck1} and Corollary \ref{rigidityoftetrahedron}. In Section \ref{boundary}, we study the admissible space of conformal ideal hyperbolic tetrahedra and construct the extended concave functional. In Section \ref{s8}, we prove Theorem \ref{theo1} and Corollary \ref{copo1}. In Section \ref{s9}, we prove Theorem \ref{highpolyhedron}.

\textbf{Declaration.} We do not use any AI tools in this work and in the preparation of this paper.

\section{Preliminaries}\label{sd1}

In this section, we introduce some preliminaries on constant curvature geometry. Moreover, we introduce the conformal tetrahedra and sphere packings in constant curvature spaces.   

We first explain some mathematical notation used in this paper:

\begin{enumerate}
\item $d_{\kappa}(x,y)$: the  hyperbolic (Euclidean, spherical, resp.) distance betweeen the two points $x$ and $y$ as $\kappa=-1 (0, 1, \text{resp}.)$.
\item $g_{\kappa}$: the standard hyperbolic (Euclidean, spherical, resp.) metric as  $\kappa=-1 (0, 1, \text{resp}.)$.
\item $\R^n_{+}:=\{x=(x_1,\cdots,x_n)\in\R^n:x_1,\cdots,x_n>0\}.$
\item $S(p,r)$: the geodesic sphere with center $p$ and radius $r$.
\item The inertia or signature of a symmetric matrix is $(p,r,q)$, where $p,r,q$ are the number of positive, zero and negative eigenvalues, respectively.
\item The Euclidean inner product $\langle x,y\rangle:=\sum\limits_{i=1}^{n}x_iy_i$, where $x=(x_1,\cdots,x_n)$,\\$y=(y_1,\cdots,y_n)\in\R^{n}.$
\item The Minkowski inner product $\langle x,y\rangle_{n,1}:=\sum\limits_{i=1}^nx_iy_i-x_{n+1}y_{n+1}$, where $x=(x_1,\cdots,x_{n+1})$, $y=(y_1,\cdots,y_{n+1})\in\R^{n+1}.$
\item $\mathbb{R}^{n,1}$: the $(n+1)$-dimensional Minkowski space $(\mathbb{R}^{n+1},\langle\cdot ,\cdot \rangle_{n,1})$.  
\item $\mathbb{E}^{n}$: the $n$-dimensional Euclidean space $(\mathbb{R}^{n},\langle\cdot ,\cdot \rangle)$.
\item $v^T$: the transpose of a column vector $v$.
\item $I_{n,1}:=\left(\begin{array}{cc}I_n & 0 \\ 0 & -1\end{array}\right)$.
\item $M^T$: the transpose of a matrix $M$.
\item $M^{ad}$: the adjoint matrix of a square matrix $M$.
\item $M^{-1}$: the inverse matrix of a square matrix $M$.
\item $|M|$: the determinant of a square matrix $M$. 
\item $M(i,j)$: the submatrix obtained from a square matrix $M$ by removing the $i$-th row and $j$-th column from $M$. 
\item $M[i,j]$: the submatrix obtained from a square matrix $M$ by removing the $i,j$-th rows and $i,j$-th columns from $M$. 
\item $M_{ij}$: the $(i,j)$-cofactor of a square matrix $M$. 
\item The matrix inner product $\langle M,N\rangle_{n,1}:=M^TI_{n,1}N$.
\item The matrix inner product $\langle M,N\rangle:=M^TN$.
\item $(i,j,k,l)$: the permutation of $(1,2,3,4)$. 

\end{enumerate}
\textbf{Notation statement}. \emph{We may abuse some notations for the sake of convenience, if there are no confusions from the context. In this paper, the dimension $n$ of all spaces considered is greater than $2$, namely $n\ge3$. }

\subsection{Constant curvature geometry}
\subsubsection{Spherical geometry}
The $n$-sphere $S^n$ is the set of all points with norm 1 in $(n+1)$-Euclidean space $\mathbb{E}^{n+1}$, namely
$$
S^n:=\left\{x=(x_1,\cdots,x_{n+1})\in\mathbb{E}^{n+1}:\langle x,x\rangle=1\right\}.
$$
It is well-known that the sphere $S^n$ equipped with the metric induced by the Euclidean inner product $\langle\cdot ,\cdot \rangle$, is the classical $n$-dimensional spherical space $\mathbb{S}^n$, namely $\mathbb{S}^n=(S^n,\langle\cdot ,\cdot \rangle)$. Moreover, for any two points $p,q\in \mathbb{S}^n$, we have
\begin{equation}\label{julis}
\langle p,q \rangle=\cos d_{1}(p,q).
\end{equation}

\subsubsection{De Sitter geometry}
Similarly, we can define the following pseudo $n$-spheres $S^{n}(1)$ and $S^n(-1)$ with norm 1 and $-1$ in $(n+1)$-Minkowski space $\rn$, respectively.
$$
S^n(1):=\left\{x=(x_1,\cdots,x_{n+1})\in\mathbb{R}^{n,1}:\langle x,x\rangle_{n,1}=1\right\},
$$
$$
S^n(-1):=\left\{x=(x_1,\cdots,x_{n+1})\in\mathbb{R}^{n,1}:\langle x,x\rangle_{n,1}=-1\right\}.
$$

The pseudo sphere $S^{n}(1)$ equipped with the metric induced by the Minkowski inner product, is the $n$-dimensional \emph{de Sitter space} $\ds$, namely $\ds=(S^{n}(1),\langle\cdot ,\cdot \rangle_{n,1})$. Note that $\ds$ is homeomorphic to $\mathbb{S}^{n-1}\times\mathbb{R}$ and not simply connected.

By a simple argument, the signature of restriction of the Minkowski inner product $\langle \cdot,\cdot\rangle_{n,1}$ to the tangent space $T_x\ds=\Span\{x\}^{\perp}=\{v\in\mathbb{R}^{n,1}:\langle x,v\rangle_{n,1}=0\}$ is $(n-1,0,1)$ for all $x\in\ds$, which implies that $\ds$ is a $n$-dimensional Lorentzian manifold. Moreover, $\ds$ has a canonical Lorentzian metric of constant sectional curvature $1$.

A \emph{geodesic} $\gamma$ in $\ds$ is a curve $\gamma:(-\infty,+\infty)\rightarrow\ds$ which satisfies $\nabla_{\gamma'}\gamma'=0$ for all time $s\in (-\infty,+\infty)$, where $\nabla$ is the Levi-Civita connection on $\ds$.

Given a point $p\in\ds$ and a non-zero tangent vector $ v\in T_p\ds$ such that $\langle v,v\rangle_{n,1}=1,~\text{or}~0,~\text{or}~-1$, then the geodesic passing through the point $p$ with velocity $v$ can be parameterized as   
\begin{equation}\label{geodesic1}
    \gamma(s)=\begin{cases} 
\cos s\cdot p+\sin s\cdot v, & \text{if } \langle v,v\rangle_{n,1}=1, \\
p+s\cdot v, & \text{if } \langle v,v\rangle_{n,1}=0,\\
\cosh s\cdot p+\sinh s\cdot v, & \text{if } \langle v,v\rangle_{n,1}=-1 .\\
\end{cases}
\end{equation}
for all time $s\in(-\infty.+\infty)$.

Given two distinct points $p,q\in\ds$, by (\ref{geodesic1}), then there exists a geodesic passing through the points $p,q$ if and only if $\langle p,q\rangle_{n,1}>-1$ or $p=-q$, and this geodesic can be parameterized as
\begin{equation}
    \gamma(s)=\begin{cases} 
\cos s\cdot p+\sin s\cdot \dfrac{q-p\langle p,q\rangle_{n,1}}{\sqrt{1-\langle p,q\rangle_{n,1}^2}}, & \text{if } |\langle p,q\rangle_{n,1}|<1, \\
p+s\cdot (q-p), & \text{if } \langle p,q\rangle_{n,1}=1,\\
\cosh s\cdot p+\sinh s\cdot \dfrac{q-p\langle p,q\rangle_{n,1}}{\sqrt{\langle p,q\rangle_{n,1}^2-1}}, & \text{if } \langle p,q\rangle_{n,1}>1 ,\\
\cos s\cdot p+\sin s\cdot v,& \text{if }  p=-q.\\
\end{cases}
\end{equation}
for all time $s\in(-\infty,+\infty)$, where $v\in T_p\ds$ is a unit tangent vector. Note that the geodesic $\gamma$ passing through the points $p,q$ is not unique if $p=-q$.



For any hyperplane $H\subset\rn$ with timelike normal vector, the restriction of the Minkowski inner product $\langle\cdot ,\cdot \rangle_{n,1}$ to the intersection of $H$ and de sitter space $\ds$ is positive definite, with signature $(n-1,0,0)$, which implies that $H\cap\ds$ is a spacelike, totally geodesic hyperplane in $\ds$. Moreover, the induced metric on $H\cap\ds$ is a round metric, namly $H\cap\ds=\mathbb{S}^{n-1}$.      


\subsubsection{Hyperbolic geometry}

The pseudo sphere $S^{n}(-1)$ is a hyperboloid with two sheets,namely
$$
S(-1)=I^n_+\sqcup I^n_{-},
$$
where
$$I^n_+:=\{x=(x_1,\cdots,x_{n+1})\in\rn:\langle x,x\rangle_{n,1}=-1,x_{n+1}>0\},$$
$$I^n_-:=\{x=(x_1,\cdots,x_{n+1})\in\rn:\langle x,x\rangle_{n,1}=-1,x_{n+1}<0\}.$$

It is well known that the hyperboloid $I^n_+$ equipped with the metric induced by the Minkowski inner product, is the classical model of the $n$-dimensional hyperbolic space $\mathbb{H}^n$, namely $\mathbb{H}^n=(I^n_+,\langle\cdot ,\cdot \rangle_{n,1})$. Moreover, for any two points $p,q\in \mathbb{H}^n$, we have
\begin{equation}\label{juli}
\langle p,q \rangle_{n,1}=-\cosh d_{-1}(p,q).
\end{equation}

A \emph{geodesic} $\gamma$ in $\mathbb{H}^n$ is a curve $\gamma:(-\infty,+\infty)\rightarrow\mathbb{H}^n$ which satisfies $\nabla_{\gamma'}\gamma'=0$ for all $s\in (-\infty,+\infty)$, where $\nabla$ is the Levi-Civita connection on $\mathbb{H}^n$. A \emph{half-geodesic} $\gamma^+$ in $\mathbb{H}^n$ is the restriction of a geodesic $\gamma$ in $\mathbb{H}^n$ to the interval $[0,+\infty)$, namely $\gamma^+=\gamma|_{[0,+\infty)}$. 

Given a point $p\in\mathbb{H}^n$ and a unit tangent vector $v\in T_p\mathbb{H}^n=\Span\{p\}^{\perp}=\{v\in\mathbb{R}^{n,1}:\langle p,v\rangle_{n,1}=0\}$, it is well known that the geodesic passing through the point $p$ with velocity $v$ can be parameterized as   
\begin{equation}\label{geodesic}
    \gamma(s)=\cosh s\cdot p+\sinh s\cdot v,~s\in(-\infty,+\infty).
\end{equation}

Given two distinct points $p,q\in\mathbb{H}^n$, then $\langle p,q\rangle_{n,1}<-1$ and the geodesic passing through the points $p,q$ can be parameterized as
\begin{equation}
    \gamma(s)=\cosh s\cdot p+\sinh s\cdot \frac{q+p\langle p,q\rangle_{n,1}}{\sqrt{\langle p,q\rangle_{n,1}^2-1}},~s\in(-\infty,+\infty).
\end{equation}

The \emph{boundary at infinity} $\partial_{\infty}\mathbb{H}^n$ of $\mathbb{H}^n$ is the set of the equivalence classes of all half-geodesics in $\mathbb{H}^n$, up to the following equivalence relation:
$$
\gamma_1^+ \sim \gamma_2^+ \Longleftrightarrow \sup _{s \in[0,+\infty)}
d_{-1}\left(\gamma_1(s), \gamma_2(s)\right)
<+\infty.
$$
We denote the equivalence class of a half-geodesic $\gamma^+$ in $\partial_{\infty}\mathbb{H}^n$ by $[\gamma^+]$ and the equivalence class $[\gamma^+]$ is also called a \emph{point at infinity}.  

The \emph{ideal boundary} $\partial I^n_+$ of $\mathbb{H}^n$ is the set of the equivalence classes of all points in the future light cone $L^+$, up to the following equivalence relation:
$$
x\sim y \Longleftrightarrow \exists ~c>0,~s.t.~x=cy,
$$
where $L^+:=\{x=(x_1,\cdots,x_{n+1})\in\rn:\langle x,x\rangle_{n,1}=0, x_{n+1}>0\}$. We denote the equivalence class of a point $x\in L^+$ in $\partial I^n_+$ by $[x]$ and the equivalence class $[x]$ is also called an \emph{ideal point}.   

Given a half-geodesic $\gamma^+$, it is well-known that $\gamma(0)+{\gamma}'(0)\in L^+$. Moreover, $\gamma^+_1\sim \gamma_2^+$ if and only if $\gamma_1(0)+{\gamma}_1'(0)\sim \gamma_2(0)+{\gamma}_2'(0)$. Then we can construct the following map from $\partial_{\infty}\mathbb{H}^n$ to $\partial I^n_+$:
$$
\begin{aligned}
 \partial_{\infty}\mathbb{H}^n & \longrightarrow \partial I^n_+ \\
[\gamma^+]  & \longmapsto [\gamma(0)+{\gamma}'(0)].
\end{aligned}
$$

Given a point $p\in\mathbb{H}^n$ and an ideal point $[x]\in\partial I^n_+$, we have $\langle p,x\rangle_{n,1}<0$ and there exists a unique (half-)geodesic $\gamma_{p,[x]}^{(+)}$ such that 
$$
\gamma_{p,[x]}(0)=p,~[\gamma_{p,[x]}(0)+{\gamma}'_{p,[x]}(0)]=[x],
$$
where $\gamma_{p,[x]}(s)=\cosh s\cdot p+\sinh s\cdot v,~ s\in(-\infty,+\infty)$ and $v=-\dfrac{x}{\langle p,x\rangle_{n,1}}-p$. Moreover, $\gamma_{p,[x]}^{(+)}$ is called a (half-)geodesic passing through the point $p$ and the ideal point $[x]$. 

Given two distinct ideal points $[x],[y]\in\partial I^n_+$, we have $\langle x,y\rangle_{n,1}<0$ and there exists a unique (half-)geodesic $\gamma_{[x],[y]}^{(+)}$ such that 
$$
[\gamma_{[x],[y]}(0)+{\gamma}'_{[x],[y]}(0)]=[x],~[\gamma_{[x],[y]}(0)-{\gamma}'_{[x],[y]}(0)]=[y],
$$
where $\gamma_{[x],[y]}(s)=\cosh s\cdot p+\sinh s\cdot v,~ s\in(-\infty,+\infty)$ and 
$$
p=\dfrac{x+y}{\sqrt{-2\langle x,y\rangle_{n,1}}},~v=\dfrac{x-y}{\sqrt{-2\langle x,y\rangle_{n,1}}}.
$$
Moreover, $\gamma_{[x],[y]}$ is called a geodesic passing through the ideal points $[x],[y]$.  

By the above argument, the map from $\partial_{\infty}\mathbb{H}^n$ to $\partial I^n_+$ is a natural 1-1 correspondence. Then the ideal boundary $\partial I^n_+$ can be identified with the boundary at infinity $\partial_{\infty}\mathbb{H}^n$ and we can compactify the $n$-dimensional hyperbolic space $\mathbb{H}^n$ by adding the boundary at infinity $\partial_{\infty}\mathbb{H}^n$, namely $\overline{\mathbb{H}^n}:=\mathbb{H}^n\bigcup\partial_{\infty}\mathbb{H}^n$.  

Moreover, adopting a common abuse of notation, we do not distinguish a point $x\in L^+$ and the ideal point $[x]\in\partial_{\infty}\mathbb{H}^n$ represented by $x$.

A vector $v\in\mathbb{R}^{n,1}$ is \emph{spacelike},  \emph{lightlike},  \emph{timelike} if $\langle v,v \rangle_{n,1}$ is positive, zero or negative, respectively. A linear subspace $V\subset \rn$ is \emph{spacelike},  \emph{lightlike},  \emph{timelike} if the restriction of Minkowski inner product $\langle\cdot ,\cdot \rangle_{n,1}$ to $V$ is positive definite, degenerate or indefinite, respectively.

\subsection{Conformal tetrahedra and sphere packings in constant curvature spaces}\label{s1} Given a combinatorial tetrahedron $t$ with vertex set $V_t$, then define a space $\Omega^{\mathbb{H}}_t$ ($\Omega^{\mathbb{E}}_t$, $\Omega^{\mathbb{S}}_t$, resp.) called \emph{hyperbolic $($Euclidean, spherical, resp.$)$ admissible space} on $t$, namely 
\begin{align*}
\Omega^{\mathbb{H}}_t:=\{&r_t=(r_i)_{i\in V_t}\in\mathbb{R}_+^{V_t}:Q^{\mathbb{H}}_3(r_t)>0\}\\
(\Omega^{\mathbb{E}}_t:=\{&r_t=(r_i)_{i\in V_t}\in\mathbb{R}_+^{V_t}:Q^\mathbb{E}_3(r_t)>0\},\\
\Omega^{\mathbb{S}}_t:=\{&r_t=(r_i)_{i\in V_t}\in(0,\pi)^{V_t}:Q^\mathbb{S}_3(r_t)>0, \\
&r_i+r_j\in(0,\pi),\text{for all}~ i\ne j\in V_t\},\text{resp.}).
\end{align*}   

It is a well-known fact (see e.g.\cite{Glickenstein05}, \cite{Ge20}, \cite{Xu20}) that $r_t=(r_i)_{i\in V_t}\in\Omega^{\mathbb{H}}_t(\Omega^{\mathbb{E}}_t,\Omega^{\mathbb{S}}_t, \text{resp}.)$ if and only if there exists a unique four-sphere configuration $\mathcal{P}(r_t)=\{S_i\}_{i\in V_t}$ in 3-dimensional hyperbolic (Euclidean, spherical, resp.) space $\mathbb{H}^3(\mathbb{E}^3,\mathbb{S}^3,\text{resp}.)$ up to isometry, which satisfies
\begin{enumerate}[1.]
    \item  $S_i=S(v_i,r_i)\subset\mathbb{H}^3(\mathbb{E}^3,\mathbb{S}^3,\text{resp}.)$ is a hyperbolic (Euclidean, spherical,
    resp.) sphere with center $v_i$ and radius $r_i$ for all $i\in V_t$;   
    \item The hyperbolic (Euclidean, spherical,
    resp.) spheres $S_i$ and $S_j$ are externally tangent to each other for all $i\ne j\in V_t$;
    \item The centers of the hyperbolic (Euclidean, spherical,
    resp.) spheres $\{S_i\}_{i\in V_t}$  are not all on the same plane.
\end{enumerate}
The four-sphere configuration $\mathcal{P}(r_t)$ is also called the \emph{hyperbolic $($Euclidean, spherical, resp.$)$ sphere packing parameterized by} $r_t\in\Omega^{\mathbb{H}}_t(\Omega^{\mathbb{E}}_t,\Omega^{\mathbb{S}}_t, \text{resp}.)$ on the combinatorial tetrahedron $t$. 

Moreover, the hyperbolic (Euclidean, spherical, resp.) sphere packing $\mathcal{P}(r_t)$ also uniquely (up to isometry) determines a non-degenerate hyperbolic (Euclidean, spherical, resp.) tetrahedron $\Delta^3(r_t)$ with vertex set $\{v_i:i\in V_t\}$, which is obtained by connecting the centers of the spheres in the sphere packing $\mathcal{P}(r_t)$ using geodesics. Note that this hyperbolic (Euclidean, spherical, resp.) tetrahedron $\Delta^3(r_t)$ satisfies
\begin{equation}\label{comgi}
\ell_{ij}=r_i+r_j,~\text{for all}~i\ne j\in V_t,
\end{equation}
where $\ell_{ij}$ is the length of the geodesic edge $e_{ij}$ connecting $v_i$ and $v_j$ of the tetrahedron $\Delta^3(r_t)$. We call such a non-degenerate hyperbolic (Euclidean, spherical, resp.) tetrahedron obtained from the corresponding sphere packing is \emph{conformal} and $\Delta^3(r_t)$ is called the \emph{conformal hyperbolic $($Euclidean, spherical, resp.$)$ tetrahedron parameterized by} $r_t\in\Omega^{\mathbb{H}}_t(\Omega^{\mathbb{E}}_t,\Omega^{\mathbb{S}}_t, \text{resp}.)$.  

Moreover, the metric on the conformal hyperbolic (Euclidean, spherical, resp.) tetrahedron $\Delta^3(r_t)$ is called a \emph{hyperbolic $($Euclidean, spherical, resp.$)$ sphere packing metric parameterized by} $r_t\in\Omega^{\mathbb{H}}_t(\Omega^{\mathbb{E}}_t,\Omega^{\mathbb{S}}_t, \text{resp}.)$. Adopting a common abuse of notation, $r_t\in\Omega^{\mathbb{H}}_t(\Omega^{\mathbb{E}}_t,\Omega^{\mathbb{S}}_t, \text{resp}.)$ is also called a hyperbolic (Euclidean, spherical, resp.) sphere packing metric.

We can also give an equivalent definition of the conformal tetrahedra, which is more natural and is used as the standard definition in this paper: A non-degenerate hyperbolic (Euclidean, spherical, resp.) tetrahedron $\Delta^3$ is \emph{conformal} if (\ref{comgi}) holds for all the geodesic edges of $\Delta^3$. Moreover, there exists the following characterization of conformal hyperbolic (Euclidean, spherical, resp.) tetrahedra.

\begin{lemma}[\cite{Glickenstein05}, \cite{Ge20}, \cite{Xu20}]\label{chara}
Given a parameter $r=(r_i)_{1\le i\le 4}\in\mathbb{R}_+^4(\mathbb{R}_+^4,$ $(0,\pi)^4, \text{resp}.)$$(r_i+r_j(i\ne j)\in(0,\pi)$ in spherical case$)$, then the hyperbolic $($Euclidean, spherical, resp.$)$ tetrahedron $\Delta^3\subset\mathbb{H}^3(\mathbb{E}^3, \mathbb{S}^3, \text{resp}.)$ with edge length $\ell_{ij}=r_i+r_j$ for all geodesic edges $e_{ij}$ of $\Delta^3$ is conformal if and only if  
$$
Q^{\mathbb{H}}_3(r)(Q^{\mathbb{E}}_3(r),Q^{\mathbb{S}}_3(r),\text{resp}.)>0.
$$
\end{lemma}

By Lemma \ref{chara}, the space $\Omega^{\mathbb{H}}_t$($\Omega^{\mathbb{E}}_t$, $\Omega^{\mathbb{S}}_t$, resp.) is also called \emph{admissible space of conformal hyperbolic $($Euclidean, spherical, resp.$)$ tetrahedra}.

It is a natural idea to generalize the concepts  ``sphere packing'' and ``conformal'' to the hyperbolic tetrahedra with ideal vertices. To do this, we introduce the \emph{compactification} of the hyperbolic admissible space $\Omega_t^{\mathbb{H}}$ and method of approximation to define the two concepts in this case.

We can construct a partial compactification $\widehat{\Omega}_t^{\mathbb{H}}$ of the space $\Omega_t^{\mathbb{H}}$, which is called the \emph{extended hyperbolic admissible space}, namely 
$$
\widehat{\Omega}_t^{\mathbb{H}}:=\left \{ r_t=(r_i)_{i\in V_t}\in\widehat{\mathbb{R}_{+}^{V_t}}:Q^{\mathbb{H}}_3(r_t)>0 \right\}.
$$

Define a space $\partial_{\infty}\Omega_t^{\mathbb{H}}$ called the \emph{boundary at infinity} of hyperbolic admissible space $\Omega_t^{\mathbb{H}}$, namely
$$
\partial_{\infty}\Omega_t^{\mathbb{H}}:=\{r_t=(r_i)_{i\in V_t}\in\widehat{\R^{V_t}_{+}}\setminus\R^{V_t}_{+}:Q^{\mathbb{H}}_3(r_t)>0\}.
$$
Moreover, we have
$$
\widehat{\Omega}_t^{\mathbb{H}}=\Omega_t^{\mathbb{H}}\bigsqcup\partial_{\infty}\Omega_t^{\mathbb{H}}.
$$

For any $r_t=(r_i)_{i\in V_t}\in\partial_{\infty}\Omega_t^{\mathbb{H}}$, by definition, there exists a proper subset $W_t\subset V_t$ such that $r_{W_t}=(r_i)_{i\in W_t}\in\mathbb{R}_+^{W_t}$ and $r_{V_t\setminus W_t}=(\infty)_{i\in V_t\setminus W_t}$. Moreover, there exists a sequence of hyperbolic sphere packing metrics $\{r_t^{(m)}\}_{m=1}^{\infty}\subset\Omega_t^{\mathbb{H}}$ such that $r^{(m)}_t\to r_t$($m\to\infty$) and the sequence of hyperbolic sphere packings $\{\mathcal{P}(r^{(m)}_t)\}_{m=1}^{\infty}\subset\mathbb{H}^3$ converges to a unique ``four-sphere'' configuration $\mathcal{P}(r_t)=\{S_i\}_{i\in V_t}$ in the Gromov-Hausdorff sense up to isometry, which satisfies
\begin{enumerate}[1.]
    \item  $S_i=S(v_i,r_i)\subset\mathbb{H}^3$ is a hyperbolic sphere with center $v_i$ and radius $r_i$ for all $i\in W_t$; 
    \item $S_i=S(v_i,r_i)\subset\overline{\mathbb{H}^3}$ is a horosphere with center (ideal point) $v_i$ and radius $r_i=\infty$ for all $i\in V_t\setminus W_t$;
    \item The (horo)spheres $\{S_i\}_{i\in V_t}$ are mutually externally tangent;
     \item The centers of the (horo)spheres $\{S_i\}_{i\in V_t}$ are not all on the same plane.
\end{enumerate}  
The four-sphere configuration $\mathcal{P}(r_t)$ is also called the \emph{ideal hyperbolic sphere packing parameterized by} $r_t\in\partial_{\infty}\Omega_t^{\mathbb{H}} $ on the combinatorial tetrahedron $t$.

An \emph{ideal hyperbolic tetrahedron} is a hyperbolic tetrahedron with at least one ideal vertex, then the ideal hyperbolic sphere packing $\mathcal{P}(r_t)$ uniquely (up to isometry) determines a non-degenerate ideal hyperbolic tetrahedron $\Delta^3(r_t)$ with vertex set $\{v_i:i\in V_t\}$, which is obtained by connecting the centers of the (horo)spheres in the ideal hyperbolic sphere packing $\mathcal{P}(r_t)$ using geodesics. Note that this ideal hyperbolic tetrahedron $\Delta^3(r_t)$ has ideal vertices $\{v_i:i\in V_t\setminus W_t\}$ and satisfies
\begin{equation}\label{comp}
\ell_{ij}=r_i+r_j,~\text{for all}~i\ne j\in V_t,
\end{equation}
where $\ell_{ij}$ is the length of the geodesic edge $e_{ij}$ connecting $v_i$ and $v_j$ of this tetrahedron $\Delta^3(r_t)$ . We also call such a non-degenerate ideal hyperbolic tetrahedron obtained from the corresponding ideal hyperbolic sphere packing is \emph{conformal} and $\Delta^3(r_t)$ is called the \emph{conformal ideal hyperbolic tetrahedron parameterized by} $r_t\in\partial_{\infty}\Omega_t^{\mathbb{H}}$.  

The metric on the conformal ideal hyperbolic tetrahedron $\Delta^3(r_t)$ is called an \emph{ideal hyperbolic sphere packing metric parameterized by} $r_t\in\partial_{\infty}\Omega_t^{\mathbb{H}}$. Adopting a common abuse of notation, $r_t\in\partial_{\infty}\Omega_t^{\mathbb{H}}$ is also called an ideal hyperbolic sphere packing metric. 

We can also give an equivalent definition of the conformal ideal hyperbolic tetrahedra, which is more natural and is used as the standard definition in this paper: A non-degenerate ideal hyperbolic tetrahedron $\Delta^3$ is \emph{conformal} if there exists a sequence of conformal hyperbolic tetrahedra $\{\Delta^3_{m}\}_{m=1}^{\infty}\subset\mathbb{H}^n$ such that $\Delta^3_{m}$ converges to $\Delta^3(m\rightarrow\infty)$ in the Gromov-Hausdorff sense. 

Moreover, we have the following result, which is the characterization of conformal ideal hyperbolic tetrahedra.  

\begin{theorem}\label{aa1}
Given a parameter $r=(r_i)_{1\le i\le 4}\in\widehat{\mathbb{R}_+^4}\setminus\mathbb{R}_+^4$, then the ideal hyperbolic tetrahedron $\Delta^3\subset\overline{\mathbb{H}^3}$ with edge length $\ell_{ij}=r_i+r_j$ for all geodesic edges $e_{ij}$ of $\Delta^3$ is conformal if and only if $Q^{\mathbb{H}}_3(r)>0$.
\end{theorem}

\pf ($\Rightarrow$) Since the ideal hyperbolic tetrahedron $\Delta^3$ is conformal, then by definition, there exists a sequence of conformal hyperbolic tetrahedra $\{\Delta^3_{m}\}_{m=1}^{\infty}\subset\mathbb{H}^3$ such that $\Delta^3_{m}\to\Delta^3(m\rightarrow\infty)$, where $\Delta^3_{m}=\Delta^3(r^{(m)})$ and $r^{(m)}\in\Omega_t^{\mathbb{H}}$ for all $m\ge1$. 

For any $m\ge1$, $\Delta^3_{m}$ corresponds to the four-sphere configuration $\mathcal{P}^{(m)}=\mathcal{P}(r^{(m)})$. Since $\Delta^3_{m}\to\Delta^3(m\rightarrow\infty)$, then $\mathcal{P}^{(m)}\to\mathcal{P},~r^{(m)}\to r(m\to\infty)$, where $\mathcal{P}=\mathcal{P}(r)$ is the four-sphere configuration parameterized by the $r$, which corresponds to the conformal ideal hyperbolic tetrahedron $\Delta^3$. 

Since $r^{(m)}\to r(m\to\infty)$ and $r^{(m)}\in\Omega_t^{\mathbb{H}}$ for all $m\ge1$, then we have $Q_{3}^{\mathbb{H}}(r)\ge0$. If $Q_{3}^{\mathbb{H}}(r)=0$, then by higher-dimensional Descartes circle Theorem (see Remark \ref{aa2}), the centers of (horo)spheres of $\mathcal{P}$ are all on the same plane, which leads to a contradiction. Then we obtain $Q_{3}^{\mathbb{H}}(r)>0$. 

($\Leftarrow$) If $Q_{3}^{\mathbb{H}}(r)>0$, then by higher-dimensional Descartes circle Theorem, the ideal hyperbolic tetrahedron $\Delta^3$ corresponds to a four-sphere configuration $\mathcal{P}=\mathcal{P}(r)$, which implies that $\Delta^3$ is non-degenerate.  

For the four-sphere configuration $\mathcal{P}$, there exists a sequence $\{r^{(m)}\}_{m=1}^{\infty}\subset\Omega_t^{\mathbb{H}}$ and a sequence of four-sphere configurations $\{\mathcal{P}^{(m)}\}_{m=1}^{\infty}$ such that $r^{(m)}\to r$, $\mathcal{P}^{(m)}\to\mathcal{P}(m\to\infty)$, where $\mathcal{P}^{(m)}=\mathcal{P}(r^{(m)})$ for all $m\ge1$.

Since $\mathcal{P}^{(m)}\to\mathcal{P}(m\to\infty)$, then we have $\Delta_m^3\to\Delta^3(m\to\infty)$, where $\Delta_m^3=\Delta^3(r^{(m)})$ is the conformal hyperbolic tetrahedron for all $m\ge1$. Then by definition, the ideal hyperbolic tetrahedron $\Delta^3$ is conformal.  
\pfd

By Theorem \ref{aa1}, the space $\partial_{\infty}\Omega_t^{\mathbb{H}}$ is also called \emph{admissible space of conformal ideal hyperbolic tetrahedra} and is abbreviated as \emph{ideal hyperbolic admissible space}.

\section{ n-simplices and dual n-simplices in constant curvature spaces}\label{s3}

In this section, we introduce hyperbolic (Euclidean, spherical, ideal hyperbolic, resp.) $n$-simplices and dual $n$-simplices.

A non-degenerate $n$-dimensional hyperbolic (Euclidean, spherical, ideal hyperbolic, resp.) \emph{simplex} $\Delta^{n}$ ($n$-simplex for abbreviation) is the geodesic convex hull of $n+1$ linearly independent (affinely independent in Euclidean case) vertices $v_1,v_2,\cdots,v_{n+1}$(at least one ideal vertex $v_i\in\partial_{\infty}\mathbb{H}^n$ in ideal hyperbolic case) in $n$-dimensional hyperbolic (Euclidean, spherical, compactified hyperbolic, resp.) space $$\mathbb{H}^n\subset\rn (  \mathbb{E}^n, \mathbb{S}^n\subset\mathbb{E}^{n+1}, \overline{\mathbb{H}^n}\subset\rn, \text{resp.}).$$ 



A hyperbolic (Euclidean, spherical, resp.) $n$-simplex $\Delta^{n}$ is \emph{conformal} if there exist $n+1$ positive numbers $r_1,\cdots,r_{n+1}>0$ such that 
\begin{equation}\label{copa}
\ell_{ij}=r_i+r_j,~\text{for all geodesic edges}~e_{ij}~\text{of}~\Delta^n,
\end{equation}
where $\ell_{ij}$ is the length of the geodesic edge $e_{ij}$ connecting $v_i$ and $v_j$ of the $n$-simplex $\Delta^n$. In hyperbolic (spherical, resp.) case, by (\ref{juli}) ((\ref{julis}), resp.), we have
\begin{equation}\label{c3}
    \langle v_i,v_j \rangle_{n,1}=-\cosh \ell_{ij} (\langle v_i,v_j \rangle=\cos \ell_{ij}, \text{resp}.).
\end{equation}

An ideal hyperbolic $n$-simplex $\Delta^{n}$ is \emph{conformal} if there exists a sequence of conformal hyperbolic $n$-simplices $\{\Delta^{n}_m\}_{m=1}^{\infty}\subset\mathbb{H}^n$ such that $\Delta^{n}_m$ converges to $\Delta^{n}$ $(m\rightarrow\infty)$ in the Gromov-Hausdorff sense.

\begin{remark}
An ideal hyperbolic $n$-simplex $\Delta^n$ can also be defined as the Gromov-Hausdorff limit of a sequence of hyperbolic $n$-simplices $\{\Delta^n_m\}_{m=1}^{\infty}$ and the edge length $\ell_{ij}:=\lim_{m\to\infty}\ell_{ij}^{(m)}$. Note that $\ell_{ij}=\infty$ is meaningful since $\infty$ is regarded as a point in this paper.
\end{remark}

For any vertex $v_i$ of a hyperbolic (Euclidean, spherical, ideal hyperbolic, resp.) $n$-simplex $\Delta^n$, we can define the $(n-1)$-dimensional opposite face $F_i$ (also called \emph{i-face}) of $v_i$, namely the geodesic convex hull of $n$ vertices $v_1,\cdots,\hat{v}_i,\cdots,v_{n+1}$, where $\hat{v}_i$ means removing the vertex $v_i$ from $\{v_1,\cdots,v_{n+1}\}$.

For any edge $e_{ij}$($i\neq j$) of a $n$-simplex $\Delta^n$, we can define the $(n-2)$-dimensional opposite face (also called \emph{ij-face}) $F_{ij}$ of $e_{ij}$, namely
$$
F_{ij}:=F_i\cap F_j.
$$ 
Moreover, we denote the dihedral angle at the $ij$-face $F_{ij}$ by $\theta_{ij}$ and define $\theta_{ii}:=\pi$ for all $1\le i\le n+1$.

We state and prove the following well-known properties for the sake of completeness and self-containedness.

\begin{prop}\label{eu1}
Given a Euclidean $n$-simplex $\Delta^n\subset\mathbb{E}^{n}$ with $n+1$ vertices $v_1,v_2,\cdots,v_{n+1}$, then there exists the unique $n+1$ affinely independent dual vertices $v_1^*,v_2^*,\cdots,v_{n+1}^*\in\mathbb{S}^{n-1}$ such that
\begin{equation}\label{eud}
		 \begin{cases}
         \langle v_i^*,v_i^*\rangle=1,&1\le i\le n+1,\\
          \langle\textbf{v}_i,v_i^*\rangle<0,& 1\le i\le n,\\
         \langle \textbf{v}_i,v_j^*\rangle =0, &  1\le i\ne j\le n.
         \end{cases}
\end{equation}
where $\textbf{v}_i=v_{i}-v_{n+1}$ for all $1\le i\le n$.
\end{prop}

\begin{proof}
Choose the unit outward normal vector of the $i$-face $F_{i}$ of the Euclidean $n$-simplex $\Delta^n$ and denote this vector by $v_i^*$ for all $1\le i\le n+1$, then we have
$$
\langle v_i^*,v_i^*\rangle=1,~1\le i\le n+1;~\langle\textbf{v}_i,v_j^*\rangle=0,~1\le i\ne j\le n.
$$
Moreover, we obtain
$$
\langle \textbf{v}_i,v_i^*\rangle=-d_0(v_i,F_i)<0,~1\le i\le n,
$$
where $d_0(v_i,F_i)(1\le i\le n)$ is the distance between the vertex $v_i$ and the $i$-face $F_i$.

Since $\langle v_i^*,v_i^*\rangle=1$ for all $1\le i\le n+1$, then $v_1^*,v_2^*,\cdots,v_{n+1}^*\in\mathbb{S}^{n-1}$. Moreover, $v_1^*,\cdots,v_{n+1}^*$ are uniquely determined by the Euclidean $n$-simplex $\Delta^n$ and $v_1^*,\cdots,\hat{v}_i^*,\cdots,v_{n+1}^*$ are linearly independent for all $1\le i\le n+1$. 
\end{proof}

Similarly, we have the following proposition, which gives the dual property of the hyperbolic (spherical, ideal hyperbolic, resp.) $n$-simplex.  

\begin{prop}\label{l1}
Given a hyperbolic $($spherical, ideal hyperbolic, resp.$)$ $n$-simplex $\Delta^n\subset\mathbb{H}^{n}(\mathbb{S}^{n},\overline{\mathbb{H}^{n}},\text{resp.})$ with $n+1$ vertices $v_1,v_2,\cdots,v_{n+1}$, then there exists the unique $n+1$ linearly independent dual vertices $v_1^*,v_2^*,\cdots,v_{n+1}^*\in\ds(\mathbb{S}^n,\ds,\text{resp}.)$ such that
\begin{equation}\label{17}
		 \begin{cases}
         \langle v_i^*,v_i^*\rangle_{(n,1)}=1,&1\le i\le n+1,\\
         \langle v_i,v_i^*\rangle_{(n,1)} <0, &  1\le i\le n+1,\\
         \langle v_i,v_j^*\rangle_{(n,1)} =0, &  1\le i\ne j\le n+1.
         \end{cases}
\end{equation}
\end{prop}

\begin{proof}
(Hyperbolic case.) Since the Minkowski inner product $\langle\cdot ,\cdot \rangle_{n,1}$ on $\rn$ is non-degenerate, then for any $1\le i\le n+1$, we have
\begin{equation}\label{bu}
\rn=H_i\oplus H_i^\perp,
\end{equation}
where $H_i$ is the supporting hyperplane of the $i$-face $F_i$, namely $H_i=\Span\{v_1,$ $\cdots,\hat{v}_i,\cdots,v_{n+1}\}$, and $H_i^\perp$ is the orthogonal complement space. Since $\dim H_i=n$, then by (\ref{bu}), we have $\dim H_i^\perp=1$, then choose a basis $0\ne u_i\in H_i^\perp$.  

If $\langle u_i,u_i\rangle_{n,1}<0$, then we have $\langle\cdot ,\cdot \rangle_{n,1}|_{H_i}$ is positive definite. For any $j\ne i$, we have $v_j\in H_i$ and $\langle v_j,v_j\rangle_{n,1}=-1<0$, which leads to a contradiction. 

If $\langle u_i,u_i\rangle_{n,1}=0$, suppose that $u_i=(u_i^1,\cdots,u_i^{n+1})$ and choose a timelike vector $w_i=(w_i^1,\cdots,w_i^{n+1})\in H_i$. By (\ref{bu}), we have
\begin{equation}\label{bu1}
		 \begin{cases}
        \langle w_i,w_i\rangle_{n,1}<0,\\
        \langle w_i,u_i\rangle_{n,1}=0,\\
        \langle u_i,u_i\rangle_{n,1}=0,
         \end{cases}
	\end{equation}
which implies 
\begin{equation}\label{bu2}
		 \begin{cases}
        \sum_{j=1}^n(w_i^j)^2<(w_i^{n+1})^2,\\
        \sum_{j=1}^nw_i^ju_i^j=w_i^{n+1}u_i^{n+1},\\
        \sum_{j=1}^n(u_i^j)^2=(u_i^{n+1})^2.
         \end{cases}
	\end{equation}
By Cauchy-Schwarz inequality and (\ref{bu2}), we have
$$
(w_i^{n+1}u_i^{n+1})^2=\left(\sum_{j=1}^n w_i^ju_i^j\right)^2\le \sum_{j=1}^n (w_i^j)^2\sum_{j=1}^n (u_i^j)^2<(w_i^{n+1}u_i^{n+1})^2,
$$
which leads to a contradiction. Then we have $\langle u_i,u_i\rangle_{n,1}>0$.

Since $v_1,v_2,\cdots,v_{n+1}$ are linearly independent, then there exist $\widetilde{w}_i\in H_i$ and $0\ne m_i\in\mathbb{R}$, such that $v_i=\widetilde{w}_i+m_iu_i$. If $m_i>0$, we choose the vertex
\begin{equation}\label{bu3}
v_i^*:=-\frac{u_i}{\sqrt{\langle u_i,u_i\rangle_{n,1}}}.
\end{equation}
If $m_i<0$, we choose the vertex
\begin{equation}\label{bu4}
v_i^*:=\frac{u_i}{\sqrt{\langle u_i,u_i\rangle_{n,1}}}.
\end{equation}
Then we obtain the $n+1$ dual vertices $v_1^*,\cdots,v_{n+1}^*\in\ds$. Moreover, by (\ref{bu}), (\ref{bu3}) and (\ref{bu4}), these dual vertices satisfies $(\ref{17})$.  

Suppose that $\sum_{i=1}^{n+1}k_i v_i^*=0$, then for any $1\le j\le n+1$, by (\ref{17}), we have $\langle \sum_{i=1}^{n+1}k_i v_i^*,v_j\rangle_{n,1}=k_j\langle v_j^*,v_j\rangle_{n,1}=0$. Since $\langle v_j^*,v_j\rangle_{n,1}<0$, we obtain $k_j=0$, which implies that $v_1^*,\cdots,v_{n+1}^*$ are linearly independent.

If $v_1^*,\cdots,v_{n+1}^*$ are not unique, then there exists $\widetilde{v}_i^*\ne v_i^*\in H^{\perp}_i$ and satisfies (\ref{17}), suppose $\widetilde{v}_i^*=ku_i$, then we have $\langle \widetilde{v}_i^*,\widetilde{v}_i^*\rangle_{n,1}=k^2\langle u_i,u_i\rangle_{n,1}=1$ and $k=\pm1/\sqrt{\langle u_i,u_i\rangle_{n,1}}$, which implies $\widetilde{v}_i^*=-v_i^*$ and we have $\langle v_i,\widetilde{v}_i^*\rangle_{n,1}=-\langle v_i,v_i^*\rangle_{n,1}>0$, which leads to a contradiction.

For any $1\le i\le n+1$, define the subspace $H_i^*:=\Span\{v_1^*,\cdots,\hat{v}_i^*,\cdots,v_{n+1}^*\}$, by (\ref{17}), we have $H_i^*\subset \Span\{v_i\}^{\perp}$. Since $v_1^*,\cdots,v_{n+1}^*$ are linearly independent, then $\dim H_i^*=n$ and $H_i^*= \Span\{v_i\}^{\perp}$. Moreover, we have
\begin{equation}\label{bp5}
\rn=\Span\{v_i\}\oplus \Span\{v_i\}^{\perp}=\Span\{v_i\}\oplus H_i^*.
\end{equation}
Since $\langle v_i,v_i\rangle_{n,1}=-1<0$, then the Minkowski inner product $\langle\cdot ,\cdot \rangle_{n,1}$ restricted to $H_i^*$ has signature $(n,0,0)$, which implies $\langle\cdot ,\cdot \rangle_{n,1}|_{H_i^*}$ is positive definite and $H_i^*$ is a spacelike hyperplane. 



(Spherical case.) The proof has been omitted since it is quite similar to the proof of hyperbolic case.


(Ideal hyperbolic case.) The proof has been omitted since it is quite similar to the proof of hyperbolic case. 


Note that for any ideal vertex $v_i$ of $\Delta^n$, the Minkowski inner product $\langle\cdot ,\cdot \rangle_{n,1}$ restricted to $H_i^*=\Span\{v_i\}^{\perp}$ has signature $(n-1,1,0)$, which implies that $H_i^*$ is a lightlike hyperplane. 
\end{proof}

Given a hyperbolic (Euclidean, spherical, ideal hyperbolic, resp.) $n$-simplex $\Delta^n\subset\mathbb{H}^n(\mathbb{E}^n,\mathbb{S}^n,\overline{\mathbb{H}^n},\text{resp}.)$, by Proposition \ref{eu1} and Proposition \ref{l1}, the dual vertices of $\Delta^n$ determine a domain in $\mathrm{d}\mathbb{S}^n(\mathbb{S}^{n-1}, \mathbb{S}^n, \mathrm{d}\mathbb{S}^n,\text{resp}.)$, namely 
$$
\sigma\cap \mathrm{d}\mathbb{S}^n(\sigma\cap \mathbb{S}^{n-1},\sigma\cap \mathbb{S}^n,\sigma\cap \mathrm{d}\mathbb{S}^n,\text{resp}.),
$$
where $\sigma$ is the convex cone spanned by the dual vertices in $\R^{n,1}(\mathbb{E}^n,\mathbb{E}^{n+1}$, $\R^{n,1}, \text{resp.})$.  

Abusing of the terminology, the above domain is called the \emph{dual $n$-simplex} of the hyperbolic (Euclidean, spherical, ideal hyperbolic, resp.) $n$-simplex, denoted by $(\Delta^n)^*$. Note that the dual $n$-simplex $(\Delta^n)^*$ is not necessarily a geometric $n$-simplex. In particular, $(\Delta^n)^*=\mathbb{S}^{n-1}$ if $\Delta^n\subset\mathbb{E}^n$ and $(\Delta^n)^*$ is a spherical $n$-simplex if $\Delta^n\subset\mathbb{S}^n$.

Similar to $n$-simplex, we denote the corresponding geodesic edge, geodesic edge length, $i$-face, $ij$-face, dihedral angle and supporting hyperplane of the dual $n$-simplex $(\Delta^n)^*$ by $e_{ij}^*$, $\ell_{ij}^*$, $F_i^*$, $F_{ij}^*$, $\theta_{ij}^*$($\theta_{ii}^*:=\pi$ for all $1\le i\le n+1$) and $H_i^*$, respectively. 

By the proof of Proposition \ref{l1}, the supporting hyperplanes of all faces of the dual $n$-simplex $(\Delta^n)^*$ are spacelike if $\Delta^n\subset\mathbb{H}^n$. Moreover, the following result implies that the dual $n$-simplex of a (ideal) hyperbolic $n$-simplex is not geodesic convex and not simply connected.

\tm\label{dualhyperbolicsimplex}
Given a hyperbolic $n$-simplex $\Delta^n\subset\mathbb{H}^{n}$, then $\partial(\Delta^n)^\ast\subset\ds$ is represented as a generator of the homology group $\mathrm{H}_{n-1}(\ds)$, where $\partial(\Delta^n)^\ast$ is the union of all $(n-1)$-dimensional faces of $(\Delta^n)^\ast$.
\tmd

\begin{proof}
It is well known that the hyperbolic $n$-simplex $\Delta^n\subset\mathbb{H}^n$ is bounded by the $n+1$ supporting hyperplanes $H_i$ of the $i$-face $F_i$ of $\Delta^n$. In other words, $\Delta^n$ is the intersection of $n+1$ hyperbolic half-spaces, namely
$$
\Delta^n=\bigcap_{i=1}^{n+1}H_i^+,
$$
where $H_i^+$ is the half-space bounded by the supporting hyperplane $H_i$ of the $i$-face $F_i$ of $\Delta^n$ for all $1\le i\le n+1$. Note that the vertex $v_i$ of the $n$-simplex $\Delta^n$ is the intersection of $n$ supporting hyperplanes, namely
$$
v_i=\bigcap_{j\ne i}H_j,~1\le i\le n+1.
$$

We have the following operation steps: Rotate the hyperplane $H_2$ around the $(n-2)$-dimensional totally geodesic subspace $H_1\bigcap H_2$ to another hyperplane $H_2^{(1)}$, such that the dihedral angle of the $(n-2)$-dimensional face on the supporting subspace $H_1\bigcap H_2$ of $\Delta^n$ keeps increasing during the rotation process and the hyperplanes $H^{(1)}_2, H_3,\cdots,H_{n+1}$ intersect at a \emph{hyperideal point} in the Klein projective model $\mathbb{K}^n=\{x=(x_1,\cdots,x_{n+1})\in\mathbb{R}^{n+1}:\sum_{i=1}^{n+1}x_i^2<1\}$, where a hyperideal point $p$ is a point outside the closure of the Klein model $\mathbb{K}^n$, namely $p\in\mathbb{R}^{n+1}\setminus\overline{\mathbb{K}^n}$. We denote the simplex bounded by the hyperplanes $H_1,H_2^{(1)},H_3,\cdots,H_{n+1}$ by $(\Delta^n)^{(1)}$ and the unit outward normal vector on the $2$-face of $(\Delta^n)^{(1)}$ by $v_2^{*(1)}$, note that $v_2^{*(1)}\in e_{12}^*\subset F_{n+1}^*$. 

For the simplex $(\Delta^n)^{(1)}$, rotate the hyperplane $H_3$ around the $(n-2)$-dimensional totally geodesic subspace $H_2^{(1)}\bigcap H_3$ to another hyperplane $H_3^{(1)}$, such that the dihedral angle of the $(n-2)$-dimensional face on the supporting subspace $H_2^{(1)}\bigcap H_3$ of $(\Delta^n)^{(1)}$ keeps increasing during the rotation process and the hyperplanes $H_1, H_3^{(1)},H_4,\cdots,$ $H_{n+1}$ intersect at a hyperideal point. We denote the simplex bounded by the hyperplanes $H_1,H_2^{(1)},H_3^{(1)},H_4,\cdots,H_{n+1}$ by $(\Delta^n)^{(2)}$ and the unit outward normal vector on the $3$-face of $(\Delta^n)^{(2)}$ by $v_3^{*(1)}$, note that $v_3^{*(1)}\in e_{23}^{*(1)}\subset F_{n+1}^*$, where $e_{23}^{*(1)}$ is the geodesic segment connecting $v_2^{*(1)}$ and $v_3^*$.

$\cdots$ 

For the simplex $(\Delta^n)^{(n-2)}$, rotate the hyperplane $H_{n}$ around the $(n-2)$-dimensional totally geodesic subspace $H_{n-1}^{(1)}\bigcap H_n$ to another hyperplane $H_n^{(1)}$, such that the dihedral angle of the $(n-2)$-dimensional face on the supporting subspace $H_{n-1}^{(1)}\bigcap H_n$ of $(\Delta^n)^{(n-2)}$ keeps increasing during the rotation process, the hyperplanes $H_1, H_2,\cdots,H_{n-2},H_n^{(1)},H_{n+1}$ intersect at a hyperideal point and the hyperplanes $H_n^{(1)}, H_{n+1}$ are ultra-parallel. We denote the simplex bounded by the hyperplanes $H_1,H_2^{(1)},\cdots,H_n^{(1)},H_{n+1}$ by $(\Delta^n)^{(n-1)}$ and the unit outward normal vector on the $n$-face of $(\Delta^n)^{(n-1)}$ by $v_n^{*(1)}$, note that $v_n^{*(1)}\in e_{(n-1)n}^{*(1)}\subset F_{n+1}^*$, where $e_{(n-1)n}^{*(1)}$ is the geodesic segment connecting $v_{n-1}^{*(1)}$ and $v_n^*$.

Since the hyperplanes $H_n^{(1)}, H_{n+1}$ are ultra-parallel, we have 
\begin{equation}\label{fkl}
\langle v_n^{*(1)}, v_{n+1}^*\rangle_{n,1}=-\cosh\ell<-1,
\end{equation}
where $\ell>0$ is the hyperbolic distance between $H_n^{(1)}$ and $H_{n+1}$. Since $v_n^{*(1)}, v_{n+1}^*\in\ds$, by (\ref{fkl}), there is no geodesic connecting $v_n^{*(1)}\in F_{n+1}^*$ and $v_n^*$, which implies that the dual $n$-simplex $(\Delta^n)^*$ is not contractible in $\ds$. 

Moreover, since $\ds$ is homeomorphic to $\mathbb{S}^{n-1}\times\mathbb{R}$ and $\partial(\Delta^n)^*$ is homeomorphic to $\mathbb{S}^{n-1}$ with codimension 1 in $\ds$, then $\partial(\Delta^n)^*$ is represented as a generator of the homology group $\mathrm{H}_{n-1}(\ds)$.    
\end{proof}

Similar to the definition (\ref{copa}), a dual $n$-simplex $(\Delta^n)^*$ is \emph{conformal} if there exist $n+1$ positive numbers $r_1^*,\cdots,r_{n+1}^*>0$ $(r_1^*,\cdots,r_{n+1}^*\in\mathbb{R}$ if $\Delta^n$ is an ideal hyperbolic $n$-simplex$)$ such that 
\begin{equation}
\ell_{ij}^*=r_i^*+r_j^*,~\text{for all edge}~e_{ij}^*~\text{of}~(\Delta^n)^*,
\end{equation}
where $\ell_{ij}^*$ is the length of the geodesic edge $e_{ij}^*$ connecting $v_i^*$ and $v_j^*$ of the dual $n$-simplex $(\Delta^n)^*$.

The following result shows a relation between $n$-simplex and dual $n$-simplex. 

\begin{corollary}\label{eio}
 Given a hyperbolic $($Euclidean, spherical, ideal hyperbolic, resp.$)$ $n$-simplex $\Delta^n\subset\mathbb{H}^{n}(\mathbb{E}^{n}, \mathbb{S}^{n}, \overline{\mathbb{H}^n},\text{resp.})$ and dual $n$-simplex $(\Delta^n)^*$, then we have
  \begin{equation}
 \langle v_i,v_j\rangle_{(n,1)}=\begin{cases}
  -\cosh(\pi-\theta_{ij}^*),& \text{ if } \Delta^n\subset\mathbb{H}^{n}, \\
\cos(\pi-\theta_{ij}^*)&\text{ if } \Delta^n\subset\mathbb{S}^{n},\\
 \end{cases}
  \end{equation}
and $\langle v_i^*,v_j^*\rangle_{(n,1)}=\cos(\pi-\theta_{ij})$ for all $1\le i,j\le n+1$.
\end{corollary}

\begin{proof}
(Hyperbolic case.) For any $1\le i\le n+1$, by (\ref{17}), we have
$$
\langle v_i,v_i\rangle_{n,1}=-1=-\cosh(\pi-\theta_{ii}^*),~\langle v_i^*,v_i^*\rangle_{n,1}=1=\cos(\pi-\theta_{ii}).
$$

For any $1\le i\le n+1$, by (\ref{17}), $v_i$ is the outward normal vector of the $i$-face $F_{i}^*$ of the dual $n$-simplex $(\Delta^n)^*$ and $v_i^*$ is the outward normal vector of the $i$-face $F_{i}$ of the $n$-simplex $\Delta^n$. Then for any $1\le i\ne j\le n+1$, we have
$$
\langle v_i,v_j\rangle_{n,1}=-\cosh(\pi-\theta_{ij}^*),~\langle v_i^*,v_j^*\rangle_{n,1}=\cos(\pi-\theta_{ij}).
$$

(Euclidean and spherical case.) The proof has been omitted since it is quite similar to the proof in hyperbolic case.

(Ideal hyperbolic case.)  Since $\Delta^n$ is an ideal hyperbolic $n$-simplex, then there exists a sequence of hyperbolic $n$-simplices $\{\Delta^n_m\}_{m=1}^{\infty}\subset\mathbb{H}^n$ such that $\Delta^n_m\to \Delta^n(m\rightarrow\infty)$, which is equivalent to $(\Delta^n_m)^*\to (\Delta^n)^*(m\rightarrow\infty)$, where $(\Delta^n_m)^*$ is the dual $n$-simplex of $\Delta_m^n$. By the proof of hyperbolic case, for any $1\le i\ne j\le n+1$, we have
$$
\langle v_i^{(m)*},v_j^{(m)*}\rangle_{n,1}=\cos(\pi-\theta_{ij}^{(m)})
$$
and we obtain $\langle v_i^{*},v_j^{*}\rangle_{n,1}=\cos(\pi-\theta_{ij})$ for all $1\le i\ne j\le n+1$ as $m\to\infty$.
\end{proof}

The following result shows an important connection between $n$-simplex and dual $n$-simplex.

\begin{corollary}\label{conedge}
Given a hyperbolic $($Euclidean, spherical, ideal hyperbolic, resp.$)$ $n$-simplex $\Delta^n\subset\mathbb{H}^{n}(\mathbb{E}^{n}, \mathbb{S}^{n}, \overline{\mathbb{H}^n},\text{resp.})$ and dual $n$-simplex $(\Delta^n)^*$, then we have
 \begin{equation}\label{chang}
 \begin{cases}
  \ell_{ij}=\pi-\theta_{ij}^*,& \text{ if } \Delta^n\subset\mathbb{H}^{n},\mathbb{S}^{n}, \\
\ell_{ij}^*=\pi-\theta_{ij},&\text{ if } \Delta^n\subset\mathbb{H}^{n},\mathbb{E}^n,\mathbb{S}^{n},\overline{\mathbb{H}^n},\\
 \end{cases}
  \end{equation}for any $1\le i\ne j\le n+1$
for all geodesic edges $e_{ij}^{(*)}(i\ne j)$ of $(\Delta^n)^{(*)}$. 
\end{corollary}

\begin{proof}
(Hyperbolic case.) For any $1\le i\ne j\le n+1$, by (\ref{c3}) and Corollary \ref{eio}, we have
\begin{equation}\label{c5}
-\cosh \ell_{ij}=\langle v_i,v_j\rangle_{n,1}=-\cosh(\pi-\theta_{ij}^*),
\end{equation}
which implies that $\ell_{ij}=\pi-\theta_{ij}^*$.

By (\ref{17}), $v_i$ is the timelike normal vector on the supporting hyperplane $H_i^*$ of the $i$-face $F_i^*$ for all $1\le i\le n+1$, then we have $F_i^*\subset H_i^*\cap\ds=\mathbb{S}^{n-1}$, which implies that the induced metric on the $i$-face $F_i^*$ is a spherical metric, then by (\ref{c3}), we have
\begin{equation}\label{c4}
\cos \ell_{ij}^*=\langle v_i^* ,v_j^* \rangle_{n,1}=\cos(\pi-\theta_{ij}),~1\le i\ne j\le n+1, 
\end{equation}
which implies that $\ell_{ij}^*=\pi-\theta_{ij}$.

(Euclidean case.) For any $1\le i\ne j\le n+1$, since $v_i^*, v_j^*\in\mathbb{S}^{n-1}$, by (\ref{c3}) and Corollary \ref{eio}, we have
$$
\cos \ell_{ij}^*=\langle v_i^*,v_j^*\rangle=\cos(\pi-\theta_{ij}),
$$
which implies that $\ell_{ij}^*=\pi-\theta_{ij}$.

(spherical case.) For any $1\le i\ne j\le n+1$, by (\ref{c3}) and Corollary \ref{eio}, we have
$$
\cos \ell_{ij}=\langle v_i,v_j\rangle=\cos(\pi-\theta_{ij}^*),~\cos \ell_{ij}^*=\langle v_i^*,v_j^*\rangle=\cos(\pi-\theta_{ij}),
$$
which implies that $\ell_{ij}=\pi-\theta_{ij}^*$ and $\ell_{ij}^*=\pi-\theta_{ij}$.

(Ideal hyperbolic case.) Since $\Delta^n$ is an ideal hyperbolic $n$-simplex, then there exists a sequence of hyperbolic $n$-simplices $\{\Delta^n_m\}_{m=1}^{\infty}\subset\mathbb{H}^n$ such that $\Delta^n_m\to \Delta^n(m\rightarrow\infty)$, which is equivalent to $(\Delta^n_m)^*\to (\Delta^n)^*(m\rightarrow\infty)$, where $(\Delta^n_m)^*$ is the dual $n$-simplex of $\Delta^n_m$. By the proof of hyperbolic case, for any $1\le i\ne j\le n+1$, we have
$$
\ell_{ij}^{(m)*}=\pi-\theta_{ij}^{(m)}
$$
and we obtain $\ell_{ij}^{*}=\pi-\theta_{ij}$ for all $1\le i\ne j\le n+1$ as $m\to\infty$.
\end{proof}

As a consequence, the ideal hyperbolic tetrahedron (3-simplex) has the following property.

\begin{corollary}\label{po9}
    Given an ideal hyperbolic tetrahedron $\Delta^3\subset\overline{\mathbb{H}^3}$ and dual tetrahedron $(\Delta^3)^*$, then we have
    $$
\ell_i^*=2\pi,~\text{for all ideal vertex}~v_i~\text{of}~\Delta^3,
    $$
where $\ell_i^*$ is the length of $\partial F_i^*$.   
\end{corollary}

\begin{proof}
Suppose that $\Delta^3$ has (ideal) vertices $\{v_1,v_2,v_3,v_4\}$, then for any ideal vertex $v_i$, by Gauss-Bonnet formula, we have
\begin{equation}\label{chang2}
\theta_{jk}+\theta_{jl}+\theta_{kl}=\pi,
\end{equation}
where $(i,j,k,l)$ is the permutation of $(1,2,3,4)$. Then by (\ref{chang}) and (\ref{chang2}), we have 
$$
\ell_{i}^*=\ell_{jk}^*+\ell_{jl}^*+\ell_{kl}^*=\pi-\theta_{jk}+\pi-\theta_{jl}+\pi-\theta_{kl}=2\pi. 
$$
\end{proof}

\begin{remark}
    The duality between the $($ideal$)$ hyperbolic tetrahedra and dual tetrahedra is also called ``hyperbolic-de Sitter duality'' $($see e.g. \cite{HR93}, \cite{Sch03}$)$. Moreover, this duality can also be generalized to the $($ideal$)$ hyperbolic polyhedra and dual polyhedra.      \end{remark}

\section{Gram matrix of n-simplices and dual n-simplices in constant curvature spaces}\label{s4}

In this section, we introduce and study the Gram matrices of hyperbolic (Euclidean, spherical, ideal hyperbolic, resp.) $n$-simplices and dual $n$-simplices. Moreover, we prove Theorem \ref{admissibleidealangle}, which gives the characterization of (ideal) hyperbolic $n$-simplices.

Given a hyperbolic (Euclidean, spherical, ideal hyperbolic, resp.) $n$-simplex $\Delta^n\subset\mathbb{H}^{n}(\mathbb{E}^n, \mathbb{S}^{n},$ $\overline{\mathbb{H}^{n}}, \text{resp.})$ with $n+1$ vertices $v_1,v_2,\cdots,v_{n+1}$, then we can define the \emph{Gram matrix} $G$ of the $n$-simplex $\Delta^n$, namely 
$$
G:=(\langle v_i,v_j\rangle_{n,1})_{ij}((\langle \textbf{v}_i,\textbf{v}_j\rangle)_{1\le i,j\le n},(\langle v_i,v_j\rangle)_{ij}, (\langle v_i,v_j\rangle_{n,1})_{ij}, \text{resp}.).
$$ 
Note that the Gram matrix of an ideal hyperbolic $n$-simplex is not uniquely determined, since the representative of an ideal vertex is not unique.

By Proposition \ref{eu1} and Proposition \ref{l1}, we can also define the Gram matrix $G^*$ of the dual $n$-simplex $(\Delta^n)^*$ with $n+1$ vertices $v_1^*,v_2^*,\cdots,v_{n+1}^*$, namely 
$$G^*:=(\langle v_i^*,v_j^*\rangle_{n,1})_{ij}((\langle v_i^*,v_j^*\rangle)_{1\le i,j\le n},(\langle v_i^*,v_j^*\rangle)_{ij}, (\langle v_i^*,v_j^*\rangle_{n,1})_{ij}, \text{resp}.).
$$ 

Given the Gram matrix $G^{(*)}$ of a $n$-simplex (dual $n$-simplex) $(\Delta^n)^{(*)}$, the $(i,j)$-cofactor of $G^{(*)}$ is defined as
$$G_{ij}^{(*)}:=(-1)^{i+j}|G^{(*)}(i,j)|,$$
where $G^{(*)}(i,j)$ is the submatrix obtained from $G^{(*)}$ by removing the $i$-th row and $j$-th column from the Gram matrix $G^{(*)}$. Moreover, note that the Gram matrix $G^{(*)}$ is a symmetric matrix. 


We have the following result, which implies an important relation between edge lengths and dihedral angles of a hyperbolic (Euclidean, spherical, ideal hyperbolic, resp.) $n$-simplex. The hyperbolic case of this result is well-known (see e.g. \cite{Murakami05}). In this paper, we provide a uniform technique to obtain the corresponding relations on hyperbolic, Euclidean, spherical and ideal hyperbolic $n$-simplices.

\begin{lemma}\label{lengthandangle}
Given a hyperbolic $($Euclidean, spherical, ideal hyperbolic, resp.$)$ $n$-simplex $\Delta^n\subset\mathbb{H}^{n}(\mathbb{E}^n, \mathbb{S}^{n}, \overline{\mathbb{H}^n}, \text{resp.})$, then we have
\begin{equation}\label{m3}
		 \begin{cases}
         \cos\theta_{ij}=\dfrac{G_{ij}}{\sqrt{G_{ii}G_{jj}}}\left(-\dfrac{G_{ij}}{\sqrt{G_{ii}G_{jj}}},-\dfrac{G_{ij}}{\sqrt{G_{ii}G_{jj}}},\dfrac{G_{ij}}{\sqrt{G_{ii}G_{jj}}}, \text{resp}.\right), \\
         \\
         \cosh \ell_{ij}(\cos \beta_{ij},\cos \ell_{ij}, \cosh \ell_{ij}, \text{resp}.)=\dfrac{G^\ast_{ij}}{\sqrt{G^\ast_{ii}G^\ast_{jj}}},
         \end{cases}
  \end{equation}
for all geodesic edges $e_{ij}$ $($all geodesic edges $e_{ij}$ not incident to $v_{n+1}$ in Euclidean case$)$ of $\Delta^n$, where $\beta_{ij}$ is the interior angle $\angle v_iv_{n+1}v_j$. 

As a consequence, the dihedral angles $\theta_{ij}$ of the hyperbolic $($Euclidean, spherical, ideal hyperbolic, resp.$)$ $n$-simplex $\Delta^n$ are real analytic functions of edge lengths $\ell_{ij}$ of $\Delta^n$ and vice versa $($up to scaling in Euclidean case$)$.
\end{lemma}

\pf
(Hyperbolic case.) Define the matrices $D:=(\langle v_i,v_j^*\rangle_{n,1})_{1\le i,j\le n+1}$ and $M^{(*)}:=(v_1^{(*)},v_2^{(*)},\cdots,v_{n+1}^{(*)})$, where the vertices $v_1^{(*)},v_2^{(*)},$ $\cdots,v_{n+1}^{(*)}\in\mathbb{R}^{n+1}$ are considered as the column vectors. By (\ref{17}), $D$ is a diagonal matrix, namely $D=\text{diag}\{\langle v_1,v_1^*\rangle_{n,1},\cdots,\langle v_{n+1},v_{n+1}^*\rangle_{n,1}\}$.

By the proof of Proposition \ref{l1}, the $n+1$ vertices $v_1^{(*)},v_2^{(*)},\cdots,v_{n+1}^{(*)}$ are linearly independent, which implies that the matrix $M^{(*)}$ is full-rank and $|M^{(*)}|\ne 0$. Moreover, we have
\begin{equation}\label{g1}
\begin{aligned}
\langle M^{(*)},M^{(*)}\rangle_{n,1}&=M^{(*)T}I_{n,1}M^{(*)}=(v_i^{(*)T}I_{n,1}v_j^{(*)})_{1\le i,j\le n+1}\\
&=(\langle v_i^{(*)},v_j^{(*)} \rangle_{n,1})_{1\le i,j\le n+1}\\
&=G^{(*)}.
\end{aligned}
\end{equation}
Similarly, we obtain 
\begin{equation}\label{g2}
\langle M,M^*\rangle_{n,1}=M^{T}I_{n,1}M^*=(\langle v_i,v_j^* \rangle_{n,1})_{1\le i,j\le n+1}=D.
\end{equation}

By (\ref{g1}), we have 
\begin{equation}\label{g30}
|G^{(*)}|=|M^{(*)T}I_{n,1}M^{(*)}|=-|M^{(*)}|^2<0,
\end{equation}
\begin{equation}\label{g5}
G^{(*)-1}=M^{(*)-1}I_{n,1}(M^{(*)T})^{-1}.
\end{equation}
Since $G^{(*)}$ is a symmetric matrix, then the adjoint matrix $G^{(*)ad}$ is symmetric. Morever, we have
\begin{equation}\label{g7}
G^{(*)ad}=(G_{ij}^{(*)})_{1\le i,j\le n+1},~G^{(*)-1}=\frac{G^{(*)ad}}{|G^{(*)}|}.
\end{equation}

By (\ref{g2}), we have 
\begin{equation}\label{g4}
M=I_{n,1}(M^{*-1})^TD,~M^*=I_{n,1}(M^{-1})^TD.
\end{equation}

By (\ref{g1}), (\ref{g5}), (\ref{g7}) and (\ref{g4}), we have
\begin{equation}\label{g4.8}
G=DG^{*-1}D=D\frac{G^{*ad}}{|G^{*}|}D,~G^*=DG^{-1}D=D\frac{G^{ad}}{|G|}D.
\end{equation}

Since $D$ is a  diagonal matrix, then by (\ref{g7}) and (\ref{g4.8}), we have
\begin{equation}\label{g8}
\langle v_i,v_j\rangle_{n,1}=\langle v_i,v_i^*\rangle_{n,1}\frac{G_{ij}^*}{|G^*|}\langle v_j,v_j^*\rangle_{n,1},~1\le i,j\le n+1,
\end{equation}
\begin{equation}\label{g10}
\langle v_i^*,v_j^*\rangle_{n,1}=\langle v_i,v_i^*\rangle_{n,1}\frac{G_{ij}}{|G|}\langle v_j,v_j^*\rangle_{n,1},~1\le i,j\le n+1
.\end{equation}
Then for any $1\le i\le n+1$, by (\ref{17}), (\ref{g8}) and (\ref{g10}), we have
\begin{equation}\label{g9}
-1=\langle v_i,v_i\rangle_{n,1}=\langle v_i,v_i^*\rangle_{n,1}^2\frac{G_{ii}^*}{|G^*|},~1=\langle v_i^*,v_i^*\rangle_{n,1}=\langle v_i,v_i^*\rangle_{n,1}^2\frac{G_{ii}}{|G|}.
\end{equation}
By (\ref{17}) and (\ref{g9}), we obtain
\begin{equation}\label{g11}
\langle v_i,v_i^*\rangle_{n,1}=-\sqrt{\frac{|G|}{G_{ii}}}=-\sqrt{\frac{-|G^*|}{G_{ii}^*}},~1\le i\le n+1.
\end{equation}
Then for any $1\le i\ne j\le n+1$, by (\ref{c3}), (\ref{g30}), (\ref{g8}), (\ref{g10}) and (\ref{g11}), we have
$$
-\cosh \ell_{ij}=\langle v_i,v_j\rangle_{n,1}=-\sqrt{\frac{-|G^*|}{G_{ii}^*}}\cdot\frac{G_{ij}^*}{|G^*|}\cdot-\sqrt{\frac{-|G^*|}{G_{jj}^*}}=-\frac{G_{ij}^*}{\sqrt{G_{ii}^*G_{jj}^*}},
$$
$$
\cos (\pi-\theta_{ij})=\langle v_i^*,v_j^*\rangle_{n,1}=-\sqrt{\frac{-|G|}{-G_{ii}}}\cdot\frac{G_{ij}}{|G|}\cdot-\sqrt{\frac{-|G|}{-G_{jj}}}=-\frac{G_{ij}}{\sqrt{G_{ii}G_{jj}}},
$$
which implies (\ref{m3}), and the dihedral angles $\theta_{ij}$ of the hyperbolic $n$-simplex $\Delta^n$ are real analytic functions of edge lengths $\ell_{ij}$ and vice versa.

(Euclidean case.) Define the matrices $D:=(\langle \textbf{v}_i,v_j^*\rangle)_{1\le i,j\le n}$, $M=(\textbf{v}_1,\textbf{v}_2$, $\cdots,\textbf{v}_n)$ and $M^*=(v_1^*,v_2^*,\cdots,v_{n}^*)$, where $\textbf{v}_i,v_i^*\in\mathbb{R}^n$ are considered as the column vectors for all $1\le i\le n$.  

By (\ref{eud}), $D$ is a diagonal matrix, namely $D=\text{diag}\{\langle \mathbf{v}_1,v_1^*\rangle,\cdots,\langle\textbf{v}_n,v_n^*\rangle\}$. Moreover, $\mathbf{v}_1,\mathbf{v}_2,\cdots,\mathbf{v}_n$, $v_1^*,v_2^*,\cdots,v_{n}^*$ are linearly independent, respectively, which implies that the matrix $M^{(*)}$ is full-rank and $|M^{(*)}|\ne 0$.  

Similar to the proof in the hyperbolic case, we have
\begin{equation}\label{g14}
\langle M^{(*)},M^{(*)}\rangle=M^{(*)T}M^{(*)}=G^{(*)},~\langle M,M^*\rangle=M^TM^*=D, 
\end{equation}
\begin{equation}\label{g15}
|G^{(*)}|=|M^{(*)T}M^{(*)}|=|M^{(*)}|^2>0, 
\end{equation}
\begin{equation}\label{g16}
\langle \textbf{v}_i,\textbf{v}_j\rangle=\langle \textbf{v}_i,v_i^*\rangle\frac{G_{ij}^*}{|G^*|}\langle \textbf{v}_j,v_j^*\rangle,~1\le i,j\le n,
\end{equation}
\begin{equation}\label{g17}
\langle v_i^*,v_j^*\rangle=\langle \textbf{v}_i,v_i^*\rangle\frac{G_{ij}}{|G|}\langle \textbf{v}_j,v_j^*\rangle,~1\le i,j\le n
.\end{equation}
Then for any $1\le i\le n$, by (\ref{g16}) and (\ref{g17}), we have
\begin{equation}\label{g18}
\ell_{in+1}^2=\langle \textbf{v}_i,\textbf{v}_i\rangle=\langle \textbf{v}_i,v_i^*\rangle^2\frac{G_{ii}^*}{|G^*|},
\end{equation}
\begin{equation}\label{g19}
1=\langle v_i^*,v_i^*\rangle=\langle \textbf{v}_i,v_i^*\rangle^2\frac{G_{ii}}{|G|}.
\end{equation}

By (\ref{eud}), (\ref{g18}) and (\ref{g19}), we have
\begin{equation}\label{g20}
\langle \textbf{v}_i,v_i^*\rangle=-\sqrt{\frac{|G|}{G_{ii}}}=-\ell_{in+1}\sqrt{\frac{|G^*|}{G_{ii}^*}},~1\le i\le n.
\end{equation}

Since $v_i^*(1\le i\le n)$ is the unit outward normal vector of the $i$-face $F_{i}$ of the Euclidean $n$-simplex $\Delta^n$, then by (\ref{g16}), (\ref{g17}) and (\ref{g20}), we obtain
\begin{align*}
\ell_{in+1}\ell_{jn+1}\cos\beta_{ij}=\langle \textbf{v}_i,\textbf{v}_j\rangle&=-\ell_{in+1}\sqrt{\frac{|G^*|}{G_{ii}^*}}\cdot\frac{G_{ij}^*}{|G^*|}\cdot-\ell_{jn+1}\sqrt{\frac{|G^*|}{G_{jj}^*}}\\&=\ell_{in+1}\ell_{jn+1}\frac{G_{ij}^*}{\sqrt{G_{ii}^*G_{jj}^*}},
\end{align*}
\begin{align*}
\cos(\pi-\theta_{ij})=\langle v_i^*,v_j^*\rangle&=-\sqrt{\frac{|G|}{G_{ii}}}\cdot\frac{G_{ij}}{|G|}\cdot-\sqrt{\frac{|G|}{G_{jj}}}\\
&=\dfrac{G_{ij}}{\sqrt{G_{ii}G_{jj}}},
\end{align*}
which implies (\ref{m3}).  

We can also arbitrarily rearrange the order of vertices of $\Delta^n$ and obtain the formulas similar to (\ref{m3}). By law of sines and cosines and (\ref{m3}), the dihedral angles $\theta_{ij}$ of the Euclidean $n$-simplex $\Delta^n$ are real analytic functions of edge lengths $\ell_{ij}$ and vice versa up to scaling. 

(Spherical case.) The proof has been omitted since it is quite similar to the proof in hyperbolic case.

(Ideal hyperbolic case.) Since $\Delta^n$ is an ideal hyperbolic $n$-simplex, then $\Delta^n$ is the Gromov-Hausdorff limit of a sequence of hyperbolic $n$-simplices $\{\Delta^n_m\}_{m=1}^{\infty}\subset\mathbb{H}^n$, which implies that
\begin{equation}\label{xi20}
\theta_{ij}^{(m)}\rightarrow\theta_{ij},~\ell_{ij}^{(m)}\rightarrow\ell_{ij},~\dfrac{G^{(m)(\ast)}_{ij}}{\sqrt{G^{(m)(\ast)}_{ii}G^{(m)(\ast)}_{jj}}}\to \dfrac{G^{(\ast)}_{ij}}{\sqrt{G^{(\ast)}_{ii}G^{(\ast)}_{jj}}},~m\rightarrow\infty.
\end{equation}

For any hyperbolic $n$-simplex $\Delta^n_m(m\ge1)$, by the proof of hyperbolic case, we have 
\begin{equation}\label{xi21}
\cos\theta_{ij}^{(m)}=\dfrac{G_{ij}^{(m)}}{\sqrt{G_{ii}^{(m)}G_{jj}^{(m)}}},~\cosh \ell_{ij}^{(m)}=\dfrac{G^{(m)(\ast)}_{ij}}{\sqrt{G^{(m)(\ast)}_{ii}G^{(m)(\ast)}_{jj}}}.
\end{equation}
Then by (\ref{xi20}) and (\ref{xi21}), we obtain (\ref{m3}) as $m\to\infty$.
\pfd

The following result provides a characterization of hyperbolic $n$-simplices, which is obtained by Luo \cite{Luo97}.

\lm[\cite{Luo97}]\label{angleadmissiblehyperbolic}
Given a set of positive numbers 
$$
\{\theta_{ij}:\theta_{ii}=\pi,\theta_{ij}=\theta_{ji}\in(0,\pi),1\le i, j\le n+1\},
$$
then there exists a hyperbolic $n$-simplex $\Delta^n\subset\mathbb{H}^n$ with $n+1$ vertices $v_1,\cdots,v_{n+1}$ such that the dihedral angle of the $ij$-face of $\Delta^n$ is $\theta_{ij}$ for all $ij$-faces if and only if the matrix $\Theta:=(\cos(\pi-\theta_{ij}))_{1\le i,j\le n+1}$ satisfies
\begin{enumerate}[1.]
\item the determinant is negative;
\item the $n\times n$ principal submatrices are positive definite;
\item the $(i,j)$-cofactors $(i\ne j)$ are positive.
\end{enumerate}
\lmd

We give the proof of Theorem \ref{admissibleidealangle}, which is the generalization of Lemma \ref{angleadmissiblehyperbolic}.

\pf[Proof of Theorem \ref{admissibleidealangle}]
$(\Rightarrow)$ By Proposition \ref{l1}, there exists the dual $n$-simplex $(\Delta^n)^*\subset\ds$ with $n+1$ vertices $v_1^*,\cdots,v_{n+1}^*$ such that $(\ref{17})$ holds. Moreover, by the proof of Corollary \ref{conedge}, we have 
$$
\langle v_i^* ,v_j^* \rangle_{n,1}=\cos(\pi-\theta_{ij}),~1\le i, j\le n+1,
$$
which implies that $\Theta=G^*$.

Define $M^*:=(v_1^{*},v_2^{*},\cdots,v_{n+1}^{*})$, where the vertices $v_1^{*},v_2^{*},\cdots,v_{n+1}^{*}\in\mathbb{R}^{n+1}$ are considered as the column vectors. By the proof of Proposition \ref{l1}, the $n+1$ vertices $v_1^{*},v_2^{*},\cdots,v_{n+1}^{*}$ are linearly independent, which implies that the matrix $M^{*}$ is full-rank and $|M^{*}|\ne 0$. Moreover, we have
\begin{equation}\label{pio}
\begin{aligned}
\langle M^{*},M^{*}\rangle_{n,1}&=M^{*T}I_{n,1}M^{*}=(v_i^{*T}I_{n,1}v_j^{*})_{1\le i,j\le n+1}\\
&=(\langle v_i^{*},v_j^{*} \rangle_{n,1})_{1\le i,j\le n+1}\\
&=G^{*}.
\end{aligned}
\end{equation}
By (\ref{pio}), we have 
\begin{equation}\label{g3}
|\Theta|=|G^{*}|=|M^{*T}I_{n,1}M^{*}|=-|M^{*}|^2<0.
\end{equation}

For any vertex $v_i\in\mathbb{H}^n$, by $(\ref{17})$, we have
$$
\Span\{v_i\}\oplus H_i^*=\mathbb{R}^{n,1},
$$
where $H_i^*=\Span\{v_1^*,\cdots,\hat{v}^*_i,\cdots,v_{n+1}^*\}$ and $\hat{v}_i^*$ means remove the vertex $v_i^*$ from $\{v_1^*,\cdots,v_{n+1}^*\}$. 

Since $\langle v_i,v_i\rangle_{n,1}=-1<0$, then the signature of $\langle \cdot,\cdot\rangle_{n,1}|_{H_i^*}$ is $(n,0,0)$, which implies that the $n\times n$ matrix $\Theta(i,i)=G^*(i,i)=(\langle v_p^*,v_q^*\rangle_{n,1})_{p,q\ne i}$ is positive definite.     

For any ideal vertex $v_i\in\partial_{\infty}\mathbb{H}^n$, $\langle v_i, v_i\rangle_{n,1}=0$. Moreover, by $(\ref{17})$, we have $\Span\{v_i\}\subset\Span\{v_i\}^{\perp}=H_i^*$ and the signature of $\langle \cdot,\cdot\rangle_{n,1}|_{H_i^*}$ is $(n-1,1,0)$, which implies that $|\Theta(i,i)|=|G^*(i,i)|=|(\langle v_p^*,v_q^*\rangle_{n,1})_{p,q\ne i}|=0$ and the $n\times n$ matrix $\Theta(i,i)$ is singular.     

Since $v_i\in H_i^*$, suppose that $v_i=\sum_{j\ne i}k_jv_j^*$, by (\ref{17}), we have
$$
k_j=\dfrac{\langle v_j,v_i\rangle_{n,1}}{\langle v_j,v_j^*\rangle_{n,1}}>0,~\text{for all}~j\ne i
$$
and $v_1^*,\cdots,\hat{v}^*_i,\cdots,v_i(\text{j-th}),\cdots,v_{n+1}^*$ is a basis of $H_i^*$. Then we have
$$
\Span\{v_i\}\oplus H_{ij}^*=H_i^*,~~\text{for all}~j\ne i,
$$
where $H_{ij}^*=\Span\{v_1^*,\cdots,\hat{v}^*_i,\cdots,\hat{v}^*_j,\cdots,v_{n+1}^*\}$. Since the signature of $\langle \cdot,\cdot\rangle_{n,1}|_{H_i^*}$ is $(n-1,1,0)$ and $\langle v_i, v_i\rangle_{n,1}=0$, then the signature of $\langle \cdot,\cdot\rangle_{n,1}|_{H_{ij}^*}$ is $(n-1,0,0)$, which implies that the $(n-1)\times (n-1)$ principal submatrix $\Theta[i,j]=G^*[i,j]=(\langle v_p^*,v_q^*\rangle_{n,1})_{p,q\ne i,j}$ of $\Theta(i,i)$ is positive definite for all $j\ne i$. 

For any two ideal vertices $v_i,v_j(i\ne j)\in\partial_{\infty}\mathbb{H}^n$, by above argument, we have $v_i\in H_i^*$($v_j\in H_j^*$, resp.) and $v_1^*,\cdots,\hat{v}^*_i,\cdots,v_i(\text{j-th}),\cdots,v_{n+1}^*$($v_1^*,\cdots,$ $v_j(\text{i-th}),\cdots,\hat{v}_j^*,\cdots,v_{n+1}^*$) is a basis of $H_i^*$($H_j^*$, resp.). Moreover, suppose that $v_i=\sum_{m\ne i}k_m^{(i)}v_m^*$ and $v_j=\sum_{m\ne j}k_m^{(j)}v_m^*$, then we have
\begin{equation}\label{xuhao}
k_m^{(i)}=\dfrac{\langle v_m,v_i\rangle_{n,1}}{\langle v_m,v_m^*\rangle_{n,1}}>0,~m\ne i;~k_m^{(j)}=\dfrac{\langle v_m,v_j\rangle_{n,1}}{\langle v_m,v_m^*\rangle_{n,1}}>0,~m\ne j,
\end{equation}

Define $M_i^*:=(v_1^*,\cdots,\hat{v}^*_i,\cdots,v_{n+1}^*)$, $\widetilde{M}_i^*:=(v_1^*,\cdots,\hat{v}^*_i,\cdots,v_i,\cdots,v_{n+1}^*)$ and $M_j^*:=(v_1^*,\cdots,\hat{v}^*_j,\cdots,v_{n+1}^*)$, $\widetilde{M}_j^*:=(v_1^*,\cdots,v_j,\cdots,\hat{v}_j^*,\cdots,v_{n+1}^*)$, then by (\ref{xuhao}), we have
\begin{equation}\label{xu1}
\widetilde{M}_i^*=M_i^*\Lambda_i,~\widetilde{M}_j^*=M_j^*\Lambda_j,
\end{equation}
where 
$$
\Lambda_i=\begin{pmatrix}
   1&&&k_1^{(i)}&\\
   &\ddots&&\vdots&\\
   &&1&k_{j-1}^{(i)}&\\
   &&&k_j^{(i)}&\\
   &&&k_{j+1}^{(i)}&1\\
   &&&\vdots&&\ddots\\
   &&&k_{n+1}^{(i)}&&&1\\
\end{pmatrix},~
\Lambda_j=\begin{pmatrix}
   1&&&k_1^{(j)}&\\
   &\ddots&&\vdots&\\
   &&1&k_{i-1}^{(j)}&\\
   &&&k_i^{(j)}&\\
   &&&k_{i+1}^{(j)}&1\\
   &&&\vdots&&\ddots\\
   &&&k_{n+1}^{(j)}&&&1\\
\end{pmatrix}.
$$
Moreover, we have $|\Lambda_i|=k_j^{(i)}>0$ and $|\Lambda_j|=k_i^{(j)}>0$. 

By (\ref{xu1}), we have
\begin{equation}\label{xu2}
\begin{aligned}
G^*(i,j)=\langle M_i^*,M_j^*\rangle_{n,1}&=M_i^{*T}I_{n,1}M_j^*\\
&=(\widetilde{M}_i^*\Lambda_i^{-1})^TI_{n,1}\widetilde{M}_j^*\Lambda_j^{-1}\\
&=(\Lambda_i^{-1})^T\widetilde{M}_i^{*T}I_{n,1}\widetilde{M}_j^*\Lambda_j^{-1}\\
&=(\Lambda_i^{-1})^T\langle \widetilde{M}_i^{*},\widetilde{M}_j^*\rangle_{n,1}\Lambda_j^{-1}.
\end{aligned}
\end{equation}
By (\ref{17}), we have
\begin{equation}\label{xu3}
|\langle \widetilde{M}_i^{*},\widetilde{M}_j^*\rangle_{n,1}|=(-1)^{i+j-1}\langle v_i,v_j\rangle_{n,1}|G^*[i,j]|.
\end{equation}

Since $v_i,v_j$ are ideal vertices, then by above argument, $G^*[i,j]$ is positive definite, which implies that $|G^*[i,j]|>0$. Moreover, by (\ref{xu2}) and (\ref{xu3}), we have
$$
\Theta_{ij}=G_{ij}^*=(-1)^{i+j}|G^*(i,j)|=-\dfrac{\langle v_i,v_j\rangle_{n,1}}{k_j^{(i)}k_i^{(j)}}|G^*[i,j]|>0.
$$

For any vertex $v_i\in\mathbb{H}^n$ and ideal vertex $v_j\in\partial_{\infty}\mathbb{H}^n$, by above argument, we have
\begin{equation}\label{xu4}
\begin{aligned}
G^*(i,j)=\langle M_i^*,M_j^*\rangle_{n,1}&=M_i^{*T}I_{n,1}M_j^*\\
&=M_i^{*T}I_{n,1}\widetilde{M}_j^*\Lambda_j^{-1}\\
&=\langle M_i^*,\widetilde{M}_j^*\rangle_{n,1}\Lambda_j^{-1}.
\end{aligned}
\end{equation}
By (\ref{17}), we have
\begin{equation}\label{xu5}
|\langle M_i^*,\widetilde{M}_j^*\rangle_{n,1}|=(-1)^{i+j-1}\langle v_j,v_j^*\rangle_{n,1}|G^*[i,j]|.
\end{equation}
Since $v_j$ are ideal vertices, then by above argument, $G^*[i,j]$ is positive definite, which implies that $|G^*[i,j]|>0$. Moreover, by (\ref{xu4}) and (\ref{xu5}), we have
$$
\Theta_{ij}=G_{ij}^*=(-1)^{i+j}|G^*(i,j)|=-\dfrac{\langle v_j,v_j^*\rangle_{n,1}}{k_i^{(j)}}|G^*[i,j]|>0.
$$

For any two vertices $v_i,v_j(i\ne j)\in\mathbb{H}^n$, by the proof of Lemma \ref{lengthandangle}, we have
$$
-\cosh \ell_{ij}=\langle v_i,v_j\rangle_{n,1}=-\frac{G_{ij}^*}{\sqrt{G_{ii}^*G_{jj}^*}}<0,
$$
which implies that $\Theta_{ij}=G_{ij}^*>0$. This complete the proof of ``$\Rightarrow$''.

($\Leftarrow$) Define $V:=\{1,\cdots,n+1\}$, then we have $V=V_{\infty}\bigsqcup V\setminus V_{\infty}$, where $V_{\infty}$ is the set consisting of all $i$ that satisfy condition $3$.

The proof is divided into the following two steps:

\noindent\textbf{Step 1:} If $V_{\infty}=\varnothing$, then by Lemma \ref{angleadmissiblehyperbolic}, we have proved this direction. 

If $|V_{\infty}|=1$, suppose that $V_{\infty}=\{i\}$, for any $j\in V\setminus\{i\}$, we can choose $v_k^*\in\ds$ for all $k\in V\setminus \{j\}$ such that 
\begin{equation}\label{xu71}
\langle v_k^*,v_l^*\rangle_{n,1}=\cos (\pi-\theta_{kl}), ~\text{for all}~ k,l\in V\setminus \{j\}
\end{equation} 
and $\{v_k^*\}_{k\in V\setminus \{j\}}$ are linearly independent.
Moreover, we have
\begin{equation}\label{xu61}
   \Span\{v_k^*\}_{k\in V\setminus \{j\}}\oplus\Span\{v_k^*\}_{k\in V\setminus \{j\}}^{\perp} =\mathbb{R}^{n,1}. 
\end{equation}

By assumption and (\ref{xu71}), since $\Theta(j,j)=(\langle v_k^*,v_l^*\rangle_{n,1})_{k,l\in V\setminus \{j\}}$ is positive definite, then by $(\ref{xu61})$, the signature of $\langle \cdot,\cdot\rangle_{n,1}$ on $\Span\{v_k^*\}_{k\in V\setminus \{j\}}^{\perp}$ is $(0,0,1)$.

Choose a basis $u_j$ of $\Span\{v_k^*\}_{k\in V\setminus \{j\}}^{\perp}$ such that $\langle u_j,u_j\rangle_{n,1}=-1$. Moreover, choose $m_k\in\mathbb{R}$ for all $k\in V$ such that
\begin{equation}\label{xu81}
    \begin{cases}
    v_j^*:=m_ju_j+\sum_{k\in V\setminus \{j\}}m_kv_k^*,\\
    \langle v_j^*,v_k^*\rangle_{n,1}=\cos(\pi-\theta_{jk}), ~k\in V\setminus \{j\}.
\end{cases} 
\end{equation}
Then for any $k\in V\setminus \{j\}$, by (\ref{xu81}), we have
\begin{equation}\label{xu101}
\begin{aligned}
\sum_{l\in V\setminus \{j\}}m_l\langle v_k^*,v_l^*\rangle_{n,1}&=\langle v_k^*,\sum_{l\in V\setminus \{j\}}m_lv_l^*\rangle_{n,1}\\
&=\langle v_k^*, v_j^*-m_ju_j\rangle_{n,1}\\
&=\langle v_k^*,v_j^*\rangle_{n,1}.
\end{aligned}
\end{equation}

Define $G_{V\setminus \{j\}}^*:=(\langle v_k^*,v_l^*\rangle_{n,1})_{k,l\in V\setminus \{j\}}$, $\textbf{m}_{ V\setminus \{j\}}=(m_{k})_{k\in V\setminus \{j\}}^T$ and $\alpha_{V\setminus \{j\}}:=(\langle v_k^*,v_j^*\rangle_{n,1})_{_{k\in V\setminus \{j\}}}^T$, then by (\ref{xu101}), we have 
\begin{equation}\label{xu91}
G_{V\setminus \{j\}}^*\textbf{m}_{ V\setminus \{j\}}=\left(\sum_{l\in V\setminus \{j\}}m_l\langle v_k^*,v_l^*\rangle_{n,1}\right)_{k\in V\setminus \{j\}}^T=\alpha_{V\setminus \{j\}}.
\end{equation}

By assumption, since the determinant of $\Theta$ is negative, then after exchanging some rows and columns of $\Theta$ and by (\ref{xu71}), (\ref{xu81}), we have
\begin{equation}\label{xu121}
\begin{vmatrix}
 1& \alpha_{V\setminus \{j\}}^T\\
 \alpha_{V\setminus \{j\}} & G_{V\setminus \{j\}}^*
\end{vmatrix}<0.
\end{equation}

Define $I_{V\setminus\{j\}}:=\text{diag}\{1\}_{k\in V\setminus\{j\}}$, then we have
\begin{equation}\label{xu131}
\begin{pmatrix}
  1& -\alpha_{V\setminus \{j\}}^TG_{V\setminus \{j\}}^{*-1}\\
  \mathbf{0}& I_{V\setminus \{j\}}
\end{pmatrix}
\begin{pmatrix}
  1& \alpha_{V\setminus \{j\}}^T\\
 \alpha_{V\setminus \{j\}} & G_{V\setminus \{j\}}^*
\end{pmatrix}=
\begin{pmatrix}
  1-\alpha_{V\setminus \{j\}}^TG_{V\setminus \{j\}}^{*-1}\alpha_{V\setminus \{j\}}& \mathbf{0}\\
  \alpha_{V\setminus \{j\}}& G_{V\setminus \{j\}}^*
\end{pmatrix}.
\end{equation}
Moreover, since $G_{V\setminus \{j\}}^*$ is positive definite, by (\ref{xu121}) and (\ref{xu131}), we have 
$$
1-\alpha_{V\setminus \{j\}}^TG_{V\setminus \{j\}}^{*-1}\alpha_{V\setminus \{j\}}<0.
$$

Fix $m_j=\sqrt{\alpha_{V\setminus \{j\}}^TG_{V\setminus \{j\}}^{*-1}\alpha_{V\setminus \{j\}}-1}>0$, then by (\ref{xu81}), $\{m_k\}_{k\in V\setminus \{j\}}$ are uniquely determined by $m_j$ and $\langle v_j^*,v_j^*\rangle_{n,1}=1$. Moreover, by (\ref{xu71}) and (\ref{xu81}), we have 
$$
\Theta=G^*=(\langle v_k^*,v_l^*\rangle_{n,1})_{1\le k,l\le n+1}.
$$

If $|V_{\infty}|\ge 2$, for any $i\ne j\in V_{\infty}$, we can choose $v_k^*\in\ds$ for all $k\in V\setminus \{i,j\}$ such that 
\begin{equation}\label{xu7}
\langle v_k^*,v_l^*\rangle_{n,1}=\cos (\pi-\theta_{kl}), ~\text{for all}~ k,l\in V\setminus \{i,j\}
\end{equation}
and $\{v_k^*\}_{k\in V\setminus \{i,j\}}$ are linearly independent. Moreover, we have
\begin{equation}\label{xu6}
   \Span\{v_k^*\}_{k\in V\setminus \{i,j\}}\oplus\Span\{v_k^*\}_{k\in V\setminus \{i,j\}}^{\perp} =\mathbb{R}^{n,1}. 
\end{equation}

By (\ref{xu7}), $(\langle v_k^*,v_l^*\rangle_{n,1})_{k,l\in V\setminus \{i,j\}}$ is a $(n-1)\times (n-1)$ principal submatrix of $\Theta$, then by assumption, $(\langle v_k^*,v_l^*\rangle_{n,1})_{k,l\in V\setminus \{i,j\}}$ is positive definite. Moreover, by $(\ref{xu6})$, the signature of $\langle \cdot,\cdot\rangle_{n,1}$ on $\Span\{v_k^*\}_{k\in V\setminus \{i,j\}}^{\perp}$ is $(1,0,1)$. 

Choose a basis $u_i,u_j$ of $\Span\{v_k^*\}_{k\in V\setminus \{i,j\}}^{\perp}$ such that 
\begin{equation}\label{aa3}
\begin{cases}
 \langle u_s,u_s\rangle_{n,1}=0, ~s\in\{i,j\},\\
 \langle u_i,u_j\rangle_{n,1}<0.
\end{cases} 
\end{equation}
Moreover, choose $m^{(i)}_i,m^{(j)}_j\in\mathbb{R}$ and $(m^{(i)}_k)_{k\in V\setminus \{i,j\}},(m^{(j)}_k)_{k\in V\setminus \{i,j\}}\in\mathbb{R}^{V\setminus \{i,j\}}$ such that
\begin{equation}\label{xu8}
    \begin{cases}
    v_s^*:=m^{(s)}_su_s+\sum_{k\in V\setminus \{i,j\}}m^{(s)}_kv_k^*,~s\in\{i,j\},\\
    \langle v_s^*,v_k^*\rangle_{n,1}=\cos(\pi-\theta_{sk}), ~ k\in V\setminus \{i,j\},\\
    \langle v_i^*,v_j^*\rangle_{n,1}=\cos(\pi-\theta_{ij}).
\end{cases} 
\end{equation}

Fix $m^{(i)}_i$, then by (\ref{xu8}), $m^{(j)}_j, (m^{(i)}_k)_{k\in V\setminus \{i,j\}},(m^{(j)}_k)_{k\in V\setminus \{i,j\}}$ are uniquely determined by $m^{(i)}_i$. Moreover, for any $k\in V\setminus \{i,j\}$ and $s\in\{i,j\}$, by (\ref{xu8}), we have 
\begin{equation}\label{xu10}
\begin{aligned}
\sum_{l\in V\setminus \{i,j\}}m_l^{(s)}\langle v_k^*,v_l^*\rangle_{n,1}&=\langle v_k^*,\sum_{l\in V\setminus \{i,j\}}m_l^{(s)}v_l^*\rangle_{n,1}\\
&=\langle v_k^*, v_s^*-m_s^{(s)}u_s\rangle_{n,1}\\
&=\langle v_k^*,v_s^*\rangle_{n,1}.
\end{aligned}
\end{equation}

Define $G_{V\setminus \{i,j\}}^*:=(\langle v_k^*,v_l^*\rangle_{n,1})_{k,l\in V\setminus \{i,j\}}$, $\textbf{m}_{ V\setminus \{i,j\}}^{(s)}=(m_{k}^{(s)})_{k\in V\setminus \{i,j\}}^T$ and $\alpha^{(s)}_{V\setminus \{i,j\}}:=(\langle v_k^*,v_s^*\rangle_{n,1})_{_{k\in V\setminus \{i,j\}}}^T$, where $s\in\{i,j\}$. Then by (\ref{xu10}), we have 
\begin{equation}\label{xu9}
G_{V\setminus \{i,j\}}^*\textbf{m}_{ V\setminus \{i,j\}}^{(s)}=\left(\sum_{l\in V\setminus \{i,j\}}m_l^{(s)}\langle v_k^*,v_l^*\rangle_{n,1}\right)_{k\in V\setminus \{i,j\}}^T=\alpha^{(s)}_{V\setminus \{i,j\}}.
\end{equation}
Moreover, by (\ref{aa3}), (\ref{xu8}) and (\ref{xu9}), for $s\in\{i,j\}$, we have
\begin{equation}\label{xu11}
\begin{aligned}
\langle v_s^*,v_s^*\rangle_{n,1}&=\langle m^{(s)}_su_s+\sum_{k\in V\setminus \{i,j\}}m^{(s)}_kv_k^*,m^{(s)}_su_s+\sum_{k\in V\setminus \{i,j\}}m^{(s)}_kv_k^* \rangle_{n,1}\\
&=\sum_{k,l\in V\setminus \{i,j\} }m^{(s)}_km^{(s)}_l\langle v_k^*,v_l^*\rangle_{n,1}\\
&=\textbf{m}_{ V\setminus \{i,j\}}^{(s)T}G_{V\setminus \{i,j\}}^*\textbf{m}_{ V\setminus \{i,j\}}^{(s)}\\
&=\alpha^{(s)T}_{V\setminus \{i,j\}}G_{V\setminus \{i,j\}}^{*-1}\alpha^{(s)}_{V\setminus \{i,j\}}.
\end{aligned}
\end{equation}

By assumption, since the principal submatrices $\Theta(i,i),\Theta(j,j)$ are singular, then after exchanging some rows and columns, for $s\in\{i,j\}$, we have
\begin{equation}\label{xu12}
\begin{vmatrix}
  1& \alpha^{(s)T}_{V\setminus \{i,j\}}\\
  \alpha^{(s)}_{V\setminus \{i,j\}}& G_{V\setminus \{i,j\}}^*
\end{vmatrix}=0.
\end{equation}

Define $I_{V\setminus\{i,j\}}:=\text{diag}\{1\}_{k\in V\setminus\{i,j\}}$ , then for $s\in\{i,j\}$, we have
\begin{equation}\label{xu13}
\begin{pmatrix}
  1& -\alpha^{(s)T}_{V\setminus \{i,j\}}G_{V\setminus \{i,j\}}^{*-1}\\
  \mathbf{0}& I_{V\setminus \{i,j\}}
\end{pmatrix}
\begin{pmatrix}
   1& \alpha^{(s)T}_{V\setminus \{i,j\}}\\
  \alpha^{(s)}_{V\setminus \{i,j\}}& G_{V\setminus \{i,j\}}^*
\end{pmatrix}=
\begin{pmatrix}
  1-\alpha^{(s)T}_{V\setminus \{i,j\}}G_{V\setminus \{i,j\}}^{*-1} \alpha^{(s)}_{V\setminus \{i,j\}}& \mathbf{0}\\
  \alpha^{(s)}_{V\setminus \{i,j\}}& G_{V\setminus \{i,j\}}^*
\end{pmatrix}.
\end{equation}
Moreover, since $G_{V\setminus \{i,j\}}^*$ is positive definite, by (\ref{xu12}) and (\ref{xu13}), we have 
$$
1-\alpha^{(s)T}_{V\setminus \{i,j\}}G_{V\setminus \{i,j\}}^{*-1} \alpha^{(s)}_{V\setminus \{i,j\}}=0,
$$
which implies that $\langle v_s^*,v_s^*\rangle_{n,1}=\alpha^{(s)T}_{V\setminus \{i,j\}}G_{V\setminus \{i,j\}}^{*-1} \alpha^{(s)}_{V\setminus \{i,j\}}=1$, where $s\in\{i,j\}$. Moreover, by (\ref{aa3}) and (\ref{xu8}, we have
$$
\Theta=G^*=(\langle v_k^*,v_l^*\rangle_{n,1})_{1\le k,l\le n+1}.
$$

\noindent\textbf{Step 2:} Define 
\begin{equation}\label{xu15}
v_i:=\begin{cases}
 \sum_{j=1}^{n+1}G_{ij}^*v_j^*, & \text{ if } i\in V_{\infty}, \\
  \dfrac{1}{\sqrt{-G_{ii}^*|G^*|}}\sum_{j=1}^{n+1}G_{ij}^*v_j^*,& \text{ if } i\in V\setminus V_{\infty}.
\end{cases}
\end{equation}
It is not difficult to check that
\begin{equation}\label{xu14}
\begin{cases}
\sum_{j=1}^{n+1}G_{ij}^*\langle v_i^*,v_j^*\rangle_{n,1}=|G^*|,&1\le i\le n+1,\\
\sum_{m=1}^{n+1}G_{im}^*\langle v_j^*,v_m^*\rangle_{n,1}=0,&1\le i\ne j\le n+1.
\end{cases}
\end{equation}

For any $i\in V_{\infty}$, by assumption, $G^*(i,i)=\Theta(i,i)$ is a singular matrix, which implies that $G_{ii}^*=(-1)^{2i}|G(i,i)|=0$. Moreover, by (\ref{xu14}), we have
\begin{equation}\label{xu16}
\langle v_i,v_i\rangle_{n,1}=\langle \sum_{j=1}^{n+1}G_{ij}^*v_j^*,\sum_{j=1}^{n+1}G_{ij}^*v_j^*\rangle_{n,1}=G_{ii}^*|G^*|=0,
\end{equation}
which implies that $v_i\in\partial_{\infty}\mathbb{H}^n$. Moreover, we have
\begin{equation}\label{xu18}
\begin{cases}
\langle v_i,v_i^*\rangle_{n,1}=\sum_{j=1}^{n+1}G_{ij}^*\langle v_i^*,v_j^*\rangle_{n,1}=|G^*|<0,\\
\\
\langle v_i,v_j^*\rangle_{n,1}=\sum_{m=1}^{n+1}G_{im}^*\langle v_j^*,v_m^*\rangle_{n,1}=0,~1\le j\ne i\le n+1.
\end{cases}
\end{equation}

For any $i\in V\setminus V_{\infty}$, by (\ref{xu14}), we have
\begin{equation}\label{xu17}
\langle v_i,v_i\rangle_{n,1}=-\dfrac{1}{G_{ii}^*|G^*|}\langle \sum_{j=1}^{n+1}G_{ij}^*v_j^*,\sum_{j=1}^{n+1}G_{ij}^*v_j^*\rangle_{n,1}=-\dfrac{1}{G_{ii}^*|G^*|}\cdot G_{ii}^*|G^*|=-1,
\end{equation}
which implies that $v_i\in\mathbb{H}^n$. Moreover, we have
\begin{equation}\label{xu19}
\begin{cases}
\langle v_i,v_i^*\rangle_{n,1}=\dfrac{1}{\sqrt{-G_{ii}^*|G^*|}}\sum_{j=1}^{n+1}G_{ij}^*\langle v_i^*,v_j^*\rangle_{n,1}=-\sqrt{\dfrac{-|G^*|}{G_{ii}^*}}<0,\\
\langle v_i,v_j^*\rangle_{n,1}=\dfrac{1}{\sqrt{-G_{ii}^*|G^*|}}\sum_{m=1}^{n+1}G_{im}^*\langle v_j^*,v_m^*\rangle_{n,1}=0,~1\le j\ne i\le n+1.
\end{cases}
\end{equation}

By (\ref{xu15}), (\ref{xu16}) and (\ref{xu17}), we have $v_1,\cdots,v_{n+1}\in\overline{\mathbb{H}^n}$ are linearly independent. Then we can define the ideal hyperbolic $n$-simplex $\Delta^n$ is the geodesic convex hull of $n+1$ vertices $v_1,v_2,\cdots,v_{n+1}$ and $\Delta^n$ is unique up to isometry. By (\ref{xu18}) and (\ref{xu19}), we have the dihedral angle of the $ij$-face of $\Delta^n$ is $\theta_{ij}$ for all $ij$-faces. \pfd

The following result shows a geometric property of the conformal dual $n$-simplex of an ideal hyperbolic $n$-simplex.

\begin{corollary}\label{po5}
  Given an ideal hyperbolic $n$-simplex $\Delta^n\subset\overline{\mathbb{H}^n}$ and the dual complex $(\Delta^n)^*\subset\ds$, if $(\Delta^n)^*$ is conformal, then there exist $n+1$ positive numbers $r_1^*,\cdots,r_{n+1}^*\in(0,\pi)$ such that 
  $$
\ell_{ij}^*=r_i^*+r_j^*,~\text{for all geodesic edges}~e_{ij}^*~\text{of}~(\Delta^n)^*.
  $$
\end{corollary}

\begin{proof}
Since $(\Delta^n)^*$ is conformal, by definition and Corollary \ref{conedge}, there exist $n+1$ positive numbers $r_1^*,\cdots,r_{n+1}^*\in\mathbb{R}$ such that 
\begin{equation}\label{po6}
\ell_{ij}^*=\pi-\theta_{ij}=r_i^*+r_j^*\in(0,\pi)
\end{equation}
for all $1\le i\ne j\le n+1$, which implies that 
\begin{equation}\label{po7}
  r_i^*=\dfrac{\ell_{ij}^*+\ell_{ik}^*-\ell_{jk}^*}{2},~\text{for all}~1\le i\ne j\ne k\le n+1. 
\end{equation}

If $n=3$, since $\Delta^3$ is an ideal hyperbolic tetrahedron, then there exists a sequence of hyperbolic tetrahedra $\{\Delta^3_m\}_{m=1}^{\infty}\subset\mathbb{H}^3$ such that $\Delta^3_m\to \Delta^3(m\rightarrow\infty)$, which is equivalent to $(\Delta^3_m)^*\to (\Delta^3)^*(m\rightarrow\infty)$, where $(\Delta^3_m)^*$ is the dual tetrahedron of $\Delta^3_m$.  

By the proof of Proposition \ref{l1}, for any $1\le i\le4$, $v_j^{(m)*},v_k^{(m)*},v_l^{(m)*}(i\ne j\ne k\ne l)$ are in the spacelike hyperplane $H_i^{(m)*}=\Span\{v_j^{(m)*},v_k^{(m)*},v_l^{(m)*}\}$, which implies that $v_j^{(m)*},v_k^{(m)*},v_l^{(m)*}\in H_i^{(m)*}\bigcap\mathrm{d}\mathbb{S}^3=\mathbb{S}^2$. By triangle inequality, we have
\begin{equation}\label{po8}
  \ell_{jk}^{(m)*}+\ell_{jl}^{(m)*}-\ell_{kl}^{(m)*}>0,~\text{for all}~1\le i\le 4. 
\end{equation}
Moreover, by (\ref{po7}), we have $r_i^*\ge 0$ for all $1\le i\le 4$ as $m\to\infty$.

If there exists a $1\le i\le4$ such that $r_i^*=0$, then by (\ref{po7}), we have
\begin{equation}\label{po10}
\ell_{ij}^*+\ell_{ik}^*=\ell_{jk}^*,~\ell_{ij}^*+\ell_{il}^*=\ell_{jl}^*,~\ell_{ik}^*+\ell_{il}^*=\ell_{kl}^*,
\end{equation}
which implies that the hyperplanes $H_j^*,H_k^*,H_l^*$ are lightlike and $v_j,v_k,v_l$ are ideal vertices. Moreover, by Corollary \ref{po9}, we have
\begin{equation}\label{po11}
\ell_{ij}^*+\ell_{ik}^*+\ell_{jk}^*=2\pi,~\ell_{ij}^*+\ell_{il}^*+\ell_{jl}^*=2\pi,~\ell_{ik}^*+\ell_{il}^*+\ell_{kl}^*=2\pi.
\end{equation}

By (\ref{po10}) and (\ref{po11}), we have $\ell_{jk}^*=\ell_{jl}^*=\ell_{kl}^*=\pi$, which contradicts (\ref{po6}) and $r_i^*> 0$ for all $1\le i\le 4$. Moreover, by (\ref{po6}), we have $r_i^*\in(0,\pi)$ for all $1\le i\le 4$. 

If $n\ge4$, for any three vertices $v_i^*,v_j^*,v_k^*(i\ne j\ne k)$ of $(\Delta^n)^*$, by Theorem \ref{admissibleidealangle}, $v_i^*,v_j^*,v_k^*$ are in the spacelike subspace $\Span\{v_i^*,v_j^*,v_k^*\}$, which implies that $v_i^*,v_j^*,v_k^*\in\Span\{v_i^*,v_j^*,v_k^*\}\bigcap\ds=\mathbb{S}^2$. By triangle inequality and (\ref{po7}), we have
$$
  r_i^*=\dfrac{\ell_{ij}^*+\ell_{ik}^*-\ell_{jk}^*}{2}>0
$$
for all $1\le i\le n+1$. Moreover, by (\ref{po6}), we have $r_i^*\in(0,\pi)$ for all $1\le i\le n+1$. This completes the proof.
\end{proof}

\section{Admissible space of conformal tetrahedra in constant curvature spaces}\label{s5}

In this section, we study the hyperbolic (Euclidean, spherical, resp.) admissible space $\Omega^{\mathbb{H}}_t(\Omega^{\mathbb{E}}_t,\Omega^{\mathbb{S}}_t,\text{resp}.)$.

The admissible space $\Omega^{\mathbb{H}}_t(\Omega^{\mathbb{E}}_t,\Omega^{\mathbb{S}}_t,\text{resp}.)$ is rewritten as the following form for the convenience of notations:  
\begin{align*}
\Omega^{\mathbb{H}}_t:=&\{r=(r_1,\cdots,r_4)\in\R^4_{+}:Q^{\mathbb{H}}_3(r)>0\}\\
(\Omega^{\mathbb{E}}_t:=&\{r=(r_1,\cdots,r_4)\in\R^4_{+}:Q^\mathbb{E}_3(r)>0\},\\
\Omega^{\mathbb{S}}_t:=&\{r=(r_1,\cdots,r_4)\in(0,\pi)^4:Q^\mathbb{S}_3(r)>0,\\
&0<r_i+r_j<\pi,1\le i\ne j\le 4\},\ \text{resp.}),
\end{align*}
where
\begin{align}
Q^{\mathbb{H}}_3(r):=&(\sum_{i=1}^4\coth r_i)^2\nonumber-2(\sum_{i=1}^4\coth^2 r_i)+4\nonumber\\
(Q^\mathbb{E}_3(r):=&(\sum_{i=1}^4\frac{1}{r_i})^2
-2(\sum_{i=1}^4\frac{1}{r^2_i}),\\
Q^\mathbb{S}_3(r):=&(\sum_{i=1}^4\cot r_i)^2
-2(\sum_{i=1}^4\cot^2 r_i)-4,\ \text{resp.}).\nonumber
\end{align}

A hyperbolic (Euclidean, spherical, resp.) tetrahedron $\Delta^3$ is \emph{regular} if all edge lengths of $\Delta^3$ are equal. It is not difficult to check that any regular hyperbolic (Euclidean, spherical, resp.) tetrahedron is conformal and can be parameterized by $r\in\Omega^{\mathbb{H}}_{e}(\Omega^{\mathbb{E}}_e,\Omega^{\mathbb{S}}_e,\text{resp.})\subset\Omega^{\mathbb{H}}_t(\Omega^{\mathbb{E}}_t,\Omega^{\mathbb{S}}_t,\text{resp.})$, where    
\begin{align*}
\Omega^{\mathbb{H}}_e:=&\{r=s(1,\cdots,1)\in\mathbb{R}^4_+:s>0\}\\
(\Omega^{\mathbb{E}}_e:=&\{r=s(1,\cdots,1)\in\R^4_{+}:s>0\},\\
\Omega^{\mathbb{S}}_e:=&\{r=s(1,\cdots,1)\in(0,\pi)^4:0<s<\pi/2\},~\text{resp.}).
\end{align*}
Moreover, the conformal hyperbolic (Euclidean, spherical, resp.) tetrahedron $\Delta^3(r)$ parameterized by $r\in\Omega^{\mathbb{H}}_{e}(\Omega^{\mathbb{E}}_e,\Omega^{\mathbb{S}}_e,\text{resp.})$ is regular.

Cooper and Rivin \cite{Cooper96} give the following topological characterization of the admissible space of conformal hyperbolic (Euclidean, resp.) tetrahedra. 
\lm[\cite{Cooper96}, Theorem 4.2, Note]\label{hyperbolicadmissiblespace}
The admissible space $\Omega^{\mathbb{H}}_t$ $(\Omega^{\mathbb{E}}_t, \text{resp}.)$ is a simply connected domain.
\lmd

The following result gives the same topological characterization of the admissible space of conformal spherical tetrahedra, which is also an important ingredient to prove Theorem \ref{main1}. 
\lm\label{sphericaladmissiblespace}
The admissible space $\Omega^{\mathbb{S}}_t$ is a simply connected domain.
\lmd
\pf
Recall that the spherical admissible space
\begin{align*}
\Omega^{\mathbb{S}}_t=\{&r=(r_1,\cdots,r_4)\in(0,\pi)^4:Q^\mathbb{S}_3(r)>0,\\
&0<r_i+r_j<\pi,1\le i\ne j\le 4\},
\end{align*}
where $Q^\mathbb{S}_3(r)=(\cot r_1+\cdots+\cot r_4)^2-2(\cot^2 r_1+\cdots+\cot^2 r_4)-4$. Then we can construct a homeomorphism $\Phi$ from the admissible space $\Omega^{\mathbb{S}}$ to the space $\mathfrak{X}$, namely  
$$
        \begin{aligned}
			\Phi:\Omega^{\mathbb{S}}_t&\longrightarrow \mathfrak{X},\\\nonumber
			r=(r_1,r_2,r_3,r_4) & \longmapsto x=(\cot r_1,\cot r_2,\cot r_3,\cot r_4),
		\end{aligned}
$$
where
\begin{align}\label{sphericaladmissiblespace3}
\mathfrak X:=\{&x=(x_1,\cdots,x_4)\in\R^4:\mathcal Q^{\mathbb{S}}_3(x)>0,\\\nonumber
&x_i+x_j>0,1\le i\ne j\le 4\}
\end{align}
and $\mathcal Q^{\mathbb{S}}_3(x)=(x_1+\cdots+x_4)^2-2(x^2_1+\cdots+x^2_4)-4$.

By the definition (\ref{sphericaladmissiblespace3}) of the space $\mathfrak{X}$, we have $x=(x_1,x_2,x_3,x_4)\in\mathfrak X$ if and only if 
\begin{equation}\label{ine}
\begin{aligned}
x_1x_2+x_1x_3+x_2x_3&>1,\\
g_2(x_1,x_2,x_3)&<x_4<f_4(x_1,x_2,x_3),\\
x_i+x_j&>0,~~~1\le i\ne j\le 4,\\
\end{aligned}
\end{equation}
where 
\begin{align*}
g_2(x_1,x_2,x_3)&=x_1+x_2+x_3-2\sqrt{x_1x_2+x_1x_3+x_2x_3-1},\\
f_4(x_1,x_2,x_3)&=x_1+x_2+x_3+2\sqrt{x_1x_2+x_1x_3+x_2x_3-1}.
\end{align*}
Moreover, the inequalities (\ref{ine}) hold if and only if   
\begin{equation}\label{spa}
\begin{aligned}
x_2&>f_1(x_1),\\
x_3&>f_2(x_1,x_2),\\
f_3(x_1,x_2,x_3)&<x_4<f_4(x_1,x_2,x_3),
\end{aligned}
\end{equation}
where 
\begin{align*}
f_1(x_1)&=-x_1,\\
f_2(x_1,x_2)&=\max\{g_1(x_1,x_2),-x_1,-x_2\},\\
f_3(x_1,x_2,x_3)&=\max\{g_2(x_1,x_2,x_3),-x_1,-x_2,-x_3\},\\
g_1(x_1,x_2)&=\dfrac{1-x_1x_2}{x_1+x_2}
\end{align*}
and $f_1, f_2, f_3, f_4$ are continuous functions. By (\ref{spa}), the space $\mathfrak{X}$ can be rewritten as
\begin{align*}
\mathfrak{X}=\{&x=(x_1,x_2,x_3,x_4)\in\mathbb{R}^4:x_2>f_1(x_1), x_3>f_2(x_1,x_2),\\
& f_3(x_1,x_2,x_3)<x_4<f_4(x_1,x_2,x_3)\}.
\end{align*}

Define a projection $\pi_1$ from the space $\mathfrak{X}$ to the space $\mathfrak {X}^{(1)}$, namely
$$
        \begin{aligned}
			\pi_1:\mathfrak X&\longrightarrow \mathfrak {X}^{(1)},\\\nonumber
			x=(x_1,x_2,x_3,x_4) & \longmapsto x^{(1)}=(x_1,x_2,x_3),
		\end{aligned}
$$
where $\mathfrak {X}^{(1)}=\{x^{(1)}=(x_1,x_2,x_3)\in\mathbb{R}^3: x_2>f_1(x_1), x_3>f_2(x_1,x_2)\}$. Moreover, define a map $s_1$ from the space $\mathfrak {X}^{(1)}$ to the space $\mathfrak X$, namely
$$
        \begin{aligned}
			s_1:\mathfrak {X}^{(1)}&\longrightarrow \mathfrak{X},\\\nonumber
			x^{(1)}=(x_1,x_2,x_3) & \longmapsto x=(x^{(1)},\frac{f_3(x^{(1)})+f_4(x^{(1)})}{2}),
		\end{aligned}
$$

Similarly, we can define a projection $\pi_2$ from the space $\mathfrak {X}^{(1)}$ to the space $\mathfrak {X}^{(2)}$, namely
$$
        \begin{aligned}
			\pi_2:\mathfrak {X}^{(1)}&\longrightarrow \mathfrak {X}^{(2)},\\\nonumber
			x^{(1)}=(x_1,x_2,x_3) & \longmapsto x^{(2)}=(x_1,x_2),
		\end{aligned}
$$
where $\mathfrak {X}^{(2)}=\{x^{(2)}=(x_1,x_2)\in\mathbb{R}^2: x_2>f_1(x_1)\}=\{x^{(2)}=(x_1,x_2)\in\mathbb{R}^2:x_1+x_2>0\}$. Note that $\mathfrak {X}^{(2)}$ is contractible, since $\mathfrak {X}^{(2)}$ is homeomorphic to $\mathbb{R}^2$. Moreover, define a map $s_2$ from the space $\mathfrak {X}^{(2)}$ to the space $\mathfrak {X}^{(1)}$, namely
$$
        \begin{aligned}
			s_2:\mathfrak {X}^{(2)}&\longrightarrow \mathfrak {X}^{(1)},\\\nonumber
			x^{(2)}=(x_1,x_2) & \longmapsto x^{(1)}=(x_1,x_2,f_2(x^{(2)})+1),
		\end{aligned}
$$

Then define a map $H_1$($H_2$, resp.) from the space $\mathfrak X\times [0,1]$($\mathfrak X^{(1)}\times [0,1]$, resp.) to the space $\mathfrak X$($\mathfrak X^{(1)}$, resp.), namely
$$
        \begin{aligned}
			H_1:\mathfrak X\times [0,1]&\longrightarrow \mathfrak{X},\\\nonumber
			(x,t) & \longmapsto (x^{(1)},tx_4+(1-t)\cdot\frac{f_3(x^{(1)})+f_4(x^{(1)})}{2}),
		\end{aligned}
$$
$$
        \begin{aligned}
			H_2:\mathfrak X^{(1)}\times [0,1]&\longrightarrow \mathfrak{X}^{(1)},\\\nonumber
			(x^{(1)},t) & \longmapsto (x^{(2)},tx_3+(1-t)\cdot(f_2(x^{(2)})+1)).
		\end{aligned}
$$

The map $H_1$($H_2$, resp.) is the homotopy between $s_1\circ\pi_1$($s_2\circ\pi_2$, resp.) and $id_{\mathfrak X}$($id_{\mathfrak X^{(1)}}$, resp.), where $id_{\mathfrak X}$($id_{\mathfrak X^{(1)}}$, resp.) is the identity map on $\mathfrak {X}$($\mathfrak X^{(1)}$, resp.). Moreover, note that $\pi_1\circ s_1=id_{\mathfrak X^{(1)}}$ and $\pi_2\circ s_2=id_{\mathfrak X^{(2)}}$, where $id_{\mathfrak X^{(2)}}$ is the identity map on $\mathfrak {X}^{(2)}$, then we have 
\begin{equation}\label{con}
\mathfrak X\simeq\mathfrak X^{(1)}\simeq\mathfrak X^{(2)},
\end{equation}
where $``\simeq"$ means ``homotopic equivalent". 

Since $\mathfrak X^{(2)}$ is contractible and $\Omega^{\mathbb{S}}_t$ is homeomorphic to $\mathfrak X$, then by (\ref{con}), $\Omega^{\mathbb{S}}_t$ is contractible, which implies that the admissible space $\Omega^{\mathbb{S}}_t$ is a simply connected domain. 
\pfd

The following proposition is a basic property in hyperbolic geometry.
\prp\label{hyperbolicgeometry}
Given two positive numbers $\alpha\in(0,\pi/2)$ and $c>0$, then there exists a hyperbolic right triangle $\Delta ABC$ such that 
\begin{equation}\label{tri}
\angle CAB=\alpha,~\angle ABC=\pi/2,~d_{-1}(A,B)=c
\end{equation}
if and only if
$\sin\alpha<1/{\cosh c}$.
\prpd

\pf
($\Rightarrow$) By assumption and hyperbolic trigonometry, we have
\begin{align}\label{hypergeometry1}
\sin\alpha=\dfrac{\sinh a}{\sinh b},~\cosh b=\cosh a\cosh c,
\end{align}
where $a=d_{-1}(B,C)$ and $b=d_{-1}(A,C)$. Then by (\ref{hypergeometry1}), we obtain
\begin{align}\label{hypergeometry2}
\sin\alpha=\dfrac{\sinh a}{\sqrt{\cosh^2 a\cosh^2 c-1}}=:f(a).
\end{align}
Moreover, we have
\begin{align}\label{hypergeometry3}
f^\prime(a)=\dfrac{\sinh^2c\cosh a}{(\cosh^2c\cosh^2a-1)^{\frac{3}{2}}}>0,~\text{for\ all}\ a\in(0,\infty),
\end{align}
which implies $f$ is a strictly increasing function on $(0,\infty)$. Since $\lim\limits_{a\rightarrow\infty}f(a)=1/\cosh c$, then by (\ref{hypergeometry2}), we have $\sin\alpha=f(a)<1/\cosh c$.

($\Leftarrow$) Since $\alpha\in(0,\pi/{2})$ and $f(0)=0$, then by assumption, there exists a unique $a\in(0,\infty)$ such that $(\ref{hypergeometry2})$ holds. Moreover, there exists a unique $b\in(0,\infty)$ such that $\sinh b=\sqrt{\cosh^2 a\cosh^2 c-1}$, which implies that (\ref{hypergeometry1}) holds and the hyperbolic right triangle $\Delta ABC$ satisfying (\ref{tri}) exists.        
\pfd

In Euclidean (spherical, resp.) geometry, the analogy of Proposition \ref{hyperbolicgeometry} also holds, which does not have any restrictions of the angle $\alpha$ and edge length $c$. More precisely, we have the following result.
\begin{prop}\label{gong}
Given two positive numbers $\alpha\in(0,\pi/2)$ and $c>0 (c\in(0,\pi),\text{resp.})$, then there exists a Euclidean (spherical, resp.) right triangle $\Delta ABC$ such that 
\begin{equation}
\angle CAB=\alpha,~\angle ABC=\pi/2,~d_{0}(A,B)(d_{1}(A,B), resp.)=c.
\end{equation}
\end{prop}  

\begin{lemma}\label{angle}
    Given a non-degenerate hyperbolic $($Euclidean, spherical, resp.$)$ tetrahedron $P$ with four vertices $v_1,v_2,v_3,v_4$, then there exists the unique unit tangent vector $u_i$ and angle $\theta_i\in(0,\pi/2)$ such that 
    \begin{equation}
    \langle u_i,u_{ij}\rangle_{3,1}=\langle u_i,u_{ik}\rangle_{3,1}=\langle u_i,u_{il}\rangle_{3,1}=\cos\theta_i
    \end{equation}
for all $1\le i\le 4$, where $(1,2,3,4)=(i,j,k,l)$ and $u_{ij}$$(u_{ik}, u_{il}, \text{resp}.)$ is the unit tangent vector of the oriented geodesic edge $\vec{e}_{ij}$ from $v_i$ to $v_j$$(v_k, v_l, \text{resp}.)$ at the vertex $v_i$.      
\end{lemma}

\begin{proof}
(Hyperbolic case.) Since the hyperbolic tetrahedron $P$ is non-degenerate, then for any $1\le i\le 4$, the unit tangent vectors $u_{ij}, u_{ik}, u_{il}\in T_{v_i}\mathbb{H}^3$ are linearly independent and form a basis in the tangent space $T_{v_i}\mathbb{H}^3$. 

Define the Gram matrix $U_i:=(\langle u_{im},u_{in}\rangle_{3,1})_{m,n=\{j,k,l\}}$ of the basis $\{u_{ij}, u_{ik},$ $u_{il}\}$, since the tangent space $T_{v_i}\mathbb{H}^3$ is spacelike under the Minkowski inner product $\langle\cdot,\cdot\rangle_{3,1}$, then the Gram matrix $U_i$ and inverse matrix $U_i^{-1}$ are positive definite. Moreover, there exists a constant $c_i=1/\sqrt{\mathbf{1}^TU_i^{-1}\mathbf{1}}>0$ and a column vector $\textbf{x}_i=(x_j,x_k,x_l)^T=c_iU_i^{-1}\mathbf{1}$ such that
\begin{equation}\label{juzhen}
\textbf{x}_i^TU_i\textbf{x}_i=1,~U_i\textbf{x}_i=c_i\mathbf{1},
\end{equation}
where $\mathbf{1}=(1,1,1)^T$. Define the tangent vector $u_i:=\sum_{m\in\{j,k,l\}}x_mu_{im}\in T_{v_i}\mathbb{H}^3$, then by (\ref{juzhen}) and Cauchy-Schwarz inequality, we have
$$
\langle u_i,u_i\rangle_{3,1}=1,~\langle u_i,u_{ij}\rangle_{3,1}=\langle u_i,u_{ik}\rangle_{3,1}=\langle u_i,u_{il}\rangle_{3,1}=c_i<1,
$$
which also implies that there exists the unique angle $\theta_i\in(0,\pi/2)$ suth that $c_i=\cos\theta_i$. 

(Euclidean and spherical cases.) The proof has been omitted since it is quite similar to the proof in hyperbolic case.
\end{proof}

The unit tangent vector $u_i$ and the angle $\theta_i$ in Lemma \ref{angle} are called \emph{isogonal unit tangent vector} and \emph{isogonal angle} at the vertex $v_i$ of the hyperbolic (Euclidean, spherical, resp.) tetrahedron $\Delta^3$, respectively. The \emph{isogonal geodesic} $\gamma_i$ at the vertex $v_i$ is the geodesic passing through $v_i$ with the isogonal tangent vector $u_i$ at the vertex $v_i$.

A hyperbolic (Euclidean, spherical, resp.) sphere $S$ is the \emph{midsphere} of a hyperbolic (Euclidean, spherical, resp.) tetrahedron $\Delta^3$ if $S$ is externally tangent to all the edges of $\Delta^3$.  

\begin{lemma}\label{mid}
A non-degenerate hyperbolic tetrahedron $\Delta^3$ with a midsphere is conformal.    
\end{lemma}

\begin{proof}
 Suppose that $\Delta^3$ has four vertices $v_1, v_2, v_3, v_4$ and midsphere $S$ with the center $o$, add the geodesic segment $e_{o,i}$($e_{o,ij}$, resp.) connecting $o$ and $v_i$($v_{ij}$, resp.) for all $1\le i\le4$ ($1\le i\ne j\le4$, resp.), where $v_{ij}\in e_{ij}$ is the tangent point of the midsphere $S$ on the geodesic edge $e_{ij}$, note that $v_{ij}=v_{ji}$. Then we have
\begin{equation}\label{edge101}
 e_{o,ij}\perp e_{ij},~e_{o,ik}\perp e_{ik},~e_{o,il}\perp e_{il},
\end{equation}
\begin{equation}\label{edge1}
d_{-1}(o,v_{ij})=d_{-1}(o,v_{ik})=d_{-1}(o,v_{il})
\end{equation}
for all $1\le i\le4$, where $\{1,2,3,4\}=\{i,j,k,l\}$.

By (\ref{edge101}) and (\ref{edge1}), for any $1\le i\le4$, the hyperbolic right triangles $\de ov_iv_{ij}$, $\de ov_iv_{ik}$ and $\de ov_iv_{il}$ are pairwise congruent, and 
$$
d_{-1}(v_i,v_{ij})=d_{-1}(v_i,v_{ik})=d_{-1}(v_i,v_{il}):=r_i>0,
$$
which implies that $\ell_{ij}=d_{-1}(v_i,v_{ij})+d_{-1}(v_j,v_{ij})=r_i+r_j$
for all $1\le i\ne j\le4$ and $\Delta^3$ is a conformal hyperbolic tetrahedron.
\end{proof}

We denote the set consisting of all $r\in\Omega^{\mathbb{H}}_t$ ($\Omega^{\mathbb{E}}_t$, $\Omega^{\mathbb{S}}_t$, resp.) such that the hyperbolic (Euclidean, spherical, resp.) tetrahedron $\Delta^3(r)$ has a midsphere by $\Omega_m^{\mathbb{H}}$ ($\Omega_m^{\mathbb{E}}$, $\Omega_m^{\mathbb{S}}$, resp.), namely
\begin{align}\nonumber
\Omega_m^{\mathbb{H}}(\Omega_m^{\mathbb{E}}, \Omega_m^{\mathbb{S}}, \text{resp}.):=\{&r=(r_1,\cdots,r_4)\in \Omega^{\mathbb{H}}_t (\Omega^{\mathbb{E}}_t, \Omega^{\mathbb{S}}_t, \text{resp}.):
\text{the tetrahedron}
\\& \text{$\Delta^3(r)$ parameterized by}\ r\ \text{has a midsphere}\}.\nonumber
\end{align}

\lm\label{tangentsphere1}
Given a parameter $r=(r_1,r_2,r_3,r_4)\in\Omega^{\mathbb{H}}_t$, then
$r\in\Omega_m^{\mathbb{H}}$ if and only if 
\begin{equation}\label{an}
\sin\theta_i<1/\cosh r_i,~1\le i\le 4,
\end{equation}
where $\theta_i$ is the isogonal angle at the vertex $v_i$ of the hyperbolic tetrahedron $\Delta^3(r)$ for all $1\le i\le 4$.
\lmd

\pf
For any hyperbolic tetrahedron $\Delta^3(r)$ parameterized by $r\in\Omega^{\mathbb{H}}_t$, choose the point $v_{ij}\in e_{ij}$ such that $d_{-1}(v_i,v_{ij})=r_i$ for all $1\le i\ne j\le4$, note that $e_{ij}=e_{ji}$ and $v_{ij}=v_{ji}$. 

($\Rightarrow$) Suppose that $\Delta^3(r)$ has the midsphere $S$ with the center $o$, then add the geodesic segment connecting $v_i$($v_{ij}$, resp.) and $o$ for all $1\le i\le4$($1\le i\ne j\le4$, resp.). Since $\Delta^3(r)$ is a conformal tetrahedron and $S$ is externally tangent to all the edges of $P$, then by Lemma \ref{angle}, for any $1\le i\le 4$, we have $\de ov_iv_{ij}$, $\de ov_iv_{ik}$ and $\de ov_iv_{il}$ are pairwise congruent, and 
$$
\angle ov_{ij}v_i=\angle ov_{ik}v_i=\angle ov_{il}v_i=\pi/{2},~\angle v_{ij}v_io=\angle v_{ik}v_io=\angle v_{il}v_io=\theta_i,
$$
where $\{1,2,3,4\}=\{i,j,k,l\}$. Then by Proposition \ref{hyperbolicgeometry}, we have (\ref{an}).

($\Leftarrow$) For any $1\le i\le 4$, we denote the isogonal tangent vector at the vertex $v_i$ of the hyperbolic tetrahedron $\Delta^3(r)$ by $u_i$, then by Proposition \ref{hyperbolicgeometry} and (\ref{an}), there exists a point $o\in\gamma_i$ such that 
\begin{equation}\label{sphere1}
e_{o,{ij}}\perp e_{ij},~e_{o,{ik}}\perp e_{ik},~e_{o,{il}}\perp e_{il},
\end{equation}
where $\gamma_i$ is the isogonal geodesic at the vertex $v_i$ and $e_{o,ij}$($e_{o,{ik}}$, $e_{o,{il}}$, resp.) is the geodesic segment connecting $o$ and $v_{ij}$($v_{ik}$, $v_{il}$, resp.). Moreover, the hyperbolic right triangles $\de ov_iv_{ij}$, $\de ov_iv_{ik}$ and $\de ov_iv_{il}$ are pairwise congruent, and satisfy 
\begin{equation}\label{sphere2}
d_{-1}(o,v_{ij})=d_{-1}(o,v_{ik})=d_{-1}(o,v_{il}):=R.
\end{equation}

Project the point $o$ perpendicularly onto the face $F_j$($F_k, F_l$, resp.) and denote the projection point on the face $F_j$($F_k, F_l$, resp.) by $o_j$($o_k, o_l$, resp.), then we have
$$
e_{o_j,ik}\perp e_{ik}~(e_{o_k,ij}\perp e_{ij},~e_{o_l,ij}\perp e_{ij},~\text{resp}.),
$$
$$
e_{o_j,il}\perp e_{il}~(e_{o_k,il}\perp e_{il},~e_{o_l,ik}\perp e_{ik},~\text{resp}.),
$$
where $e_{o_p,iq}$ is the geodesic segment connecting $o_p$ and $v_{iq}$ for all $p\ne q\in\{j,k,l\}$. Moreover, adding the geodesic segment connecting $v_{i}$ and $o_j$($o_k,o_l$, resp.), we have $\de o_lv_iv_{ij}$ and $\de o_lv_iv_{ik}$ are congruent and 
$$
\angle o_jv_iv_{ik}=\angle o_jv_iv_{il}~(\angle o_kv_iv_{ij}=\angle o_kv_iv_{il},~\angle o_lv_iv_{ij}=\angle o_lv_iv_{ik},~\text{resp}.), 
$$
which implies that $o_j$($o_k, o_l$, resp.) is on the angle bisector of $\angle v_kv_iv_l$($\angle v_jv_iv_l$, $\angle v_jv_iv_k$, resp.). Since $\Delta^3(r)$ is a conformal tetrahedron, then $o_j$($o_k, o_l$, resp.) is the incenter of the face $F_j$($F_k, F_l$, resp.) and we have
$$
e_{o_j,kl}\perp e_{kl}~(e_{o_k,jl}\perp e_{jl},~e_{o_l,jk}\perp e_{jk},~\text{resp}.),
$$
\begin{equation}\label{sphere3}
e_{o,kl}\perp e_{kl}~(e_{o,jl}\perp e_{jl},~e_{o,jk}\perp e_{jk},~\text{resp}.),
\end{equation}
where $e_{o_p,qn}$($e_{o,pq}$, resp.) is the geodesic segment connecting $o_p$($o$, resp.) and $v_{qn}$($v_{pq}$, resp.) for all $p\ne q\ne n(p\ne q,\text{resp}.)\in\{j,k,l\}$. 

Add the geodesic segment connecting $o$ and $v_j$($v_{k}$, $v_l$ resp.), then we have $\de ov_jv_{ij}$($\de ov_kv_{ik}$, $\de ov_lv_{il}$, resp.), $\de ov_jv_{jk}$($\de ov_kv_{jk}$, $\de ov_lv_{jl}$, resp.) and $\de ov_jv_{jl}$($\de ov_kv_{kl}$, $\de ov_lv_{kl}$, resp.) are pairwise congruent, and 
\begin{equation}\label{sphere4}
d_{-1}(o,v_{jk})=d_{-1}(o,v_{jl})=d_{-1}(o,v_{kl})=R.
\end{equation}

By (\ref{sphere1}), (\ref{sphere2}), (\ref{sphere3}) and (\ref{sphere4}), there exists a hyperbolic sphere $S$ with the center $o$ and radius $R$ such that $S$ is externally tangent to all the edges of the hyperbolic tetrahedron $\Delta^3(r)$, which implies that $S$ is the midsphere of $\Delta^3(r)$ and $r\in\Omega_m^{\mathbb{H}}$.
\pfd

\begin{remark}\label{geo}
    By Lemma \ref{angle} and the proof of Lemma \ref{tangentsphere1}, given a conformal hyperbolic tetrahedron $\Delta^3(r)$ parameterized by $r\in\Omega^{\mathbb{H}}$, then the inequality $($\ref{an}$)$ holds at a certain vertex is equivalent to this inequality holds at all vertices, also equivalent to the all isogonal geodesics at vertices intersect at the center of the midsphere of $\Delta^3(r)$.    
\end{remark}

\begin{corollary}\label{omega}
$\Omega_m^{\mathbb{H}}\subset\Omega^{\mathbb{H}}_t$ is a non-empty, open and proper subset.
\end{corollary}

\pf
For any $r=s(1,1,1,1)\in\Omega^{\mathbb{H}}_e$, since $\Delta^3(r)$ is a regular hyperbolic tetrahedron, by symmetry of $\Delta^3(r)$ and hyperbolic trigonometry, we have
\begin{equation}\label{tri1}
\cos\Theta=\frac{\cosh^22s-\cosh2s}{\sinh^22s}=1-\frac{1}{2\cosh^2s},~\Theta\in(0,\pi/2),
\end{equation}
\begin{equation}\label{tri2}
\langle u_{ij},u_{ik}\rangle_{3,1}=\langle u_{ij},u_{il}\rangle_{3,1}=\langle u_{ik},u_{il}\rangle_{3,1}=\cos\Theta,~ 1\le i\le 4,
\end{equation}
where $\{1,2,3,4\}=\{i,j,k,l\}$ and $u_{ij}$$(u_{ik}, u_{il}, \text{resp}.)$ is the unit tangent vector of the oriented geodesic edge from $v_i$ to $v_j$$(v_k, v_l, \text{resp}.)$ at the vertex $v_i$. Moreover, by symmetry and (\ref{tri2}), we have
\begin{equation}\label{tri4}
    u_i=\frac{1}{\sqrt{3+6\cos\Theta}}\cdot( u_{ij}+u_{ik}+u_{il}),~1\le i\le 4,
\end{equation}
where $u_i$ the isogonal unit tangent vector at the vertex $v_i$ of $\Delta^3(r)$. 

By symmetry, the isogonal angle at the four vertices of the regular hyperbolic tetrahedron $P(r)$ are all equal, denoted by $\theta$. Then for any $1\le i\le4$, by Lemma \ref{angle}, (\ref{tri1}), (\ref{tri2}) and (\ref{tri4}), we have
$$
 \langle u_i,u_{ij}\rangle_{3,1}=\langle u_i,u_{ik}\rangle_{3,1}=\langle u_i,u_{il}\rangle_{3,1}=\cos\theta=\sqrt{1+2\cos\Theta}/\sqrt{3},
$$
$$
\sin^2\theta=1-\cos^2\theta=\frac{2}{3}\cdot(1-\cos\Theta)=\frac{2}{3}\cdot \frac{1}{2\cosh^2s}=\frac{1}{3\cosh^2s},
$$
which implies 
$$
\sin\theta=\frac{1}{\sqrt{3}\cosh s}<\frac{1}{\cosh s }.
$$
Then by Lemma \ref{tangentsphere1}, $r\in\Omega_{m}^{\mathbb{H}}$ and $\Omega_{m}^{\mathbb{H}}$ is non-empty.

By Lemma \ref{angle} and Lemma \ref{tangentsphere1}, the isogonal angle $\theta_i=\theta_i(r)$($1\le i\le 4$) of the conformal tetrahedron $\Delta^3(r)$ depends continuously on $r\in\Omega^{\mathbb{H}}_t$ and
$$
\Omega_m^{\mathbb{H}}=\{r\in \Omega^{\mathbb{H}}:~\sin\theta_i(r)<1/\cosh r_i,~1\le i\le 4\},
$$
which implies that $\Omega_m^{\mathbb{H}}$ is open.

Given a positive number $\lambda\in(\frac{1}{2}\ln\frac{4}{3},\frac{1}{2}\ln\frac{1+\sqrt{3}}{2})$, for any $\mu>0$, we have
$$
3\coth\mu-2\sqrt{3\coth^2\mu+1}<1<\coth\lambda<7<3\coth\mu+2\sqrt{3\coth^2\mu+1},
$$
which implies $Q^{\mathbb{H}}(r)>0$ and $r\in\Omega^{\mathbb{H}}$, where $r=(\lambda,\mu,\mu,\mu)$.

Define $f(\mu):=\sinh\mu/\sinh(\lambda+\mu)$, then $f$ is a strictly increasing function on $(0,\infty)$. Since $\lambda>\frac{1}{2}\ln\frac{4}{3}$, then we have 
$$
0<f(\mu)<\lim_{\mu\to\infty}f(\mu)=e^{-\lambda}<\sqrt{3}/2
$$
for all $\mu\in(0,\infty)$, which implies that $3+6\cos\Theta_1=3(1+2\cos\Theta_1)=3(3-4f^2(\mu))>0$.

By symmetry of $\Delta^3(r)$ and hyperbolic trigonometry, similar to above analysis of the regular hyperbolic tetrahedron, we have
\begin{equation}\label{trj1}
\cos\Theta_1=\frac{\cosh^2(\lambda+\mu)-\cosh2\mu}{\sinh^2(\lambda+\mu)}=1-\frac{2\sinh^2\mu}{\sinh^2(\lambda+\mu)},
\end{equation}
\begin{equation}\label{trj2}
\langle u_{12},u_{13}\rangle_{3,1}=\langle u_{12},u_{14}\rangle_{3,1}=\langle u_{13},u_{14}\rangle_{3,1}=\cos\Theta_1,
\end{equation}
\begin{equation}\label{trj3}
 u_1=\frac{1}{\sqrt{3+6\cos\Theta_1}}\cdot( u_{12}+u_{13}+u_{14}),
\end{equation}
\begin{equation}\label{trj4}
 \langle u_1,u_{12}\rangle_{3,1}=\langle u_1,u_{13}\rangle_{3,1}=\langle u_1,u_{14}\rangle_{3,1}=\cos\theta_1=\sqrt{1+2\cos\Theta_1}/\sqrt{3},
\end{equation}
where $\Theta_1=\angle v_2v_1v_3=\angle v_2v_1v_4=\angle v_3v_1v_4$ and $\theta_1$ is the isogonal angle at the vertex $v_1$. Moreover, since $\lambda<\frac{1}{2}\ln\frac{1+\sqrt{3}}{2}$, by (\ref{trj1}) and (\ref{trj4}), we have
\begin{equation}\label{mu}
\begin{aligned}
\sin\theta_1&=\sqrt{1-\cos^2\theta_1}\\
&=\sqrt{\frac{2}{3}}\cdot\sqrt{1-\cos\Theta_1}\\
&=\frac{2}{\sqrt{3}}\cdot \frac{\sinh\mu}{\sinh(\lambda+\mu)}\\
&\rightarrow\frac{2}{\sqrt{3}}\cdot e^{-\lambda}>1/\cosh\lambda~(\mu\rightarrow\infty),
\end{aligned}
\end{equation}
which implies that there exists a sufficiently large $\mu$ such that $\sin\theta_1>1/\cosh\lambda$. Then by Lemma \ref{tangentsphere1}, we have $r=(\lambda,\mu,\mu,\mu)\notin\Omega_m^{\mathbb{H}}$, which implies $r\in\Omega^{\mathbb{H}}_t\setminus\Omega_m^{\mathbb{H}}$ and $\Omega_m^{\mathbb{H}}$ is a proper subset.
\pfd

By Proposition \ref{gong}, Lemma \ref{angle} and the argument similar to the proofs of Lemma \ref{mid} and Lemma \ref{tangentsphere1}, we have the following result that shows the equivalence between the conformality of a Euclidean (spherical, resp.) tetrahedron and the existence of the midsphere. It is essentially obtained by Cooper and Rivin \cite{Cooper96}.

\lm[\cite{Cooper96}, Lemma 4.1]\label{euclidtangentsphere}
A non-degenerate Euclidean $($spherical, resp.$)$ tetrahedron $\Delta^3$ is conformal if and only if $\Delta^3$ has a midsphere.
\lmd

By Lemma \ref{euclidtangentsphere}, we immediately have the following result of admissible space $\Omega^{\mathbb{E}}$ ($\Omega^{\mathbb{S}}$, resp.) of the conformal Euclidean (spherical, resp.) tetrahedron.

\begin{corollary}\label{reg}
   $\Omega^{\mathbb{E}}_t=\Omega_m^{\mathbb{E}}$ $( \Omega^{\mathbb{S}}_t=\Omega_m^{\mathbb{S}}, \text{resp}.)$. 
\end{corollary}

\begin{remark}\label{regh}
 By Lemma \ref{euclidtangentsphere} and the argument similar to the proof of Lemma \ref{tangentsphere1}, given a conformal Euclidean $($spherical, resp.$)$ tetrahedron $\Delta^3(r)$ parameterized by $r\in\Omega^{\mathbb{E}}_t (\Omega^{\mathbb{S}}_t, \text{resp}.)$, then all isogonal geodesics at vertices intersect at the center of the midsphere of $\Delta^3(r)$.  \end{remark}

\rmk\label{reop}
In the proof of Corollary \ref{omega}, we have constructed a counterexample of a conformal hyperbolic tetrahedron that does not have the midsphere, which implies that the Note below the Theorem 4.2 in \cite{Cooper96} is incorrect and Lemma \ref{euclidtangentsphere} does not hold in the hyperbolic case. It is quite different from Euclidean and spherical cases, which shows a surprising geometric property of conformal hyperbolic tetrahedra. 
\rmkd

We denote the set consisting of all $r\in\Omega_m^{\mathbb{H}}$ ($\Omega_m^{\mathbb{E}}$, $\Omega_m^{\mathbb{S}}$, resp.) such that the center of the midsphere of the hyperbolic (Euclidean, spherical, resp.) tetrahedron $\Delta^3(r)$ lies in the interior of $\Delta^3(r)$ by $\Omega_c^{\mathbb{H}}$ ($\Omega_c^{\mathbb{E}}$, $\Omega_c^{\mathbb{S}}$, resp.), namely
\begin{align}\nonumber
\Omega_c^{\mathbb{H}}(\Omega_c^{\mathbb{E}}, \Omega_c^{\mathbb{S}}, \text{resp}.):=\{&r=(r_1,\cdots,r_4)\in \Omega_m^{\mathbb{H}} (\Omega_m^{\mathbb{E}}, \Omega_m^{\mathbb{S}}, \text{resp}.):
\text{the center of }
\\\nonumber& \text{the midsphere of the tetrahedron $\Delta^3(r)$ parameterized }\\
&\text{by $r$ lies in the interior of $\Delta^3(r)$}\}.\nonumber
\end{align}

\begin{lemma}\label{center}
 $\Omega_c^{\mathbb{H}}(\Omega_c^{\mathbb{E}}, \Omega_c^{\mathbb{S}}, \text{resp}.)\subset\Omega^{\mathbb{H}}_t(\Omega^{\mathbb{E}}_t, \Omega^{\mathbb{S}}_t, \text{resp}.)$  is a non-empty and open subset.
\end{lemma}

\begin{proof}
By the proof of Corollary \ref{omega}, Corollary \ref{reg} and symmetry of the regular tetrahedron, we have
$\Omega_e^{\mathbb{H}}(\Omega_e^{\mathbb{E}},\Omega_e^{\mathbb{S}},\text{resp}.)\subset\Omega_m^{\mathbb{H}}(\Omega_m^{\mathbb{E}},\Omega_m^{\mathbb{S}},\text{resp}.)$ and the center of the midsphere of the regular tetrahedron $P(r)$ lies in its interior for all $r\in\Omega_e^{\mathbb{H}}(\Omega_e^{\mathbb{E}},\Omega_e^{\mathbb{S}},\text{resp}.)$, which implies that $\Omega_e^{\mathbb{H}}(\Omega_e^{\mathbb{E}},\Omega_e^{\mathbb{S}})\subset\Omega_c^{\mathbb{H}}(\Omega_c^{\mathbb{E}},\Omega_c^{\mathbb{S}},\text{resp}.)$ and $\Omega_c^{\mathbb{H}}(\Omega_c^{\mathbb{E}},\Omega_c^{\mathbb{S}},\text{resp}.)$ is non-empty.

For any $r\in \Omega_m^{\mathbb{H}}(\Omega_m^{\mathbb{E}},\Omega_m^{\mathbb{S}},\text{resp}.)$, by Lemma \ref{tangentsphere1}, Corollary \ref{omega}, Remark \ref{regh} and Remark \ref{reop}, the center of the midsphere of the tetrahedron $\Delta^3(r)$ lies in the intersection of isogonal geodesics at vertices, which implies that the center of the midsphere of $\Delta^3(r)$ depends continuously on $r\in \Omega_m^{\mathbb{H}}(\Omega_m^{\mathbb{E}},\Omega_m^{\mathbb{S}},\text{resp}.)$. By definition, Corollary \ref{omega} and Corollary \ref{reg},  $\Omega_c^{\mathbb{H}}(\Omega_c^{\mathbb{E}}, \Omega_c^{\mathbb{S}}, \text{resp}.)\subset\Omega^{\mathbb{H}}_t(\Omega^{\mathbb{E}}_t, \Omega^{\mathbb{S}}_t, \text{resp}.)$ is an open subset.       
\end{proof}

Before giving the proof of Theorem \ref{maintheorem3}, we establish the following result, which is attributed to Cooper and Rivin \cite{Cooper96}. 

\begin{prop}\label{equalangle}
Given three vertices $v_1,v_2,v_3\in\mathbb{H}^3(\mathbb{S}^3,\text{resp}.)$ and a curve $\gamma:(a,b)\to\mathbb{H}^3(\mathbb{S}^3,\text{resp}.)$, if for any time $s\in(a,b)$, the hyperbolic $($spherical, resp.$)$ tetrahedron $\Delta^3(s)$ with four vertices $v_1,v_2,v_3,\gamma(s)$ is conformal, then the angles between the tangent vector $\gamma'(s)$ and the tangent vectors of the three edges of $\Delta^3(s)$ at the vertex $\gamma(s)$  are equal for all time $s\in(a,b)$.      
\end{prop}

\pf
Since the hyperbolic (spherical, resp.) tetrahedron $\Delta^3(t)$ is conformal for all time $s\in(a,b)$, then there exists a function $r:(a,b)\to\mathbb{R}_+$ and three positive numbers $r_1,r_2,r_3>0$, such that 
\begin{equation}\label{clp}
  \ell_{v_i,\gamma(t)}=r_i+r(s),~1\le i\le3,  
\end{equation}
for all time $s\in(a,b)$, where $\ell_{v_i,\gamma(s)}$ is the length of  the edge $e_{v_i,\gamma(s)}$  connecting $v_i$ and $\gamma(s)$. Moreover, by (\ref{juli}) ((\ref{julis}), resp.), we have
\begin{equation}\label{ahg}
\langle v_i,\gamma(s)\rangle_{3,1}=-\cosh\ell_{v_i,\gamma(s)}(\langle v_i,\gamma(s)\rangle=\cos\ell_{v_i,\gamma(t)},\text{resp}.) 
\end{equation}
for all $1\le i\le3$ and time $t\in(a,b)$. Moreover, by (\ref{clp}) and (\ref{ahg}), we have
\begin{equation}\label{ah}
\begin{aligned}
\langle v_i,\gamma'(s)\rangle_{3,1}&=-\sinh\ell_{v_i,\gamma(s)}\cdot r'(s)\\
(\langle v_i,\gamma'(s)\rangle&=-\sin\ell_{v_i,\gamma(s)}\cdot r'(s),\text{resp}.),~1\le i\le3 
\end{aligned}
\end{equation}
for all time $s\in(a,b)$. 

The supporting geodesic of the oriented geodesic edge $\vec{e}_{\gamma(s),v_i}$ of $\Delta^3(s)$ from the vertex $\gamma(s)$ to the vertex $v_i$($1\le i\le3$) can be parameterized as  
\begin{equation}\label{geode}
\begin{aligned}
    \gamma_i(l)&=\cosh l\cdot\gamma(s)+\sinh l\cdot u_i(t)\\
    (&=\cos l\cdot\gamma(s)+\sin l\cdot u_i(s),\text{resp}.),~l\in(-\infty,+\infty)
\end{aligned}
\end{equation}
for all time $s\in (a,b)$, where $u_i(s)$ is the unit tangent vector of the oriented geodesic edge $\vec{e}_{\gamma(s),v_i}$ at the vertex $\gamma(s)$. Moreover, by (\ref{geode}), we have
   \begin{equation}\label{geod}
\begin{aligned}
   v_i&=\cosh\ell_{v_i,\gamma(s)} \cdot\gamma(s)+\sinh \ell_{v_i,\gamma(s)}\cdot u_i(s)\\
    (&=\cos \ell_{v_i,\gamma(s)}\cdot\gamma(s)+\sin \ell_{v_i,\gamma(s)}\cdot u_i(s),\text{resp}.),
\end{aligned}
\end{equation}
\begin{equation}\label{aa5}
\begin{aligned}
   \langle v_i,\gamma'(s)\rangle_{3,1}&=\sinh \ell_{v_i,\gamma(s)}\langle u_i(s),\gamma'(s)\rangle_{3,1}\\
    (\langle v_i,\gamma'(s)\rangle&=\sin \ell_{v_i,\gamma(s)}\langle u_i(s),\gamma'(s)\rangle,\text{resp}.)
\end{aligned}
\end{equation}
for all $s\in(a,b)$. Then by (\ref{ah}) and (\ref{aa5}), we have
$$
\langle u_i(s),\gamma'(s)\rangle_{3,1}(\langle u_i(s),\gamma'(s)\rangle,\text{resp}.)=-r'(s),~1\le i\le3
$$
for all time $s\in(a,b)$, which implies that the angles between the tangent vector $\gamma'(s)$ and the tangent vectors of the three geodesic edges of $\Delta^3(s)$ at the vertex $\gamma(s)$ are equal for all time $s\in(a,b)$.  
\pfd

\begin{corollary}\label{iop}
    For any tetrahedron $t\in T$ and $r_t=(r_i)_{i\in V_t}\in\Omega^{\mathbb{H}}_e (\Omega^{\mathbb{S}}_e,\text{resp.})$, we have the following inequalities
    \begin{equation}\label{aom}   
    \begin{cases}
    \left.\dfrac{\partial \alpha_i^t}{\partial r_i}\right|_{r_t}\le0,\left.\dfrac{\partial \mathrm{Vol}_t}{\partial r_i}\right|_{r_t}\ge 0,& \text{for all}~i\in V_t,\\
    \\
  \left.\dfrac{\partial \alpha_i^t}{\partial r_j}\right|_{r_t}\ge0,&  \text{for all}~i\ne j\in V_t,\\
\end{cases}
    \end{equation}    
    where $\alpha_i^t$ is the solid angle at the vertex $v_i$ of the conformal hyperbolic $($spherical, resp.$)$ tetrahedron $\Delta^3(r_t)$ and $\mathrm{Vol}_t$ is the hyperbolic $($spherical, resp.$)$ volume of $\Delta^3(r_t)$.
\end{corollary}

\begin{proof}
For any tetrahedron $t\in T$, $r_t=(r_i)_{i\in V_t}\in\Omega^{\mathbb{H}}_e (\Omega^{\mathbb{S}}_e,\text{resp.})$ and sufficiently small $\epsilon>0$, since $\Delta^3(r_t)$ is a regular hyperbolic (spherical, resp.) tetrahedron, by Proposition \ref{equalangle}, $\Delta^3(r_t-\boldsymbol{\epsilon}_i)$ is strictly contained in $\Delta^3(r_t)$ for all $i\in V_t$, where $\Delta^3(r_t-\boldsymbol{\epsilon}_i)$ and $\Delta^3(r_t)$ are considered as tetrahedra with the same vertices $\{v_j\}_{j\in V_t\setminus\{i\}}$, and $\boldsymbol{\epsilon}_i=(0,\epsilon(\text{i-th}),0,0)$. By this containment, we have
$$
\begin{cases}
\alpha_i^t(r_t-\boldsymbol{\epsilon}_i)>\alpha_i^t(r_t),& \text{for all}~i\in V_t,\\
 \alpha_i^t(r_t-\boldsymbol{\epsilon}_j)<\alpha_i^t(r_t),& \text{for all}~i\ne j\in V_t,\\
  \mathrm{Vol}_t(r_t-\boldsymbol{\epsilon}_i)<\mathrm{Vol}_t(r_t),& \text{for all}~i\in V_t,
\end{cases}
$$
which implies (\ref{aom}). 
\end{proof}

\begin{remark}
 Corollary \ref{iop} does not hold on $\Omega^{\mathbb{H}}_t (\Omega^{\mathbb{S}}_t,\text{resp.})$, since $\Delta^3(r_t-\epsilon_i)$ is not strictly contained in $\Delta^3(r_t)$ for some $r_t$, which implies that the proof of Lemma 3.2 in \cite{Cooper96} is incorrect.     
\end{remark}

\tm\label{maintheorem3}
The signature of the Jacobian matrix $\left(\dfrac{\partial \alpha_i^t}{\partial r_j}\right)_{i,j\in V_t}$ is $(0,0,4)$ $((1,0,3), \text{resp}.)$ on $ \Omega^{\mathbb{H}}_t(\Omega^{\mathbb{S}}_t, \text{resp}.)$ for all tetrahedron $t\in T$. 
\tmd

\pf

For any tetrahedron $t\in T$, by the proof of Corollary \ref{rigidityoftetrahedron}, the solid angle $\alpha^t=(\alpha_i^t)_{i\in V_t}$ is a real analytic function of $r_t=(r_i)_{i\in V_t}$ and vice versa.   
 By the chain rule for partial derivatives, the Jacobian matrix $\left(\dfrac{\partial \alpha_i^t}{\partial r_j}\right)_{i,j\in V_t}$ is non-degenerate on $\Omega^{\mathbb{H}}_t(\Omega^{\mathbb{S}}_t, \text{resp}.)$.

For any $r_t\in\Omega^{\mathbb{H}}_e(\Omega^{\mathbb{S}}_e, \text{resp}.)$, by symmetry of the regular tetrahedron and Corollary \ref{iop}, we have
\begin{equation}\label{enig}
\begin{cases}
\lambda^t(r_t):=\left.\dfrac{\partial \alpha_i^t}{\partial r_i}\right|_{r_t}\le0,& \text{for all}~i\in V_t,\\
  \\
  \mu^t(r_t):=\left.\dfrac{\partial \alpha_i^t}{\partial r_j}\right|_{r_t}\ge0,& \text{for all}~i\ne j\in V_t.
\end{cases}
\end{equation}

The eigenvalues of the matrix $\left(\left.\dfrac{\partial \alpha_i^t}{\partial r_j}\right|_{r_t}\right)_{i,j\in V_t}$ are $\lambda^t(r_t)+3\mu^t(r_t)$(multiplicity one) and $\lambda^t(r_t)-\mu^t(r_t)$(multiplicity three) on $\Omega^{\mathbb{H}}_e (\Omega^{\mathbb{S}}_e,\text{resp.})$. 

By Schl$\ddot{\text{a}}$fli formula, we have
\begin{equation}\label{fja}
    \sum_{i\in V_t}r_i\dfrac{\partial \alpha_i^t}{\partial r_j}=2\kappa\dfrac{\partial \mathrm{Vol}_t}{\partial r_j},~r_t\in\Omega^{\mathbb{H}}_t (\Omega^{\mathbb{S}}_t,\text{resp.})
\end{equation}
for all $j\in V_t$, where $\kappa=-1(\kappa=1,\text{resp}.)$. Since the matrix $\left(\dfrac{\partial \alpha_i^t}{\partial r_j}\right)_{i,j\in V_t}$ is non-degenerate on the simply connected domain $\Omega^{\mathbb{H}}_t(\Omega^{\mathbb{S}}_t, \text{resp}.)$, then by Corollary \ref{iop}, (\ref{enig}) and (\ref{fja}), we have
\begin{equation}\label{degen}
\lambda^t(r_t)+3\mu^t(r_t)<0(>0,\text{resp}.),~\lambda^t(r_t)-\mu^t(r_t)<0,
\end{equation}
for all $r_t\in\Omega^{\mathbb{H}}_e (\Omega^{\mathbb{S}}_e,\text{resp.})$, 
which implies that the signature of the matrix $\left(\dfrac{\partial \alpha_i^t}{\partial r_j}\right)_{i,j\in V_t}$ is $(0,0,4)$($(1,0,3)$, resp.) on $\Omega^{\mathbb{H}}_e (\Omega^{\mathbb{S}}_e,\text{resp})$ and $\Omega^{\mathbb{H}}_t (\Omega^{\mathbb{S}}_t,\text{resp})$.
\pfd

\section{Dual characterization of conformal tetrahedra in constant curvature spaces}\label{sdual}

In this section, we prove Theorem \ref{main1}, which shows the conformal equivalence of hyperbolic (Euclidean, spherical, ideal hyperbolic, resp.) tetrahedra and dual tetrahedra. As some consequences, we also give the proof of Corollary \ref{sumangle}, Corollary \ref{ck1} and Corollary \ref{rigidityoftetrahedron}.

\begin{lemma}\label{dihesum}
Given a non-degenerate hyperbolic $($Euclidean, spherical, ideal hyperbolic, resp.$)$ tetrahedron $\Delta^3\subset\mathbb{H}^3(\mathbb{E}^3, \mathbb{S}^3, \overline{\mathbb{H}^3}, \text{resp}.)$, if $\Delta^3$ is conformal, then the three pairs of the sums of opposite dihedral angles of $\Delta^3$ are equal.
\end{lemma}

\begin{proof}
(Hyperbolic, Euclidean and spherical case.) For any $r=(r_1,r_2,r_3,r_4)$ $\in\Omega_c^{\mathbb{H}}(\Omega_c^{\mathbb{E}},\Omega_c^{\mathbb{S}},\text{resp.})$ and the conformal tetrahedron $\Delta^3(r)$ parameterized by $r$, choose the point $v_{ij}\in e_{ij}$ such that $d_{-1}(d_{0},d_{1},\text{resp}.)(v_i,v_{ij})=r_i$ for all $1\le i\ne j\le4$, note that $e_{ij}=e_{ji}$ and $v_{ij}=v_{ji}$. 

Suppose that $\Delta^3(r)$ has the midsphere $S$ with the center $o$, then add the geodesic segment $e_{o,ij}$ connecting $o$ and $v_{ij}$ for all $1\le i\ne j\le4$. Recall that the $i$-face $F_i$($1\le i\le 4$) of $P$ is the opposite face of the vertex $v_i$.

Project the center $o$ perpendicularly onto the face $F_i$ ($1\le i\le4$) and denote the projection point on the face $F_i$ by $o_i$. Add the geodesic segment $e_{o_i,jk}$ ($e_{o_i,jl}$, $e_{o_i,kl}$, resp.) connecting $o_i$ and $v_{jk}$($v_{jl}$, $v_{kl}$, resp.) for all $1\le i\le 4$, where $(1,2,3,4)=(i,j,k,l)$. 

By the proofs of Lemma \ref{mid} and Lemma \ref{tangentsphere1}, for any $1\le i\le 4$, $o_i$ is the incenter of the face $F_i$ and 
\begin{equation}\label{edge}
d_{\kappa}(o_i,v_{jk})=d_{\kappa}(o_i,v_{jl})=d_{\kappa}(o_i,v_{kl}),~1\le i\le 4,
\end{equation}
where $\kappa=-1 (\kappa=0,\kappa=1,\text{resp}.)$. Moreover, we have
\begin{equation}\label{dihe}
e_{kl}\perp e_{o,kl},~e_{kl}\perp e_{o_i,kl},~e_{kl}\perp e_{o_j,kl},~1\le k\ne l\le4,
\end{equation}
which implies that $o,o_i,o_j,v_{kl}\in H_{kl}$, where $H_{kl}$ is the plane perpendicular to edge $e_{kl}$ at $v_{kl}$.     

Add the geodesic segment $e_{o,o_i}$ connecting $o$ and $o_i$ for all $1\le i\ne j\le4$, by (\ref{edge}), we have $\de oo_iv_{jk}$, $\de oo_iv_{jl}$ and $\de oo_iv_{kl}$ are pairwise congruent and 
\begin{equation}\label{angl}
 \angle ov_{jk}o_i=\angle ov_{kl}o_i=\angle ov_{jl}o_i:=\theta_i
\end{equation}
 for all $1\le i\le4$. Note that the center $o$ lies in the interior of $\Delta^3(r)$, since $r\in\Omega_c^{\mathbb{H}}(\Omega_c^{\mathbb{E}},\Omega_c^{\mathbb{S}},\text{resp.})$. Then by (\ref{dihe}), (\ref{angl}) and the definition of the dihedral angle, we have
\begin{equation}\label{agl}
\theta_{ij}=\angle o_i v_{kl}o_j=\angle ov_{kl}o_i+\angle ov_{kl}o_j=\theta_i+\theta_j
\end{equation}
for all edge $e_{kl}$ of $\Delta^3(r)$, where $\theta_{ij}\in(0,\pi)$ is the dihedral angle of the edge $e_{kl}$. Moreover, by Lemma \ref{center} and (\ref{agl}), we have
\begin{equation}\label{som}
\theta_{ij}+\theta_{kl}=\theta_{ik}+\theta_{jl}=\theta_{il}+\theta_{jk},
\end{equation}
which holds on the open subset $\Omega_c^{\mathbb{H}}(\Omega_c^{\mathbb{E}},\Omega_c^{\mathbb{S}},\text{resp.})\subset\Omega^{\mathbb{H}}_t(\Omega^{\mathbb{E}}_t,\Omega^{\mathbb{S}}_t,\text{resp.})$.

By Lemma \ref{lengthandangle}, Lemma \ref{hyperbolicadmissiblespace} and Lemma \ref{sphericaladmissiblespace}, all dihedral angles of a conformal hyperbolic tetrahedron are the real analytic functions on the connected domain $\Omega^{\mathbb{H}}_t(\Omega^{\mathbb{E}}_t,\Omega^{\mathbb{S}}_t,\text{resp.})$. 

By the property of real analytic functions, the equality (\ref{som}) holds on the admissible space $\Omega^{\mathbb{H}}_t(\Omega^{\mathbb{E}}_t,\Omega^{\mathbb{S}}_t,\text{resp.})$, which implies that the three pairs of the sums of opposite dihedral angles of $\Delta^3$ are equal for all conformal hyperbolic (Euclidean, spherical, resp.) tetrahedra $\Delta^3$.

(Ideal hyperbolic case.) If the ideal hyperbolic tetrahedron $\Delta^3\subset\overline{\mathbb{H}^3}$ is conformal, then by definition, $\Delta^3$ is the Gromov-Hausdorff limit of a sequence of conformal hyperbolic tetrahedra $\{\Delta^3_m\}_{m=1}^{\infty}\subset\mathbb{H}^3$, which implies that 
\begin{equation}\label{xy2}
\theta_{ij}^{(m)}\rightarrow\theta_{ij}(m\rightarrow\infty),~\text{for all}~1\le i\ne j\le 4.
\end{equation}

Since $\Delta^3_m$ is a conformal hyperbolic tetrahedron, then by the proof of hyperbolic case, we have
\begin{equation}\label{xy1}
\theta_{ij}^{(m)}+\theta_{kl}^{(m)}=\theta_{ik}^{(m)}+\theta_{jl}^{(m)}=\theta_{il}^{(m)}+\theta_{jk}^{(m)}
\end{equation}
for all permutations $(i,j,k,l)$ of $(1,2,3,4)$. Moreover, by (\ref{xy2}) and (\ref{xy1}), we have
$$
\theta_{ij}+\theta_{kl}=\theta_{ik}+\theta_{jl}=\theta_{il}+\theta_{jk}.
$$
This complete the proof.
\end{proof}

Now we give the proof of Theorem \ref{main1}.
\pf[Proof of Theorem \ref{main1}]

(Hyperbolic case.) ($\Rightarrow$) If the hyperbolic tetrahedron $\Delta^3\subset\mathbb{H}^3$ is conformal, then by Lemma \ref{dihesum}, we have
\begin{equation}\label{so}
\theta_{ij}+\theta_{kl}=\theta_{ik}+\theta_{jl}=\theta_{il}+\theta_{jk}
\end{equation}
for all permutations $(i,j,k,l)$ of $(1,2,3,4)$ and $\theta_{ij}\in(0,\pi)$ is the dihedral angle of the geodesic edge $e_{kl}$ of $\Delta^3$ for all $1\le k\ne l\le 4$. Moreover, by Corollary \ref{conedge} and (\ref{so}), we have
\begin{equation}\label{duedge}
\ell_{ij}^*+\ell_{kl}^*=\ell_{ik}^*+\ell_{jl}^*=\ell_{il}^*+\ell_{jk}^*,
\end{equation}
where $\ell _{ij}^*=\pi-\theta_{ij}\in (0,\pi)$ is the length of the corresponding geodesic edge $e_{ij}^*$ of the dual tetrahedron $(\Delta^3)^\ast$ for all permutations $(i,j,k,l)$ of $(1,2,3,4)$. Moreover, by (\ref{duedge}) and triangle inequality, we have
$$
\frac{\ell_{ij}^*+\ell_{ik}^*-\ell_{jk}^*}{2}=\frac{\ell_{ij}^*+\ell _{il }^*-\ell_{jl}^*}{2}=\frac{\ell_{ik}^*+\ell _{il }^*-\ell_{kl}^*}{2}:=r_i^*>0
$$
for all $1\le i\le 4$, which implies that $\ell_{ij}^*=r_i^*+r_j^*$ for all edges $e_{ij}^*$ of $(\Delta^3)^*$ and the dual tetrahedron $(\Delta^3)^*$ is conformal.

($\Leftarrow$) If the dual tetrahedron $(\Delta^3)^*\subset\mathrm{d}\mathbb{S}^3$ is conformal, then by Corollary \ref{conedge}, there exist positive numbers $r_1^*,r_2^*,r_3^*,r_4^*>0$ such that 
\begin{equation}\label{dui56}
\ell_{ij}^*=\pi-\theta_{ij}=r_i^*+r_j^*\in(0,\pi),
\end{equation}
for all geodesic edges $e_{ij}^*$ of $(\Delta^3)^*$, note that $r_i^*\in(0,\pi)$ for all $1\le i\le4$. 

Recall that the Gram matrix $G^*=(\langle v_i^*,v_j^*\rangle_{3,1})_{1\le i,j\le 4}$, by (\ref{c4}), (\ref{dui56}) and $(\Delta^3)^*\subset\mathrm{d}\mathbb{S}^3$, we have
\begin{equation}\label{highhyper}
\langle v_i^*,v_j^*\rangle_{3,1}=
\begin{cases}
1=s_i^{*2}(c_i^{*2}+1),&1\le i=j\le 4\\
\cos(r_i^*+r_j^*)=s_i^*s_j^*(c_i^*c_j^*-1),&1\le i\ne j\le 4,\\
\end{cases}
\end{equation}
where $s_i^*=\sin r_i^*>0,c_i^*=\cot r_i^*$ for all $1\le i\le 4$. 

Recall that $G^\ast_{ij}$ is the $(i,j)$-cofactor of $G^{*}$ for all $1\le i,j\le 4$, then by the argument similar to the proof of spherical case of Theorem \ref{highpolyhedron}, we have 
\begin{equation}\label{angleminor1}
G^\ast_{ii}=4s_j^{*2}s_k^{*2}s_l^{*2}(c_j^*c_k^*+c_j^*c_l^*+c_k^*c_l^*-1),
\end{equation}
\begin{equation}\label{angleminor2}
G^\ast_{ij}=2s_i^*s_j^*s_k^{*2}s_l^{*2}(c_k^{*2}+c_l^{*2}-(c_i^*+c_j^*)(c_k^*+c_l^*)+2)
\end{equation}
for all permutations $(i,j,k,l)$ of $(1,2,3,4)$. Moreover, since $r_i^*\in (0,\pi)$ and $r_i^*+r_j^*\in (0,\pi)$ for all $1\le i\ne j\le 4$, then we have 
\begin{equation}\label{oddd}
c_i^*+c_j^*>0,~1\le i\ne j\le 4. 
\end{equation}
Moreover, by Lemma \ref{m3}, (\ref{angleminor1}), (\ref{angleminor2}) and (\ref{oddd}), we have
\begin{align}
\cosh \ell_{ij}&=\frac{2s_i^*s_j^*s_k^{*2}s_l^{*2}(c_k^{*2}+c_l^{*2}-(c_i^*+c_j^*)(c_k^*+c_l^*)+2)}{\sqrt{G^\ast_{ii}G^\ast_{jj}}},\label{angleminor5}\\
\sinh \ell_{ij}&=\frac{2s_i^*s_j^*s_k^{*2}s_l^{*2}|c_k^*+c_l^*|A^*}{\sqrt{G^\ast_{ii}G^\ast_{jj}}}
=\frac{2s_i^*s_j^*s_k^{*2}s_l^{*2}(c_k^*+c_l^*)A^*}{\sqrt{G^\ast_{ii}G^\ast_{jj}}}\label{angleminor6}
\end{align}
for all permutations $(i,j,k,l)$ of $(1,2,3,4)$, where
$$
A^*=\sqrt{2\sum_{i=1}^4c_i^{*2}-(\sum_{i=1}^4c_i^*)^2+4}.
$$
Moreover, by (\ref{angleminor5}) and (\ref{angleminor6}), we have
\begin{align} 
\sinh (\ell_{ij}+\ell_{kl})&=\frac{A^*(B^*-C^*)D^*}{E^*},\label{angleminor7}
\end{align}
for all permutations $(i,j,k,l)$ of $(1,2,3,4)$,
where 
$$
B^*=\sum_{i=1}^4c_i^*,~C^*=\sum_{i=1}^4\hat{c_i^*}c_j^*c_k^*c_l^*,~D^*=8\prod_{i=1}^{4} s_i^{*3},~E^*=\sqrt{\prod_{i=1}^{4}G^\ast_{ii}}
$$
and $\hat{c_i^*}$ means remove the $c_i^*$ from $\{c_i^*,c_j^*,c_k^*,c_l^*\}=\{c_1^*,c_2^*,c_3^*,c_4^*\}$. 

Since $A^*,B^*,C^*,D^*,E^*$ are invariants, then by (\ref{angleminor7}), we have
$$
\sinh (\ell_{ij}+\ell_{kl})=\sinh (\ell_{ik}+\ell_{kl})=\sinh (\ell_{il}+\ell_{jk}),
$$
which implies that
\begin{equation}\label{duedg1}
\ell_{ij}+\ell_{kl}=\ell_{ik}+\ell_{jl}=\ell_{il}+\ell_{jk}
\end{equation}
for all permutations $(i,j,k,l)$ of $(1,2,3,4)$. Moreover, by (\ref{duedg1}) and triangle inequality, we have
$$
\frac{\ell_{ij}+\ell_{ik}-\ell_{jk}}{2}=\frac{\ell_{ij}+\ell _{il }-\ell_{jl}}{2}=\frac{\ell_{ik}+\ell _{il }-\ell_{kl}}{2}:=r_i>0
$$
for all $1\le i\le 4$, which implies that $\ell_{ij}=r_i+r_j$ for all geodesic edges $e_{ij}$ of $\Delta^3$ and the hyperbolic tetrahedron $\Delta^3$ is conformal.

(Euclidean case.) $(\Rightarrow)$  The proof has been omitted since it is quite similar to the proof of hyperbolic case.

$(\Leftarrow)$ If the dual tetrahedron $(\Delta^3)^*\subset\mathbb{S}^2$ is conformal, then by Corollary \ref{conedge}, there exist positive numbers $r_1^*,r_2^*,r_3^*,r_4^*>0$ such that 
\begin{equation}\label{do}
\ell_{ij}^*=\pi-\theta_{ij}=r_i^*+r_j^*\in(0,\pi),
\end{equation}
for all geodesic edges $e_{ij}^*$ of $(\Delta^3)^*$, note that $r_i^*\in(0,\pi)$ for all $1\le i\le4$. 

Denote the incircle of the $i$-face $F_i$ of $\Delta^3$ by $C_i$ and suppose that $C_i$ has radius $R_i$ and is tangent to the geodesic edge $e_{jk}$ at $v_{i,jk}$ for all $1\le i\le 4$ and all permutations $(i,j,k,l)$ of $(1,2,3,4)$. Moreover, we have
\begin{align*}
d_0(v_i,v_{j,ik})&=d_0(v_i,v_{j,il}):=r_{i,j},\\
d_0(v_i,v_{k,ij})&=d_0(v_i,v_{k,il}):=r_{i,k},\\
d_0(v_i,v_{l,ij})&=d_0(v_i,v_{l,ik}):=r_{i,l}
\end{align*}
for all $1\le i\le4$. Moreover, for any geodesic edge $e_{ij}(i\ne j)$ of $\Delta^3$, we have
\begin{equation}\label{ao7}
\ell_{ij}=r_{i,k}+r_{j,k}=r_{i,l}+r_{j,l},
\end{equation}
\begin{equation}\label{ao4}
\begin{aligned}
&\tan \frac{\angle v_i v_j v_l}{2}=\frac{R_k}{r_{j,k}},~\tan \frac{\angle v_j v_i v_l}{2}=\frac{R_k}{r_{i,k}},\\
&\tan \frac{\angle v_i v_j v_k}{2}=\frac{R_l}{r_{j,l}},~\tan \frac{\angle v_j v_i v_k}{2}=\frac{R_l}{r_{i,l}}.
\end{aligned}
\end{equation}

By Lemma \ref{lengthandangle}, we have
\begin{equation}\label{ao1}
\cos\angle v_iv_jv_k=\dfrac{\cos\ell_{il}^*\cos\ell_{kl}^*-\cos\ell_{ik}^*}{\sin\ell_{il}^*\sin\ell_{kl}^*}
\end{equation}
for all permutations $(i,j,k,l)$ of $(1,2,3,4)$. Then by (\ref{do}) and (\ref{ao1}), we have
\begin{equation}\label{ao2}
 \dfrac{\tan^2 \dfrac{\angle v_i v_j v_l}{2}}{\tan^2 \dfrac{\angle v_i v_j v_k}{2}}=\dfrac{\tan^2 \dfrac{\angle v_j v_i v_l}{2}}{\tan^2 \dfrac{\angle v_j v_i v_k}{2}}=\dfrac{\sin^2 r_k^*}{\sin^2 r_l^*}.
\end{equation}

Since $r_i^*\in(0,\pi)$ for all $1\le i\le4$, by (\ref{ao2}), we have
\begin{equation}\label{ao3}
 \dfrac{\tan \dfrac{\angle v_i v_j v_l}{2}}{\tan \dfrac{\angle v_i v_j v_k}{2}}=\dfrac{\tan \dfrac{\angle v_j v_i v_l}{2}}{\tan \dfrac{\angle v_j v_i v_k}{2}}=\dfrac{\sin r_k^*}{\sin r_l^*}.
\end{equation}
By (\ref{ao4}) and (\ref{ao3}), we have
\begin{equation}\label{ao6}
\dfrac{r_{i,k}}{r_{i,l}}=\dfrac{r_{j,k}}{r_{j,l}}.
\end{equation}

Then by (\ref{ao7}) and (\ref{ao6}), we have $r_{i,k}=r_{i,l}$ and $r_{j,k}=r_{j,l}$, which implies that   
$$
r_{i,j}=r_{i,k}=r_{i,l}:=r_i>0
$$
for all $1\le i\le 4$. Then by (\ref{ao7}), we have $\ell_{ij}=r_i+r_j$ for all geodesic edges $e_{ij}$ of $\Delta^3$, which implies that the Euclidean tetrahedron $\Delta^3$ is conformal.

(Spherical case.) $(\Rightarrow)$  The proof has been omitted since it is quite similar to the proof of hyperbolic case.

$(\Leftarrow)$ If the dual tetrahedron $(\Delta^3)^*\subset\mathbb{S}^3$ is conformal, since $(\Delta^3)^*$ is also a non-degenerate spherical tetrahedron, then by Proposition \ref{l1} and $(\Rightarrow)$, then spherical tetrahedron $\Delta^3=((\Delta^3)^*)^*$ is conformal.  

(Ideal hyperbolic case.) $(\Rightarrow)$ The proof has been omitted since it is quite similar to the proof of hyperbolic case.

$(\Leftarrow)$ If the dual tetrahedron $(\Delta^3)^*\subset\mathrm{d}\mathbb{S}^3$ is conformal, then by Corollary \ref{po5}, there exist positive numbers $r_1^*,r_2^*,r_3^*,r_4^*\in(0,\pi)$ such that 
\begin{equation}\label{dui}
\ell_{ij}^*=\pi-\theta_{ij}=r_i^*+r_j^*\in(0,\pi),
\end{equation}
for all geodesic edges $e_{ij}^*$ of $(\Delta^3)^*$.

Define $\theta_i:=\pi/{2}-r_i^*$ for all $1\le i\le4$, then by (\ref{dui}), we have
\begin{equation}\label{xi6}
\theta_{ij}=\theta_i+\theta_j, ~\text{for all geodesic edge}~ e_{kl}~\text{of}~\Delta^3.
\end{equation}
where $\theta_i\in(-\pi/2,\pi/2)$.

For any ideal vertex $v_i\in\partial_{\infty}\mathbb{H}^3$, we have
\begin{equation}\label{xi7}
\theta_{jk}+\theta_{jl}+\theta_{kl}=\pi.
\end{equation} 
Then by (\ref{xi6}) and (\ref{xi7}), we have
\begin{equation}\label{idea1}
\theta_{j}+\theta_k+\theta_l=\pi/2
\end{equation}
and there exists at most a $s\in\{j,k,l\}$ such that $\theta_s<0$. Moreover, since $\Delta^3$ is the Gromov-Hausdorff limit of a sequence of hyperbolic tetrahedra, then by Theorem \ref{admissibleidealangle} and the proof of hyperbolic case, we have
\begin{equation}\label{xu45}
\Theta_{ii}=G_{ii}^*=4c_j^{2}c_k^{2}c_l^{2}(t_jt_k+t_jt_l+t_kt_l-1)=0,
\end{equation}
where $\Theta=(\cos(\pi-\theta_{ij}))_{1\le i,j\le4}$, $c_s=\cos\theta_{s}>0$ and $t_s=\tan\theta_s$ for all $s\in\{j,k,l\}$.  

 Choose a sequence $\{\theta_{i}^{(m)}\}_{m=1}^{\infty}$ contained in a sufficiently small interval of $\theta_{i}$ such that $\theta_s^{(m)}>\theta_i$, $\theta_{i}^{(m)}\to\theta_{i}(m\to \infty)$ for all $m\ge 1$ and $1\le i\le4$. Then by above argument and (\ref{xu45}), we have $\Theta_{ii}^{(m)}>0$ for all ideal vertex $v_i\in\partial_{\infty}\mathbb{H}^3$. 

Define $\theta_{ii}^{(m)}:=\pi$ and $\theta_{ij}^{(m)}:=\theta_{i}^{(m)}+\theta_{j}^{(m)}\in(0,\pi)$ for all $1\le i,j\le4$ and $m\ge1$. Then we have $\Theta^{(m)}\to\Theta(m\to\infty)$ and 
\begin{equation}\label{po2}
|\Theta^{(m)}|<0,~\Theta^{(m)}(i,i)~\text{is positive definite},~\Theta_{ij}^{(m)}>0
\end{equation}
for all $1\le i,j\le4$ and $m\ge1$. 

By Theorem \ref{admissibleidealangle} and (\ref{po2}), there exists a sequence of hyperbolic tetrahedra $\{\Delta^3_m\}_{m=1}^{\infty}$ such that the dihedral angle of the $ij$-face of $\Delta^3_m$ is $\theta_{ij}^{(m)}$ for all $ij$-faces and $m\ge1$. Moreover, since $\theta_{ij}^{(m)}\to \theta_{ij}(m\to\infty)$, then we have $\Delta^3_m\to \Delta^3(m\to\infty)$ up to isometry.

Define $r_i^{(m)*}:=\pi/2-\theta_{i}^{(m)}>0$ for all $1\le i\le4$ and $m\ge1$. By Corollary \ref{conedge}, we have
$$
\ell_{ij}^{(m)*}=\pi-\theta_{ij}^{(m)}=r_i^{(m)*}+r_j^{(m)*},
$$
for all geodesic edges $e_{kl}^{(m)*}$ of the dual tetrahedron $(\Delta^3_m)^*$ of $\Delta^3_m$, which implies that the dual tetrahedron $(\Delta^3_m)^*$ is conformal for all $m\ge1$. 

By the proof of hyperbolic case, $\{\Delta^3_m\}_{m=1}^{\infty}$ is a sequence of conformal hyperbolic tetrahedra, which implies that the ideal hyperbolic tetrahedron $\Delta^3$ is conformal.  
\pfd

Now we give the proofs of Corollary \ref{sumangle}, Corollary \ref{ck1} and Corollary \ref{rigidityoftetrahedron}.

\pf[Proof of Corollary \ref{sumangle}]
$(\Rightarrow)$ We have proved this direction by Lemma \ref{dihesum}. 

$(\Leftarrow)$ If the three pairs of the sums of opposite dihedral angles of $\Delta^3$ are equal, then by the argument similar to proof of $(\Rightarrow)$ of hyperbolic case in the proof of Theorem \ref{main1}, the dual tetrahedron $(\Delta^3)^*$ is conformal. Then by Theorem \ref{main1}, $\Delta^3$ is conformal.
\pfd

\pf[Proof of Corollary \ref{ck1}]
For any conformal ideal tetrahedron $\Delta^3$, by Theorem \ref{main1}, the dual tetrahedron $(\Delta^3)^*$ is conformal, then by the proof of $(\Rightarrow)$ of hyperbolic case in the proof of Theorem \ref{main1}, there exist $\theta_1,\cdots,\theta_4\in(-\pi/2,\pi/2)$ such that 
\begin{equation}\label{idal1}
\theta_{ij}=\theta_i+\theta_j, ~\text{for all geodesic edge}~ e_{kl}~\text{of}~\Delta^3.
\end{equation}

Since $v_i(1\le i\le4)$ is an ideal vertex, then we have
\begin{equation}\label{idal2}
\theta_{jk}+\theta_{jl}+\theta_{kl}=\pi, ~\text{for all }~1\le i\le 4.
\end{equation}
Moreover, by (\ref{idal1}) and (\ref{idal2}), we have
$$
\theta_{j}+\theta_{k}+\theta_{l}=\pi/2, ~\text{for all}~1\le i\le 4,
$$
and $\theta_{1}=\theta_{2}=\theta_{3}=\theta_{4}=\pi/6$. Moreover, by (\ref{xi6}), we obtain
$$
\theta_{ij}=\pi/3,~\text{for all geodesic edges}~ e_{kl}~\text{of}~\Delta^3,
 $$   
 which implies that $\Delta^3$ is the unique regular ideal tetrahedron up to isometry. 
\pfd

\begin{remark}
    By the proof of Corollary \ref{ck1}, the conformal ideal tetrahedron is the Gromov-Hausdorff limit of a sequence of regular hyperbolic tetrahedra.
\end{remark}

\pf[Proof of Corollary \ref{rigidityoftetrahedron}] 
Since $|T|=1$, then we have $\Omega^{\mathbb{H}}(\Omega^{\mathbb{E}},\Omega^{\mathbb{S}},\partial_{\infty}\Omega^{\mathbb{H}},\text{resp}.)$ $=\Omega_t^{\mathbb{H}}(\Omega^{\mathbb{E}}_t,\Omega^{\mathbb{S}}_t,\partial_{\infty}\Omega^{\mathbb{H}}_t,\text{resp}.)$ and $V=V_t$ is a boundary vertex set. Moreover, define $\widetilde{\Omega}^{\mathbb{E}}$ as the set of the equivalence classes of all $r$ in the admissible space $\Omega^{\mathbb{E}}$, up to scaling. Adopting a common abuse of notation, we do not distinguish a $r\in\Omega^{\mathbb{E}}$ and the equivalence class $[r]\in\widetilde{\Omega}^{\mathbb{E}}$ represented by $r$.

For any hyperbolic (Euclidean, spherical, ideal hyperbolic, resp.) sphere packing (metric) $r=(r_i)_{i\in V}\in\Omega^{\mathbb{H}}(\widetilde{\Omega}^{\mathbb{E}},\Omega^{\mathbb{S}}, \partial_{\infty}\Omega^{\mathbb{H}},$ $\text{resp}.)$, $\Delta^3(r)$ is a conformal hyperbolic (Euclidean, spherical, ideal hyperbolic, resp.) tetrahedron with
\begin{equation}\label{ahj}
    \ell_{ij}=r_i+r_j,~\text{for all geodesic edge}~e_{ij}~\text{of}~ \Delta^3(r).
\end{equation}
 Moreover, (\ref{ahj}) implies that 
\begin{equation}\label{ahj1}
   r_i=\frac{\ell_{ij}+\ell_{ik}-\ell_{jk}}{2}=\frac{\ell_{ij}+\ell _{il }-\ell_{jl}}{2}=\frac{\ell_{ik}+\ell _{il }-\ell_{kl}}{2},~i\in V
.\end{equation}

By Gauss-Bonnet formula, we have
\begin{equation}\label{ango}
  \alpha_{i}=\theta_{jk}+\theta_{jl}+\theta_{kl}-\pi  
\end{equation}
for all $i\in V$, where $\alpha_{i}$ is the solid angle at the vertex $v_i$ of $\Delta^3(r)$.

By Lemma \ref{lengthandangle}, (\ref{ahj}) and (\ref{ango}), the solid angle $\alpha=(\alpha_i)_{i\in V}$ of the conformal tetrahedron is the smooth real analytic function of $r\in\Omega^{\mathbb{H}}(\widetilde{\Omega}^{\mathbb{E}}, \Omega^{\mathbb{S}}, \partial_{\infty}\Omega^{\mathbb{H}}$, $\text{resp}.)$, namely 
\begin{equation}\label{real}
\alpha=\alpha(r),~r\in\Omega^{\mathbb{H}}(\widetilde{\Omega}^{\mathbb{E}},\Omega^{\mathbb{S}}, \partial_{\infty}\Omega^{\mathbb{H}}, \text{resp}.),  
\end{equation}

Since the hyperbolic (Euclidean, spherical, ideal hyperbolic, resp.) tetrahedron $\Delta^3=\Delta^3(r)$ is conformal, then by Theorem \ref{main1}, Corollary \ref{conedge} and Corollary \ref{po5}, the dual tetrahedron $(\Delta^3)^*$ is conformal and there exists a $r^*=(r_i^*)_{i\in V}\in\mathbb{R}_+^{V}$ such that
\begin{equation}\label{jio}
    \ell_{ij}^*=\pi-\theta_{ij}=r_i^*+r_j^*,~\text{for all geodesic edge}~e_{ij}^*~\text{of}~(\Delta^3(r))^*.
\end{equation}  
Then by (\ref{ango}) and (\ref{jio}), we have 
\begin{equation}\label{sumsoli1}
\begin{aligned}
\alpha_i&=\theta_{jk}+\theta_{jl}+\theta_{kl}-\pi\\
&=\pi-(r_j^*+r_k^*)+\pi-(r_j^*+r_l^*)\\
&+\pi-(r_k^*+r_l^*)-\pi\\
&=2\pi-2\sum_{s\ne i}r_s^*
\end{aligned}
\end{equation}
for all $i\in V$. Moreover, by (\ref{sumsoli1}), we obtain
\begin{equation}\label{sumsoli}
\begin{aligned}
\sum_{i\in V}\alpha_i&=\sum_{i\in V}(2\pi-2\sum_{s\ne i}r_s^*)\\
&=8\pi-2\sum_{i\in V}\sum_{s\ne i}r_s^*\\
&=8\pi-6\sum_{i\in V}r_i^*.
\end{aligned}
\end{equation}
By (\ref{sumsoli1}) and (\ref{sumsoli}), we have the following explicit formula of $r_i^*$ in terms of the solid angle $\alpha_i$, namely
\begin{equation}\label{sumsoli3}
r_i^*=\frac{2\pi+2\alpha_i-\sum_{s\ne i}\alpha_s}{6},~\text{for all}~i\in V.
\end{equation}

By Lemma \ref{lengthandangle}, (\ref{ahj1}), (\ref{jio}) and (\ref{sumsoli3}), the sphere packing metric $r$ is the smooth real analytic function of the solid angle $\alpha\in\alpha(\Omega^{\mathbb{H}}_t)(\alpha(\widetilde{\Omega}^{\mathbb{E}}),\alpha(\Omega^{\mathbb{S}}),$ $\alpha(\partial_{\infty}\Omega^{\mathbb{H}}), \text{resp}.)$, namely 
\begin{equation}\label{bil}
    r=r(\alpha),~\alpha\in\alpha(\Omega^{\mathbb{H}})(\alpha(\widetilde{\Omega}^{\mathbb{E}}),\alpha(\Omega^{\mathbb{S}}), \alpha(\partial_{\infty}\Omega^{\mathbb{H}}), \text{resp}.).
\end{equation}

By (\ref{q96}), (\ref{real}) and (\ref{bil}), the combinatorial scalar curvature map $\mathbb{K}$ is an injective map on $\Omega^{\mathbb{H}}(\widetilde{\Omega}^{\mathbb{E}},\Omega^{\mathbb{S}}, \partial_{\infty}\Omega^{\mathbb{H}}, \text{resp}.)$. This completes the proof.
\pfd

\begin{remark}
    By Lemma \ref{lengthandangle}, $($\ref{q96}$)$, $($\ref{ahj1}$)$, $($\ref{jio}$)$ and $($\ref{sumsoli3}$)$, the rigidity map, namely the inverse of combinatorial scalar curvature map $\mathbb{K}$, can be expressed explicitly.    
\end{remark}

\section{Admissible space of conformal ideal hyperbolic tetrahedra}\label{boundary}

In this section, we study the ideal hyperbolic admissible space and construct the extended concave functional. Moreover, we use the symbol ``$\Omega_t$'' instead of $\Omega^{\mathbb{H}}_t$ to represent this hyperbolic admissible space for the convenience of notations.

Define $\widehat{\mathbb{R}_{+}^{V_t}}:=\{r_t=(r_i)_{i\in V_t}:0<r_i\le\infty~ \text{for all}~ i\in V_t\}$, $\coth \infty:=1$ and recall that
$$
Q^{\mathbb{H}}_3(r_t):=(\sum\limits_{i\in V_t}\coth r_i)^2-2\sum\limits_{i\in V_t}\coth^2 r_i+4.
$$

For any subset $W_t\subset V_t$, the space $\widehat{\mathbb{R}_{+}^{V_t}}$ has the subspace 
$$
\mathcal{R}^{W_t}:=\left \{ r_t=(r_i)_{i\in V_t}\in\widehat{\mathbb{R}_{+}^{V_t}}:  r_{W_t}=(r_i)_{i\in W_t}\in\mathbb{R}^{W_t}_{+},~ r_{V_t\setminus W_t}=(r_i)_{i\in V_t\setminus W_t}=(\infty)_{i\in V_t\setminus W_t}\right\},
$$
where $\mathcal{R}^{\varnothing_t}=\{{\boldsymbol{\infty}_t}\}$($\boldsymbol{\infty}_t=(\infty)_{i\in V_t}$) if $W_t=\varnothing_t$($\varnothing_t$ is the empty subset of $V_t$) and $\mathcal{R}^{V_t}=\mathbb{R}^{V_t}_+$ if $W_t=V_t$. 

It is not difficult to check that there exists a cellular decomposition of $\widehat{\mathbb{R}_{+}^{V_t}}$ and the subspace $\mathcal{R}^{W_t}$ is a $|W_t|$-dimensional cell, namely
\begin{equation}
\widehat{\mathbb{R}_{+}^{V_t}}=\bigsqcup_{W_t\subset V_t}\mathcal{R}^{W_t}=\mathbb{R}_{+}^{V_t}\bigsqcup\partial_{\infty}\mathbb{R}_{+}^{V_t},
\end{equation}
where $\partial_{\infty}\mathbb{R}_{+}^{V_t}:=\bigsqcup_{W_t\subsetneq V_t}\mathcal{R}^{W_t}$ is called \emph{boundary at infinity of} $\mathbb{R}_{+}^{V_t}$.

For any non-empty subset $W_t\subset V_t$, define the following projection  
\begin{equation}
		\begin{aligned}
			p^{W_t}: \mathcal{R}^{W_t} & \longrightarrow\mathbb{R}^{W_t}_{+}  \\
			r_t^{W_t}=(r_i)_{i\in V_t} & \longmapsto r_{W_t}=(r_i)_{i\in W_t},
		\end{aligned}
\end{equation}
the projections $p^{W_t}$ is a homeomorphism from $\mathcal{R}^{W_t}$ to $\mathbb{R}_{+}^{W_t}$ and define $r^{W_t}_t:=(p^{W_t})^{-1}(r_t)\in\mathcal{R}^{W_t}$ for all $r_{W_t}\in \mathbb{R}^{W_t}_{+}$.

Recall that the extended hyperbolic admissible space of the hyperbolic admissible space $\Omega_t$ is defined as
$$
\widehat{\Omega}_t:=\left \{ r_t=(r_i)_{i\in V_t}\in\widehat{\mathbb{R}_{+}^{V_t}}:Q^{\mathbb{H}}_3(r_t)>0 \right\}.
$$
For any subset $W_t\subset V_t$, the extended space $\widehat{\Omega}_t$ has the subspace 
$$
\Omega^{W_t}:=\left \{ r_t^{W_t}=(r_i)_{i\in V_t}\in\mathcal{R}^{W_t}:  Q^{\mathbb{H}}_3(r_t^{W_t})>0\right\},
$$
where $\Omega^{\varnothing_t}=\{{\boldsymbol{\infty}}_t\}$($\boldsymbol{\infty}=(\infty)_{i\in V_t}$) if $W_t=\varnothing_t$ and $\Omega^{V_t}=\Omega_t$ if $W_t=V_t$. 

It is not difficult to check that there exists a cellular decomposition of $\widehat{\Omega}_t$ and the subspace $\Omega^{W_t}$ is a $|W_t|$-dimensional cell, namely
\begin{equation}\label{cell2}
\widehat{\Omega}_t=\bigsqcup_{W_t\subset V_t}\Omega^{W_t}=\Omega_t\bigsqcup\partial_{\infty}\Omega_t,
\end{equation} 
where $\partial_{\infty}\Omega_t:=\bigsqcup_{W_t\subsetneq V_t}\Omega^{W_t}$ is the ideal hyperbolic admissible space. Moreover, the space $\mathbb{R}^{W_t}_{+}$ has the subspace 
$$
\Omega_{W_t}:=\{r_{W_t}=(r_i)_{i\in W_t}\in\mathbb{R}_+^{W_t}:Q^{\mathbb{H}}_3(r^{W_t}_t)>0\}.
$$
for all non-empty subset $W_t\subset V_t$. Then we have $p^{W_t}(\Omega^{W_t})=\Omega_{W_t}$ and $\Omega^{V_t}=\Omega_{V_t}=\Omega_t$ 

\begin{prop}\label{piece}
    The space $\Omega_{W_t}$ is a simply connected domain with piecewise analytic boundary for all non-empty subset $W_t\subset V_t$.
\end{prop}

\begin{proof}
 For any non-empty and proper subset $W_t\subset V_t$, by the argument similar to the proof of Lemma \ref{sphericaladmissiblespace}, the space $\Omega_{W_t}$ is contractible and $\Omega_{W_t}$ is a simply connected domain.  

If $W_t=V_t$, by Lemma \ref{hyperbolicadmissiblespace}, the space $\Omega_{V_t}=\Omega_t$ is a simply connected domain.

For any non-empty subset $W_t\subset V_t$, by definition, the space $\Omega_{W_t}$ is bounded by a piecewise analytic submanifold. This completes the proof.
\end{proof}

Given a non-empty subset $W_t\subset V_t$ and $r_{W_t}\in\Omega_{W_t}$, suppose that $r^{W_t}_t=(r_i)_{i\in V_t}\in \Omega^{W_t}$, then for any $i\in W_t$, define $r_i^{W_t}:=(r_j)_{j\in V_t\setminus\{i\}}$ and 
$$
f(r_i^{W_t}):=\sum_{j\in V_t\setminus\{i\} }x_j+2\sqrt{\sum_{j\ne k\in V_t\setminus\{i\}}x_jx_k+1},
$$
where $x_j=\coth r_j$ for all $j\in V_t\setminus\{i\}$.

\begin{prop}\label{amn}
 For any non-empty subset $W_t\subset V_t$, we have 
\begin{equation}
 \mathbb{R}_+^{W_t}\setminus\Omega_{W_t}=\bigsqcup_{i\in W_t} U_i^{W_t},
\end{equation}
where 
$$
U_i^{W_t}=\{r_{W_t}=(r_j)_{j\in W_t}\in\mathbb{R}_+^{W_t}:0<r_i\le \coth^{-1}f(r_i^{W_t}) \}
$$
is a connected component of $\mathbb{R}_+^{W_t}\setminus\Omega_{W_t}$ for all $i\in W_t$.
\end{prop}

\begin{proof}
For any non-empty subset $W_t\subset V_t$, by the argument similar to the proof of Lemma \ref{sphericaladmissiblespace}, $U_i^{W_t}$ is  contractible and $U_i^{W_t}$ is a simply connected domain for all $i\in W_t$.    

For any $i\in W_t$ and $r_{W_t}=(r_j)_{j\in W_t}\in U_i^{W_t}$, we have $r_i<\min_{j\in W_t\setminus\{i\}}r_j$, which implies that $U_i^{W_t}$ and $U_j^{W_t}$ are disjoint for all $i\ne j\in W_t$.

For any $r_{W_t}\in\bigsqcup_{i\in W_t} U_i^{W_t}$, by definition, we have $r_{W_t}\in\mathbb{R}_+^{W_t}\setminus\Omega_{W_t}$, which implies that $\bigsqcup_{i\in W_t} U_i^{W_t}\subset\mathbb{R}_+^{W_t}\setminus\Omega_{W_t}$. 

For any $r_{W_t}=(r_i)_{i\in W_t}\in\mathbb{R}_+^{W_t}\setminus\Omega_{W_t}$, we have $Q^{\mathbb{H}}(r^{W_t}_t)\le 0$, then suppose that $r_i=\min_{j\in W_t}r_j$, we have $r_i<\min_{j\in W_t\setminus\{i\}}r_j$ and $\coth r_i\ge f(r_i^{W_t})$, which implies that $r_{W_t}\in U_i^{W_t}$ and $\mathbb{R}_+^{W_t}\setminus\Omega_{W_t}\subset\bigsqcup_{i\in W_t} U_i^{W_t}$. 
\end{proof}

Following Cooper and Rivin \cite{Cooper96}, define the following functional $\mathcal{S}_t$ on $\Omega_{t}$, namely
\begin{equation}\label{func89}
\mathcal{S}_t(r_t):=\sum\limits_{i\in V_t}\alpha_i^t(r_t)r_i+2 \mathrm{Vol}_t(r_t),~r_t=(r_i)_{i\in V_t}\in\Omega_{t},
\end{equation}
where $\alpha_i^t(r_t)$ is the solid angle at the vertex $v_i$ of conformal hyperbolic tetrahedron $\Delta^3(r_t)$ parameterized by $r_t$ and $\mathrm{Vol}_t(r_t)$ is the volume of $\Delta^3(r_t)$. 

The classical Schl$\ddot{\text{a}}$fli formula (see e.g. \cite{Milnor94}) gives the first-order variation of the volume $\mathrm{Vol}$ of a non-degenerate hyperbolic (Euclidean, spherical, resp.) tetrahedron $\Delta^3$, namely
\begin{equation}\label{schla}
2\kappa d\mathrm{Vol}=\sum_{{ij}}\ell_{ij}d\theta_{kl},
\end{equation}
where $\kappa=-1(0, 1, \text{resp}.)$ and the sum is taken over all geodesic edge $e_{ij}$. 
\begin{remark}\label{ki3}
The Schl$\ddot{\text{a}}$fli formula $($\ref{schla}$)$ is also valid if $\Delta^3$ is an ideal hyperbolic tetrahedron and the $\ell_{ij}$ is the length of the edge $e_{ij}$ truncated by the horospheres at the ideal vertices if $e_{ij}$ has ideal vertices. Note the formula $($\ref{schla}$)$ does not depend on the choice of the horospheres at ideal vertices, since the sum of the dihedral angles of the edges with the same ideal vertex is always equal to $\pi$. 
\end{remark}

Then by (\ref{schla}), we have the following first-order variation of $\mathcal{S}_t$:
\begin{equation}
d\mathcal{S}_t=\sum\limits_{i\in V_t}\alpha_i^tdr_i.
\end{equation}

For any non-empty subset $W_t\subset V_t$ and $r_{W_t}=(r_i)_{i\in W_t}\in\Omega_{W_t}$, we have $r^{W_t}_t=(r_i)_{i\in V_t}\in\Omega^{W_t}$ and there exists a sequence $\{r_t^{(m)}\}_{m=1}^{\infty}\subset\Omega_t$ such that $r_t^{(m)}\to r_t^{W_t}$, $\Delta^3(r_t^{(m)})\to \Delta^3(r_t^{W_t})(m\to\infty)$. Then we can define the following functional $\mathcal{S}_{W_t}$ on $\Omega_{W_t}$, namely
\begin{equation}\label{func}
\mathcal{S}_{W_t}(r_{W_t}):=\lim_{m\to\infty}\mathcal{S}_t(r_t^{(m)}),~r_{W_t}\in\Omega_{W_t}.
\end{equation}

The following result implies that the functional $\mathcal{S}_{W_t}$ is well-defined.

\begin{theorem}\label{yu3}
  The functional 
$$
  \mathcal{S}_{W_t}(r_{W_t})=\sum_{i\in W_t}\alpha_i^{t}(r^{W_t}_t)r_i+2 \mathrm{Vol}_{t}(r^{W_t}_t),~r_{W_t}=(r_i)_{i\in W_t}\in\Omega_{W_t}
  $$
and the first-order variation 
\begin{equation}\label{ki4}
d\mathcal{S}_{W_t}=\lim_{m\to\infty}d\mathcal{S}_{t}(r_t^{(m)})=\sum_{i\in W_t}\alpha_i^{t}dr_i
\end{equation}
for all non-empty subset $W_t\subset V_t$, where $\alpha_i^{t}(r^{W_t}_t)$ is the solid angle at the vertex $v_i$ of the (ideal) hyperbolic tetrahedron $\Delta^3(r_t^{W_t})$ parameterized by $r_t^{W_t}$ and $\mathrm{Vol}_{t}(r^{W_t}_t)$ is the volume of $\Delta^3(r_t^{W_t})$. 

In particular, $\mathcal{S}_{V_t}=\mathcal{S}_{t}$ and $d\mathcal{S}_{V_t}=d\mathcal{S}_{t}$ if $W_t=V_t$.
\end{theorem}

\begin{proof}
For any non-empty subset $W_t\subset V_t$ and $r_{W_t}=(r_i)_{i\in W_t}\in\Omega_{W_t}$, suppose that $r_t^{(m)}=(r_i^{(m)})_{i\in V_t}$ for all $k\ge 1$, then by (\ref{func89}) and (\ref{func}), we have
\begin{equation}\label{func1}
    \mathcal{S}_{W_t}(r_{W_t})=\lim_{m\to\infty}\sum_{i\in V_t}\alpha_i^{(m)}r_i^{(m)}+2\lim_{m\to\infty}\mathrm{Vol}^{(m)},
\end{equation}
where $\alpha_i^{(m)}=\alpha_i^{t}(r_t^{(m)})$ and $\mathrm{Vol}^{(m)}=\mathrm{Vol}_t(r_t^{(m)})$.

If $W_t\subset V_t$ is a non-empty and proper subset, then for any $i\in V_t\setminus W_t$, choose the point $v_{ij}^{(m)}\in e_{ij}^{(m)}$ such that $d_{-1}(v_i^{(m)},v_{ij}^{(m)})=r_i^{(m)}$ for all $j\in V_t\setminus\{i\}$, note that $e_{ij}=e_{ji}$ and $v_{ij}=v_{ji}$. Moreover, $\{v_{ij}^{(m)}\}_{m=1}^{\infty}$ converges to the tangent point $v_{ij}\in e_{ij}$ of the four-sphere configuration $\mathcal{P}(r^{W_t}_t)$ of the ideal hyperbolic tetrahedron $\Delta^3(r_t^{W_t})$ for all $j\in V_t\setminus\{i\}$.

Suppose that $V_t=\{i,j,k,l\}$, define $\widetilde{F}_i^{(m)}:=H(v_{ij}^{(m)},v_{ik}^{(m)},v_{il}^{(m)})\bigcap\Delta^3(r_t^{(m)})$ $(m\ge1)$($\widetilde{F}_i:=H(v_{ij},v_{ik},v_{il})\bigcap\Delta^3(r_t^{W_t})$,resp.), where $H(v_{ij}^{(m)},v_{ik}^{(m)},v_{il}^{(m)})$ ($H(v_{ij},v_{ik},v_{il})$, resp.) is the totally geodesic plane containing $v_{ij}^{(m)},v_{ik}^{(m)},v_{il}^{(m)}$ ($v_{ij},v_{ik},v_{il}$, resp.), note that $\text{Area}(\widetilde{F}_i^{(m)})$($\text{Area}(\widetilde{F}_i)$, resp.)$\le\pi$. 

Since $\text{Area}(\Delta v_iv_jv_k)$($\text{Area}(\Delta v_iv_jv_l)$, $\text{Area}(\Delta v_iv_kv_l)$, resp.)$\le\pi$, then the radius of the incircle of $\Delta v_iv_jv_k$($\Delta v_iv_jv_l$, $\Delta v_iv_kv_l$, resp.) has an upper bound, which implies that $\ell_{ij,ik}$, $\ell_{ij,il}$, $\ell_{ik,il}$ and $\text{diam}(\widetilde{F}_i)$ have upper bounds, where $\ell_{ij,ik}$($\ell_{ij,il}$, $\ell_{ik,il}$, resp.) is the length of the geodesic segment $e_{ij,ik}$($e_{ij,il}$, $e_{ik,il}$, resp.) connecting $v_{ij}$($v_{ij}$, $v_{ik}$, resp.) and $v_{ik}$($v_{il}$, $v_{il}$, resp.).  

Choose the point $\widetilde{o}_i^{(m)}\in\widetilde{F}_i^{(m)}(m\ge1)$ such that $\delta^{(m)}_i:=d_{-1}(v_i^{(m)},\widetilde{o}_i^{(m)})=$ $\min_{p\in\widetilde{F}_i^{(m)}}d_{-1}(v_i^{(m)},p)$. By triangle inequality, we have
\begin{equation}\label{kl1}
    r_i^{(m)}-\delta^{(m)}_i\le d_{-1}(v_{ij}^{(m)}, \widetilde{o}_i^{(m)})\le\text{diam}(\widetilde{F}_i^{(m)}),~\text{for all}~m\ge1,
\end{equation}

Denote the intersection of the tetrahedron $\Delta^3(r_t^{(m)})$ and the geodesic sphere $S(v_i^{(m)},\delta^{(m)}_i)$ by $S_i^{(m)}$, then by the corresponding relation of points on the geodesic passing through $v_i^{(m)}$, we have the following diffeomorphism between $S_i^{(m)}$ and $\widetilde{F}_i^{(m)}$ in terms of polar coordinates:
\begin{equation}\label{yingshe}
\begin{aligned}
\Phi:S_i^{(m)}&\longrightarrow \widetilde{F}^{(m)}\\(\delta^{(m)}_i,s^{(m)}_i)&\longmapsto(\rho^{(m)}_i,s^{(m)}_i), 
\end{aligned}
\end{equation}
where $s^{(m)}_i\in \mathbb{S}_i^{(m)}:=S(v_i^{(m)},1)$ and $\rho^{(m)}_i=\rho^{(m)}_i(s^{(m)}_i)$.

By (\ref{yingshe}), we have 
\begin{align*}\label{yu89}
d\text{Area}(\widetilde{F}_i^{(m)})(\rho_i^{(m)},s^{(m)}_i)&\geq \sinh^2(\rho_i^{(m)})d\text{Area}(\mathbb{S}_i^{(m)})(s_i^{(m)})\\
&\geq\sinh^2(\delta^{(m)}_i)d\text{Area}(\mathbb{S}_i^{(m)})(s_i^{(m)})\\
&=d\text{Area}(S_i^{(m)})(\delta^{(m)}_i,s^{(m)}_i),    
\end{align*}
which implies that 
\begin{equation}\label{yu85}
\pi\ge\text{Area}(\widetilde{F}_i^{(m)})\ge\text{Area}(S_i^{(m)})=\alpha_i^{(m)}\sinh^2(\delta^{(m)}_i),~\text{for all}~m\ge1.
\end{equation}

Since $\text{diam}(\widetilde{F}_i)$ has an upper bound, then by (\ref{kl1}) and (\ref{yu85}), there exists a constant $C_i$ such that 
$$
\lim\sup_{m\rightarrow\infty}\alpha_{i}^{(m)}e^{2r_i^{(m)}}\leq C_i<\infty,
$$
which implies that $\lim\sup_{m\rightarrow\infty}\alpha_{i}^{(m)}e^{r_i^{(m)}}\leq C_i<\infty$ and 
\begin{equation}\label{yu2}
\lim_{m\rightarrow\infty}\alpha_{i}^{(m)}r_i^{(m)}=0,~\text{for all}~i\in V_t\setminus W_t.
\end{equation}

By (\ref{func1}) and (\ref{yu2}), we have
\begin{equation}\label{ki1}
 \mathcal{S}_{W_t}(r_{W_t})=\lim_{m\to\infty}\sum_{i\in W_t}\alpha_i^{(m)}r_i^{(m)}+2\lim_{m\to\infty}\mathrm{Vol}^{(m)}=\sum_{i\in W_t}\alpha_i^{t}(r^{W_t}_t)r_i+2 \mathrm{Vol}_{t}(r^{W_t}_t).
\end{equation}
Moreover, by Remark \ref{ki3} and (\ref{ki1}), we have $d\mathcal{S}_{W_t}=\sum_{i\in W_t}\alpha_i^{t}dr_i$.

If $W_t=V_t$, then by (\ref{func1}), we have $\mathcal{S}_{V_t}=\mathcal{S}_{t}$ and $d\mathcal{S}_{V_t}=d\mathcal{S}_{t}$.    
\end{proof}

For any non-empty subset $W_t\subset V_t$, by Theorem \ref{yu3}, we have
$$
\nabla\mathcal{S}_{W_t}=(\alpha_i^{t})_{i\in W_t},~  \text{Hess}~\mathcal{S}_{W_t}=\left(\frac{\partial\alpha_i^{t}}{\partial r_j}\right)_{i,j\in W_t}.
$$
Moreover, by (\ref{ki4}), the functional $\mathcal{S}_{W_t}$ can be rewritten as
\begin{equation}\label{fun3}
 \mathcal{S}_{W_t}(r_{W_t})= \mathcal{S}_{W_t}(\mathbf{1}_{W_t})+\int_{\mathbf{1}_{W_t}}^{r_{W_t}}\sum_{i\in W_t}\alpha_i^{t}dr_i,    
\end{equation}
where $\mathbf{1}_{W_t}=(1)_{i\in W_t}\in\Omega_{W_t}$.

\begin{lemma}\label{bb7}
    For any $r_t\in\partial_{\infty}\Omega_t$ and sequence $\{r_t^{(m)}\}_{m=1}^\infty\subset\Omega_t$ such that $r_t^{(m)}\to r_t(m\to\infty)$, we have $\Delta^3(r_t^{(m)})\to\Delta^3(r_t)(m\to\infty)$ up to isometry.
\end{lemma}

\begin{proof}
Suppose that $r_t=(r_i)_{i\in V_t}$ and $r_t^{(m)}=(r_i^{(m)})_{i\in V_t}(m\ge1)$. Since $\Delta^3(r_t^{(m)})$ is conformal, by the proof of Theorem \ref{highpolyhedron}, we have
\begin{equation}\label{bb1}
G_{ij}^{(m)}=-2\dfrac{\Pi_{k\in V_t} s_k^{(m)2}}{s_i^{(m)}s_j^{(m)}}Q_{ij}^{(m)},~G_{ii}^{(m)}=-2\dfrac{\Pi_{k\in V_t} s_k^{(m)2}}{s_i^{(m)2}}Q_{ii}^{(m)},~m\ge1
\end{equation}
for all $i\ne j\in V_t$ and $i\in V_t$, where  $s_i^{(m)}=\sinh r_i^{(m)}$, $c_i^{(m)}=\coth r_i^{(m)}$ and  
$$
Q_{ij}^{(m)}=\sum_{k\in V_t}c_k^{(m)2}-(c_i^{(m)}+c_j^{(m)})\sum_{k\in V_t}c_k^{(m)}+2c_i^{(m)}c_j^{(m)}-2,
$$

$$
Q_{ii}^{(m)}=(\sum_{k\in V_t\setminus \{i\}}c_k^{(m)})^2-\sum_{k\in V_t\setminus \{i\}}c_k^{(m)2}+2.
$$

By Lemma \ref{lengthandangle} and (\ref{bb1}), we have
\begin{equation}\label{bb2}
\cos\theta_{ij}^{(m)}=\dfrac{G_{ij}^{(m)}}{\sqrt{G_{ii}^{(m)}G_{jj}^{(m)}}}=-\dfrac{Q_{ij}^{(m)}}{\sqrt{Q_{ii}^{(m)}Q_{jj}^{(m)}}},~m\ge1.
\end{equation}

 Since any 2-dimensional simplex of $\Delta^3(r_t)$ is also conformal, then by Theorem \ref{highpolyhedron}, we have
$$
Q_{ii}:=(\sum_{k\in V_t\setminus \{i\}}c_k)^2-\sum_{k\in V_t\setminus \{i\}}c_k^{2}+2=Q_2(r_{V_t\setminus \{i\}})>0,~\text{for all}~i\in V_t, 
$$
where $c_k=\coth r_k$ and $r_{V_t\setminus \{i\}}=(r_k)_{k\in V_t\setminus \{i\}}$.

Moreover, define 
$$
Q_{ij}:=\sum_{k\in V_t}c_k^{2}-(c_i+c_j)\sum_{k\in V_t}c_k+2c_ic_j-2, 
$$
since $r_t^{(m)}\to r_t(m\to\infty)$, then we have
\begin{equation}\label{bb3}
Q_{ij}^{(m)}\to Q_{ij},~Q_{ii}^{(m)}\to Q_{ii},~Q_{jj}^{(m)}\to Q_{jj}(m\to\infty). 
\end{equation}

By (\ref{bb2}) and (\ref{bb3}), we have
\begin{equation}\label{bb4}
\cos\theta_{ij}^{(m)}\to-\dfrac{Q_{ij}}{\sqrt{Q_{ii}Q_{jj}}}(m\to\infty).
\end{equation}

Since $\Delta^3(r_t)$ is a conformal ideal hyperbolic tetrahedron, then there exists a sequence of conformal hyperbolic tetrahedra $\{\widetilde{\Delta}^3_m\}_{m=1}^{\infty}\subset\mathbb{H}^3$ such that $\widetilde{\Delta}^3_m\to\Delta^3(r_t)(m\to\infty)$, which implies that there exists a sequence $\{\widetilde{r_t}^{(m)}\}_{m=1}^{\infty}\subset\Omega_t$ such that $\widetilde{\Delta}^3_m=\Delta^3(\widetilde{r_t}^{(m)})(m\ge1)$, $\widetilde{r_t}^{(m)}\to r_t$,$\widetilde{\theta_{ij}}^{(m)}\to\theta_{ij}(m\to\infty)$.

By above argument and Lemma \ref{lengthandangle}, we have
\begin{equation}\label{bb5}
\cos\widetilde{\theta_{ij}}^{(m)}\to-\dfrac{Q_{ij}}{\sqrt{Q_{ii}Q_{jj}}}=\cos\theta_{ij}(m\to\infty)
\end{equation}

By (\ref{bb4}) and (\ref{bb5}), we have
\begin{equation}\label{bb6}
\cos\theta_{ij}^{(m)}\to\cos\theta_{ij},~\theta_{ij}^{(m)}\to\theta_{ij}(m\to\infty),~\text{for all}~i\ne j\in V_t.
\end{equation}

Since $r_t^{(m)}\to r_t(m\to\infty)$, then we have $\ell_{ij}^{(m)}\to\ell_{ij}(m\to\infty)$ for all $i\ne j\in V_t$. Moreover, by (\ref{bb6}), we obtain $\Delta^3(r_t^{(m)})\to\Delta^3(r_t)(m\to\infty)$ up to isometry.   
\end{proof}

\begin{remark}
    Lemma \ref{bb7} can also be generalized to ideal hyperbolic $n$-simplices. 
\end{remark}

\begin{prop}\label{ki5}
   For any tetrahedron $t\in T$, non-empty and proper subset $W_t\subset V_t$, there exists a $r_{W_t}\in\Omega_{W_t}$ such that the isogonal geodesic at the vertex $v_i$ of the ideal hyperbolic tetrahedron $\Delta^3(r_t^{W_t})$ passes through its interior for all $i\in W_t$. 
\end{prop}

\begin{proof}
For any non-empty and proper subset $W_t\subset V_t$, choose a $r_{W_t}=(r)_{i\in W_t}\in\Omega_{W_t}$, note that $r_t^{W_t}=(r_i)_{i\in V_t}\in\Omega^{W_t}$, where
$$
r_i=\begin{cases}
 r, &  i\in W_t, \\
 \infty, & i\in V_t\setminus W_t.
\end{cases}
$$

We can construct a sequence $\{r_t^{(m)}\}_{m=1}^{\infty}\subset\Omega_{t}$ such that $r_t^{(m)}\to r_t^{W_t}(m\to \infty)$, where $r_t^{(m)}=(r_i^{(m)})_{i\in V_t}$ and 
$$
r_i^{(m)}=\begin{cases}
 r, &  i\in W_t, \\
 R^{(m)}, & i\in V_t\setminus W_t.
\end{cases}
$$
Then by Lemma \ref{bb7}, after an isometry, we have $\Delta^3(r_t^{(m)})\to\Delta^3(r_t^{W_t})(m\to\infty)$, where $\Delta^3(r_t^{(m)})$ and $\Delta^3(r_t^{W_t})$ have the same vertex $v_i$ if $i\in W_t$, and $\Delta^3(r_t^{(m)})(\Delta^3(r_t^{W_t}),\text{resp}.)$ has vertex $v_i^{(m)}(v_i,\text{resp}.)$ if $i\in V_t\setminus W_t$.

Assume that $V_t=\{i,j,k,l\}$, if $|W_t|=1$, then suppose that $W_t=\{i\}$, by symmetry, the isogonal geodesic $\gamma^{(m)}$ at the vertex $v_i$ of the conformal ideal hyperbolic tetrahedron $\Delta^3(r_t^{(m)})$ passes through its interior and the center of $\Delta v_j^{(m)}v_k^{(m)}v_l^{(m)}$ for all $n\ge1$. Note that $\gamma^{(m)}$ is always the same geodesic, denote by $\gamma$, then as $m\to\infty$, the isogonal geodesic at the vertex $v_i$ of the conformal ideal hyperbolic tetrahedron $\Delta^3(r_t^{W_t})$ is $\gamma$ and passes through its interior.  

If $|W_t|=2$, then for any $i\in W_t$, suppose that $W_t=\{i,j\}$, add the geodesic segment $e_{i,kl}^{(m)}$($e_{j,kl}^{(m)}$, $e_{ij,kl}^{(m)}$, resp.) connecting $v_i$ ($v_{j}$, $v_{ij}$, resp.) and $v_{kl}^{(m)}$ for all $m\ge1$, where $v_{ij}\in e_{ij}$($v_{kl}^{(m)}\in e_{kl}^{(m)}$, resp.) such that $d_{-1}(v_i,v_{ij})=d_{-1}(v_j,v_{ij})=r$($d_{-1}(v_k^{(m)},v_{kl}^{(m)})=d_{-1}(v_l^{(m)},v_{kl}^{(m)})=R^{(m)}$, resp.). Then we have
$$
 e_{i,kl}^{(m)}\perp e_{kl}^{(m)},~e_{j,kl}^{(m)}\perp e_{kl}^{(m)},~e_{ij}\perp e_{ij,kl}^{(m)},
$$
$$
\ell_{i,kl}^{(m)}=\ell_{j,kl}^{(m)}:=\ell_1^{(m)},~\theta_{ik,i(kl)}^{(m)}=\theta_{il,i(kl)}^{(m)}:=\theta_1^{(m)}
$$
for all $m\ge 1$, where $\theta_{ik,i(kl)}^{(m)}$($\theta_{il,i(kl)}^{(m)}$, $\theta_{ij,i(kl)}^{(m)}$, resp.) is the angle between the geodesic segments $e_{ik}^{(m)}$($e_{il}^{(m)}$, $e_{ij}^{(m)}$, resp.) and $e_{i,kl}^{(m)}$. 

Define $\theta_{2}^{(m)}:=\theta_{ij,i(kl)}^{(m)}$, then for the hyperbolic triangles $\Delta v_iv_k^{(m)}v_{kl}^{(m)}$ and $\Delta v_iv_{ij}v_{kl}^{(m)}$, by hyperbolic trigonometry, we have
\begin{equation}\label{kop5}
\sin\theta_1^{(m)}=\dfrac{\sinh R^{(m)}}{\sinh (r+R^{(m)})}\to\dfrac{1}{e^r}:=\sin\theta_1,
\end{equation}
\begin{equation}\label{kop6}
\cosh\ell_1^{(m)}=\dfrac{\cosh (r+R^{(m)})}{\cosh R^{(m)}}\to e^r:=\cosh\ell_1,
\end{equation}
\begin{equation}\label{kop7}
\cosh\ell_2^{(m)}=\dfrac{\cosh \ell_1^{(m)}}{\cosh r}\to\dfrac{\cosh \ell_1}{\cosh r}=2\dfrac{e^r}{e^r+e^{-r}}:=\cosh\ell_2,
\end{equation}
\begin{equation}\label{kop2}
\sin\theta_2^{(m)}=\dfrac{\sinh \ell_2^{(m)}}{\sinh \ell_1^{(m)}}\to\dfrac{\sinh \ell_2}{\sinh \ell_1}:=\sin\theta_2
\end{equation}
as $m\to\infty$. Since $\cosh\ell_2=2\dfrac{e^r}{e^r+e^{-r}}\to2(r\to\infty)$, then we can choose a sufficiently large $r$ such that $\cosh\ell_2>\sqrt{2}$. Moreover, by (\ref{kop5}), (\ref{kop6}) and (\ref{kop2}), we have
$$
\sin\theta_2=\dfrac{\sinh \ell_2}{\sinh \ell_1}=\dfrac{\sqrt{\cosh^2 \ell_2-1}}{\sinh \ell_1}>\dfrac{1}{\sinh \ell_1}>\dfrac{1}{\cosh \ell_1}=\sin\theta_1,
$$
which implies that $\theta_2>\theta_1$.

For the conformal ideal hyperbolic tetrahedron $\Delta^3(r_t^{W_t})$, choose a point $s\in e_{j,kl}$ and denote the geodesic passing through $v_i$ and $s$ by $\gamma_{is}$, then as $s$ moves continuously along the geodesic segment $e_{j,kl}$ from $v_{kl}$ to $v_{j}$, the angle $\theta_{ik,is}=\theta_{il,is}$ increases strictly from $\theta_1$ to $\angle v_kv_iv_j$ and $\theta_{ij,is}$ decreases strictly from $\theta_2$ to $0$. 

Since $\theta_2>\theta_1$, then there exists the unique point $s\in e_{j,kl}$ such that $\theta_{ij,is}=\theta_{ik,is}=\theta_{il,is}$, which implies that $\gamma_{is}$ is the isogonal geodesic at the vertex $v_i$ of $\Delta^3(r_t^{W_t})$ and passes through its interior.  

If $|W_t|=3$, then for any $i\in W_t$, suppose that $W_t=\{i,j,k\}$, add the geodesic segment $e_{i,jk}$($e_{ol}^{(m)}$, resp.) connecting $v_i$($v_l^{(m)}$, resp.) and $v_{jk}$($o$, resp.) for all $m\ge1$, where $v_{jk}\in e_{jk}$ such that $d_{-1}(v_j,v_{jk})=d_{-1}(v_k,v_{jk})=r$ and $o\in e_{i,jk}$ is the center of $\Delta v_iv_jv_k$. Then we have
$$
 e_{i,jk}\perp e_{ol}^{(m)},e_{jk},~\theta_{ij,i(jk)}=\theta_{ik,i(jk)}:=\theta_1
$$
for all $m\ge 1$. Define $\ell_1:=\ell_{oi}\le2r$, $\ell_2^{(m)}:=\ell_{ol}^{(m)}$ and $\theta_2^{(m)}:=\theta_{il,i(jk)}^{(m)}$, then for the triangles $\Delta v_iv_jv_{jk}$ and $\Delta ov_iv_l^{(m)}$, by hyperbolic trigonometry, we have
\begin{equation}\label{kop4}
\sin\theta_1=\dfrac{\sinh r}{\sinh 2r}=\dfrac{1}{2\cosh r},~\sin\theta_2^{(m)}=\dfrac{\sinh \ell_2^{(m)}}{\sinh (r+R^{(m)})},
\end{equation}
\begin{equation}\label{kop1}
\cosh\ell_2^{(m)}=\dfrac{\cosh(r+R^{(m)})}{\cosh\ell_1}\to\infty(m\to\infty)
\end{equation}
for all $m\ge1$. Moreover, by (\ref{kop1}), we have $\ell_2^{(m)}\to\infty(m\to\infty)$ and 
\begin{equation}\label{kop3}
\begin{aligned}
\lim_{m\to\infty}\sin\theta_2^{(m)}&=\lim_{m\to\infty}\dfrac{\sinh \ell_2^{(m)}}{\sinh (r+R^{(m)})}=\lim_{m\to\infty}\dfrac{\cosh \ell_2^{(m)}}{\sinh (r+R^{(m)})}\\
&=\lim_{m\to\infty}\dfrac{\cosh(r+R^{(m)})}{\cosh\ell_1\sinh (r+R^{(m)})}=\dfrac{1}{\cosh\ell_1}:=\sin\theta_2\\
&\ge\dfrac{1}{\cosh2r}.
\end{aligned}
\end{equation}

Since $\cosh 2r\to1(r\to 0)$ and $2\cosh r\to 2(r\to0)$, then we can choose a sufficiently small $r$ such that $2\cosh r>\cosh 2r$. Moreover, by (\ref{kop4}) and (\ref{kop3}), we have
$$
\sin\theta_2\ge\dfrac{1}{\cosh2r}>\dfrac{1}{2\cosh r}=\sin\theta_1,
$$
which implies that $\theta_2>\theta_1$.

For the conformal ideal hyperbolic tetrahedron $\Delta^3(r_t^{W_t})$, choose a point $s\in e_{l,jk}$ and denote the geodesic passing through $v_i$ and $s$ by $\gamma_{is}$, then as $s$ moves continuously along the geodesic segment $e_{l,jk}$ from $v_{jk}$ to $v_{l}$, the angle $\theta_{ij,is}=\theta_{ik,is}$ increases strictly from $\theta_1$ to $\angle v_lv_iv_j$ and $\theta_{il,is}$ decreases strictly from $\theta_2$ to $0$. 

Since $\theta_2>\theta_1$, then there exists the unique point $s\in e_{l,jk}$ such that $\theta_{ij,is}=\theta_{ik,is}=\theta_{il,is}$, which implies that $\gamma_{is}$ is the isogonal geodesic at the vertex $v_i$ of $\Delta^3(r_t^{W_t})$ and passes through its interior.      
\end{proof}

The following result is the generalization of Corollary \ref{iop}.

\begin{corollary}\label{ko86}
    For any tetrahedron $t\in T$, non-empty and proper subset $W_t\subset V_t$, there exists a $r_{W_t}\in\Omega_{W_t}$ such that 
     \begin{equation}\label{ju8}  
    \begin{cases}
    \left.\dfrac{\partial \alpha_i^t}{\partial r_i}\right|_{r_{W_t}}\le0,\left.\dfrac{\partial \mathrm{Vol}_t}{\partial r_i}\right|_{r_{W_t}}\ge 0,& \text{for all}~i\in W_t,\\
    \\
  \left.\dfrac{\partial \alpha_i^t}{\partial r_j}\right|_{r_{W_t}}\ge0,&  \text{for all}~i\ne j\in W_t,\\
\end{cases}
    \end{equation}  
where $\alpha_i^t$ is the solid angle at the vertex $i$ of the $($ideal$)$ hyperbolic tetrahedron $\Delta^3(r_t^{W_t})$ and $\mathrm{Vol}_t$ is the hyperbolic volume of $\Delta^3(r_t^{W_t})$.
\end{corollary}

\begin{proof}
For any non-empty and proper subset $W_t\subset V_t$, by Proposition \ref{ki5}, there exists a $r_{W_t}\in\Omega_{W_t}$ and sufficiently small $\epsilon>0$ such that $\Delta^3(r_t^{W_t}-\boldsymbol{\epsilon}_i)$ is strictly contained in $\Delta^3(r_t^{W_t})$ for all $i\in W_t$, where $\Delta^3(r_t^{W_t})$ and $\Delta^3(r_t^{W_t}-\boldsymbol{\epsilon}_i)$ are considered as tetrahedra with the same vertices $\{v_j\}_{j\in V_t\setminus\{i\}}$ and $\boldsymbol{\epsilon}_i=(0,\epsilon(\text{i-th}),0,0)$. By this containment, we have
$$
\begin{cases}
\alpha_i^t(r_t^{W_t}-\boldsymbol{\epsilon}_i)>\alpha_i^t(r_t^{W_t}),& \text{for all}~i\in W_t,\\
 \alpha_i^t(r_t^{W_t}-\boldsymbol{\epsilon}_j)<\alpha_i^t(r_t^{W_t}),& \text{for all}~i\ne j\in W_t,\\
  \mathrm{Vol}_t(r_t^{W_t}-\boldsymbol{\epsilon}_i)<\mathrm{Vol}_t(r_t^{W_t}),& \text{for all}~i\in W_t,
\end{cases}
$$
which implies (\ref{ju8}).    
\end{proof}

The following result is the generalization of Theorem \ref{maintheorem3}.

\begin{theorem}\label{negac}
The signature of the Jacobian matrix $\left(\dfrac{\partial \alpha_i^t}{\partial r_j}\right)_{i,j\in W_t}$ is $(0,0,|W_t|)$ on $ \Omega_{W_t}$ for all tetrahedron $t\in T$ and non-empty subset $W_t\subset V_t$.  
\end{theorem}

\begin{proof}
For any non-empty subset $W_t\subset V_t$ and $r_{W_t}=(r_i)_{i\in W_t}\in\Omega_{W_t}$, we have $r^{W_t}_t\in\Omega^{W_t}_t$. Choose a sequence $\{r_t^{(m)}\}_{m=1}^{\infty}\subset \Omega_t$ such that $r_t^{(m)}\to r^{W_t}_t$, where $r_t^{(m)}=(r_i^{(m)})_{i\in V_t}$, then by Lemma \ref{bb7}, we have $\Delta^3(r_t^{(m)})\to\Delta^3(r^{W_t}_t)(m\to\infty)$ up to isometry. Moreover, we have
\begin{equation}\label{hj8}
\alpha_i^t=\begin{cases}
 \theta_{jk}^t+\theta_{jl}^t+\theta_{kl}^t-\pi, &\text{for all}~i\in W_t, \\
 0 & \text{for all}~i\in V_t\setminus W_t,
\end{cases}
\end{equation}
where $\theta_{jk}^t,\theta_{jl}^t,\theta_{kl}^t$ are the dihedral angles of $\Delta^3(r^{W_t}_t)$. 

By the proof of Lemma \ref{bb7}, we have
\begin{equation}\label{cc1}
\theta_{ij}^t=\arccos\left(-\dfrac{Q_{ij}}{\sqrt{Q_{ii}Q_{jj}}}\right),~\text{for all}~i\ne j\in V_t, 
\end{equation}
where 
$$
Q_{ij}:=\sum_{k\in V_t}c_k^{2}-(c_i+c_j)\sum_{k\in V_t}c_k+2c_ic_j-2,
$$
$$
Q_{ss}=(\sum_{k\in V_t\setminus \{s\}}c_k)^2-\sum_{k\in V_t\setminus \{s\}}c_k^{2}+2>0,~s=i,j,
$$
$c_k=\coth r_k$ if $k\in W_t$ and $c_k=1$ if $k\in V_t\setminus W_t$.

By (\ref{hj8}) and (\ref{cc1}), the solid angle $\alpha_{W_t}^t=(\alpha_i^{t})_{i\in W_t}$ is the smooth real analytic function of $r_{W_t}=(r_i)_{i\in W_t}$.

If $1\le|W_t|\le2$, then for any $i\in W_t$ and $j,k\in V_t\setminus W_t$, by Lemma \ref{lengthandangle}, we have
\begin{equation}\label{hj1}
\begin{aligned}
A^{(m)}:&=c_i^{(m)}c_j^{(m)}(1+t_i^{(m)}t_j^{(m)})=\dfrac{G^{(m)*}_{ij}}{\sqrt{G^{(m)*}_{ii}G^{(m)*}_{jj}}},~m\ge1,\\
B^{(m)}:&=c_i^{(m)}c_k^{(m)}(1+t_i^{(m)}t_k^{(m)})=\dfrac{G^{(m)*}_{ik}}{\sqrt{G^{(m)*}_{ii}G^{(m)*}_{kk}}},~m\ge1,\\
C^{(m)}:&=c_j^{(m)}c_k^{(m)}(1+t_j^{(m)}t_k^{(m)})=\dfrac{G^{(m)*}_{jk}}{\sqrt{G^{(m)*}_{jj}G^{(m)*}_{kk}}},~m\ge1,
\end{aligned}
\end{equation}
where $c_s^{(m)}=\cosh r_s^{(m)}$, $t_s^{(m)}=\tanh r_s^{(m)}$, $s=i,j,k$.

By (\ref{hj1}), we have
\begin{align*}
D^{(m)}:&=\sqrt{A^{(m)}B^{(m)}C^{(m)}}\\
&=c_i^{(m)}c_j^{(m)}c_k^{(m)}\sqrt{(1+t_i^{(m)}t_j^{(m)})(1+t_i^{(m)}t_k^{(m)})(1+t_j^{(m)}t_k^{(m)})}\\
&=\sqrt{\dfrac{G^{(m)*}_{ij}G^{(m)*}_{ik}G^{(m)*}_{jk}}{G^{(m)*}_{ii}G^{(m)*}_{jj}G^{(m)*}_{kk}}},~m\ge1,
\end{align*}
\begin{equation}\label{hj2}
\dfrac{D^{(m)}}{C^{(m)}}=c_i^{(m)}\sqrt{\dfrac{(1+t_i^{(m)}t_j^{(m)})(1+t_i^{(m)}t_k^{(m)})}{1+t_j^{(m)}t_k^{(m)}}}=\sqrt{\dfrac{G^{(m)*}_{ij}G^{(m)*}_{ik}}{G^{(m)*}_{ii}G^{(m)*}_{jk}}},~m\ge1.
\end{equation}

Since $r_i^{(m)}\to r_i$, $r_j^{(m)},r_k^{(m)}\to\infty(m\to\infty)$, then by (\ref{hj2}), we have
\begin{equation}\label{hj10}
r_i=\dfrac{1}{2}\ln 2+\dfrac{1}{2}\ln\dfrac{G^{\ast}_{ij}G^{\ast}_{ik}}{G^{\ast}_{ii}G^{\ast}_{jk}},~\text{for all}~i\in W_t,
\end{equation}
as $m\to\infty$.

If $|W_t|=3$, then for any $i\in W_t$, choose $j\in W_t$ and $k\in V_t\setminus W_t$, by Lemma \ref{lengthandangle}, we also have (\ref{hj1}), then we obtain
\begin{equation}\label{hj3}
\dfrac{B^{(m)}}{C^{(m)}}=\dfrac{c_i^{(m)}(1+t_i^{(m)}t_k^{(m)})}{c_j^{(m)}(1+t_j^{(m)}t_k^{(m)})}=\dfrac{G^{(m)*}_{ik}}{G^{(m)*}_{jk}}\sqrt{\dfrac{G^{(m)*}_{jj}}{G^{(m)*}_{ii}}},~m\ge1.
\end{equation}

Since $r_i^{(m)}\to r_i$, $r_j^{(m)}\to r_j$, $r_k^{(m)}\to\infty(m\to\infty)$, then by (\ref{hj2}), we have
\begin{equation}\label{hj4}
r_i-r_j=\ln\dfrac{G^{\ast}_{ik}}{G^{\ast}_{jk}} +\dfrac{1}{2}\ln\dfrac{G^{\ast}_{jj}}{G^{\ast}_{ii}},
\end{equation}
as $m\to\infty$.

Moreover, by Lemma \ref{lengthandangle}, we have
\begin{equation}\label{hj5}
r_i+r_j=\text{arcosh}\left(\dfrac{G^{\ast}_{ij}}{\sqrt{G^{\ast}_{ii}G^{\ast}_{jj}}}\right).
\end{equation}

By (\ref{hj4}) and (\ref{hj5}), we obtain
\begin{equation}\label{hj6}
r_i=\dfrac{1}{2}\text{arcosh}\left(\dfrac{G^{\ast}_{ij}}{\sqrt{G^{\ast}_{ii}G^{\ast}_{jj}}}\right)+\dfrac{1}{2}\ln\dfrac{G^{\ast}_{ik}}{G^{\ast}_{jk}}+\dfrac{1}{4}\ln\dfrac{G^{\ast}_{jj}}{G^{\ast}_{ii}},~\text{for all}~i\in W_t.
\end{equation}

If $|W_t|=4$, then we have $W_t=V_t$. For any $i\in W_t$, choose $j,k\in W_t$, by Lemma \ref{lengthandangle}, we obtain   
\begin{equation}\label{hj7}
r_i=\dfrac{1}{2}\text{arcosh}\left(\dfrac{G^{\ast}_{ij}}{\sqrt{G^{\ast}_{ii}G^{\ast}_{jj}}}\right)+\dfrac{1}{2}\text{arcosh}\left(\dfrac{G^{\ast}_{ik}}{\sqrt{G^{\ast}_{ii}G^{\ast}_{kk}}}\right)-\dfrac{1}{2}\text{arcosh}\left(\dfrac{G^{\ast}_{jk}}{\sqrt{G^{\ast}_{jj}G^{\ast}_{kk}}}\right),
\end{equation}
for all $i\in W_t$.

Since $\Delta^3(r^{W_t}_t)$ is a conformal ideal hyperbolic tetrahedron, then by Theorem \ref{main1}, the dual tetrahedron $\Delta^3(r^{W_t}_t)^*$ is conformal. Moreover, by Corollary \ref{po5}, there exists a $r_t^*=(r_i^*)_{i\in V_t}\in(0,\pi)^{V_t}$ such that 
\begin{equation}\label{dd1}
\ell_{ij}^{*}=\pi-\theta_{ij}^t=r_i^{*}+r_j^{*}\in(0,\pi).
\end{equation} 

By the proof of Corollary \ref{rigidityoftetrahedron}, we have
\begin{equation}\label{ss1}
r_i^*=\frac{2\pi+2\alpha_i^t-\sum_{j\ne i\in V_t}\alpha_j^t}{6},~\text{for all}~i\in V_t.
\end{equation}

By (\ref{hj8}), (\ref{hj10}), (\ref{hj6}), (\ref{hj7}), (\ref{dd1}) and (\ref{ss1}), $r_{W_t}=(r_i)_{i\in W_t}$ is the smooth real analytic function of $\alpha_{W_t}^t=(\alpha_i^{t})_{i\in W_t}$.  

By above argument, $\text{Hess}~\mathcal{S}_{W_t}=\left(\dfrac{\partial \alpha_i^{t}}{\partial r_j}\right)_{i,j\in W_t}$ is non-degenerate on $\Omega_{W_t}$ for all non-empty subset $W_t\subset V_t$.

By Theorem \ref{maintheorem3}, Proposition \ref{piece}, Remark \ref{ki3}, Corollary \ref{ko86} and the argument similar to the proof of Theorem \ref{maintheorem3}, the signature of the Jacobian matrix $\left(\dfrac{\partial \alpha_i^t}{\partial r_j}\right)_{i,j\in W_t}$ is $(0,0,|W_t|)$ on $\Omega_{W_t}$ for all non-empty subset $W_t\subset V_t$. 
\end{proof}

For any non-empty subset $W_t\subset V_t$, by Theorem \ref{negac}, $\text{Hess}~\mathcal{S}_{W_t}=\left(\dfrac{\partial\alpha_i^{t}}{\partial r_j}\right)_{i,j\in W_t}$ is negative definite on $\Omega_{W_t}$ and we immediately obtain the following corollary.

\begin{corollary}\label{dd1}
    The functional $\mathcal{S}_{W_t}$ is strictly concave on $\Omega_{W_t}$ for all non-empty subset $W_t\subset V_t$ and tetrahedron $t\in T$.
\end{corollary}

Given a non-empty subset $W_t\subset V_t$, for any $r_{W_t}=(r_i)_{i\in W_t}\in\mathbb{R}^{W_t}_+\setminus\Omega_{W_t}$, there is no conformal (ideal) hyperbolic tetrahedron parameterized by $r^{W_t}_t$, since there is no corresponding four-sphere configuration $\mathcal{P}(r^{W_t}_t)$.      

We have defined the functional $\mathcal{S}_{W_t}$ on $\Omega_{W_t}$, it is a natural idea to extend the functional $\mathcal{S}_{W_t}$ from $\Omega_{W_t}$ to $\mathbb{R}_+^{W_t}$ and analysis the property of functional $\mathcal{S}_{W_t}$ on $\mathbb{R}_+^{W_t}$. By (\ref{func}), it is natural to extend the solid angle $\alpha_i^{t}$ defined on $\Omega_{W_t}$.   

Xu\cite{Xu20} introduce a natural extension of solid angle, for any $i\in W_t$, by Proposition \ref{amn}, the extended solid angle $\widetilde{\alpha}_i^{t}$ is defined as
\begin{equation}\label{soa}
\widetilde{\alpha}_i^{t}(r^{W_t}_t):=\begin{cases}
 \alpha_i^{t}(r^{W_t}_t), & r_{W_t}\in \Omega_{W_t}, \\
  2\pi,& r_{W_t}\in U_i^{W_t},\\
  0,&r_{W_t}\in\bigsqcup_{j\in W\setminus\{i\}} U_j^{W_t}. 
\end{cases}
\end{equation}

Note that the extended solid angle $\widetilde{\alpha}_i^{t}$ is piecewise constant on connected components of $\mathbb{R}_+^{W_t}\setminus\Omega_{W_t}$. Moreover, the extended solid angle $\widetilde{\alpha}_i^{W_t}$ defined on $\mathbb{R}_+^{W_t}$ is a continuous extension of the solid angle $\alpha_i^{t}$ defined on $\Omega_{W_t}$.  
   
We have extended the solid angle $\alpha_i^{t}$, then the next step is to extend the functional $\mathcal{S}_{W_t}$ from $\Omega_{W_t}$ to $\mathbb{R}_+^{W_t}$. We need the following fundamental $C^1$-smooth and convex extension theory, which is pioneered by Luo\cite{Luo11}.   

\begin{definition}
A differential 1-form $\omega=\sum_{i=1}^m f_i(x) d x_i$ on an open subset $U \subset \mathbb{R}^m$ is continuous if $f_i(x)(1\le i\le m)$ is a continuous function on $U$. A continuous 1-form $\omega$ is closed if $\int_{\partial \tau} \omega=0$ for all Euclidean triangle $\tau\subset U$.
\end{definition}

By the standard approximation theory, if $\omega$ is closed and $\gamma$ is a piecewise smooth null homologous loop in $U$, then $\int_\gamma \omega=0$. If $U$ is a simply connected domain, then the integral $F(x)=\int_{x_0}^x \omega$ is well-defined, independent of the choice of piecewise smooth paths in $U$ from $x_0$ to $x$. Moreover, the function $F(x)$ is $C^1$-smooth such that $\frac{\partial F(x)}{\partial x_i}=f_i(x)$ for all $1\le i\le m$.

\begin{lemma}[\cite{Luo11}, Corollary 2.6]\label{exten}
Given an open convex subset $X \subset \mathbb{R}^m$ and an open subset $U \subset X$ bounded by a real analytic codimension-1 submanifold in $X$. If 
\begin{enumerate}[1.]
    \item $\omega=\sum_{i=1}^m f_i(x) d x_i$ is a continuous closed 1-form on $U$ such that $F(x)=\int_{x_{0}}^x \omega$ is locally convex $($concave, resp.$)$ on $U$;
    \item $f_i(1\le i\le m)$ can be extended continuously to $X$ by constant functions to a function $\widetilde{f_i}$ on $X$,
\end{enumerate}
then $\widetilde{F}(x)=\int_{x_0}^{x} \sum_{i=1}^n \widetilde{f_i} d x_i$ is a $C^1$-smooth convex $($concave, resp.$)$ function on $X$, which is a $C^1$-extension of $F$.   
\end{lemma}

For any non-empty subset $W_t\subset V_t$, the extened functional $\widetilde{\mathcal{S}}_{W_t}$ on $\mathbb{R}_+^{W_t}$ is defined as
\begin{equation}\label{k4}
\widetilde{\mathcal{S}}_{W_t}(r_{W_t}):= \mathcal{S}_{W_t}(\mathbf{1}_{W_t})+\int_{\mathbf{1}_{W_t}}^{r_{W_t}}\sum_{i\in W_t}\widetilde{\alpha}_i^{t}dr_i,~r_{W_t}=(r_i)_{i\in W_t}\in\mathbb{R}_+^{W_t}. 
\end{equation}
It is not difficult to check that $\widetilde{\mathcal{S}}_{W_t}$ is well-defined. Moreover, by Proposition \ref{piece}, Proposition \ref{negac}, Lemma \ref{exten}, (\ref{ki4}), (\ref{fun3}) and (\ref{k4}), $\widetilde{\mathcal{S}}_{W_t}$ is $C^1$-smooth concave on $\mathbb{R}_+^{W_t}$ and is a $C^1$-extension of the functional $\mathcal{S}_{W_t}$.   

\section{Global rigidity of (ideal) hyperbolic sphere packings on 3-manifolds}\label{s8}

In this section, we study the extension of the hyperbolic admissible space on compact triangulated 3-manifolds and prove Theorem \ref{theo1}. As a consequence, we obtain Corollary \ref{copo1}, which gives the global rigidity of (ideal) hyperbolic sphere packings.

Moreover, we use the symbol ``$\Omega$'' instead of $\Omega^{\mathbb{H}}$ to represent this hyperbolic admissible space on 3-manifolds for the convenience of notations.

Given a compact 3-dimensional manifold $M$ with a triangulation $\mathcal{T}=(V,E,F,T)$, where $V,E,F,T$ denote the set of vertices, edges, faces and tetrahedra, respectively.   

Define $\widehat{\mathbb{R}_{+}^{V}}:=\{r=(r_i)_{i\in V}:0<r_i\le\infty~ \text{for all}~ i\in V\}$, then for any subset $W\subset V$, the space $\widehat{\mathbb{R}_{+}^{V}}$ has the subspace 
$$
\mathcal{R}^{W}:=\left \{ r^W=(r_i)_{i\in V}\in\widehat{\mathbb{R}_{+}^{V}}:  r_{W}=(r_i)_{i\in W}\in\mathbb{R}^{W}_{+},~ r_{V\setminus W}=(r_i)_{i\in V\setminus W}=(\infty)_{i\in V\setminus W}\right\},
$$
where $\mathcal{R}^{\varnothing}=\{{\boldsymbol{\infty}}\}$($\boldsymbol{\infty}=(\infty)_{i\in V}$) if $W=\varnothing$ and $\mathcal{R}^{V}=\mathbb{R}^{V}_+$ if $W=V$. 

It is not difficult to check that there exists a cellular decomposition of $\widehat{\mathbb{R}_{+}^{V}}$ and the subspace $\mathcal{R}^{W}$ is a $|W|$-dimensional cell, namely
\begin{equation}
\widehat{\mathbb{R}_{+}^{V}}=\bigsqcup_{W\subset V_t}\mathcal{R}^{W}=\mathbb{R}_{+}^{V}\bigsqcup\partial_{\infty}\mathbb{R}_{+}^{V},
\end{equation}
where $\partial_{\infty}\mathbb{R}_{+}^{V}:=\bigsqcup_{W\subsetneq V}\mathcal{R}^{W}$ is called \emph{boundary at infinity of} $\mathbb{R}_{+}^{V}$.

For any non-empty subset $W\subset V$, define the following projection  
\begin{equation}
		\begin{aligned}
			p^{W}: \mathcal{R}^{W} & \longrightarrow\mathbb{R}^{W}_{+}  \\
			r^W=(r_i)_{i\in V} & \longmapsto r_{W}=(r_i)_{i\in W},
		\end{aligned}
\end{equation}
the projections $p^{W}$ is a homeomorphism from $\mathcal{R}^{W}$ to $\mathbb{R}_{+}^{W}$ and define $r^{W}:=(p^{W})^{-1}(r_W)\in\mathcal{R}^{W}$ for all $r_{W}\in \mathbb{R}^{W}_{+}$.

Recall that the extended hyperbolic admissible space on compact triangulated 3-manifold $(M,\mathcal{T})$ is defined as
$$
\widehat{\Omega}:=\left \{ r=(r_i)_{i\in V}\in\widehat{\mathbb{R}_{+}^{V}}:Q^{\mathbb{H}}_3(r_t)>0,~\text{for all}~t\in T \right\}.
$$
where $r_t=(r_i)_{i\in V_t}$ and $V_t$ is the vertex set of the tetrahedron $t\in T$.

For any subset $W\subset V$, the extended space $\widehat{\Omega}$ has the subspace 
$$
\Omega^{W}:=\left \{ r=(r_i)_{i\in V}\in\mathcal{R}^{W}:  Q^{\mathbb{H}}_3(r_t)>0,~\text{for all}~t\in T\right\},
$$
where $\Omega^{\varnothing}=\{\boldsymbol{\infty}\}$($\boldsymbol{\infty}=(\infty)_{i\in V}$) if $W=\varnothing$ and $\Omega^{V}=\Omega$ if $W=V$. 

It is not difficult to check that there exists a cellular decomposition of $\widehat{\Omega}$ and the subspace $\Omega^{W}$ is a $|W|$-dimensional cell, namely
\begin{equation}\label{aa4}
\widehat{\Omega}=\bigsqcup_{W\subset V}\Omega^{W}=\Omega\bigsqcup\partial_{\infty}\Omega,
\end{equation} 
where $\partial_{\infty}\Omega:=\bigsqcup_{W\subsetneq V}\Omega^{W}$ is the ideal hyperbolic admissible space on compact triangulated 3-manifold $(M,\mathcal{T})$. Moreover, the space $\mathbb{R}^{W}_{+}$ has the subspace 
$$
\Omega_{W}:=\{r_{W}=(r_i)_{i\in W}\in\mathbb{R}_+^{W}:Q^{\mathbb{H}}_3(r^{W}_t)>0,~\text{for all }t\in T\}
$$
for all non-empty subset $W\subset V$. Then we have $p^{W}(\Omega^{W})=\Omega_{W}$ and $\Omega^{V}=\Omega_{V}=\Omega$.

Moreover, for any non-empty subset $W\subset V$, define the following functional $\mathcal{F}_W$ on $\Omega_W$, namely
\begin{equation}
\mathcal{F}_{W}(r_{W}):=\sum\limits_{i\in W\setminus V^b}4\pi r_i+\sum\limits_{i\in W\cap V^b}2\pi r_i-\sum_{t\in T_W}\mathcal{S}_{W\cap V_t},~r_{W}=(r_i)_{i\in W}\in\Omega_{W},
\end{equation} 
where $V^b$ is the boundary vertex set of $V$ and $T_W$ is the set of all tetrahedra with at least one vertex in $W$. Moreover, by (\ref{ki4}), we have
\begin{equation}\label{acala}
\begin{aligned}
d\mathcal{F}_{W}&=\sum\limits_{i\in W\setminus V^b}4\pi dr_i+\sum\limits_{i\in W\cap V^b}2\pi dr_i-\sum_{t\in T_W}\sum_{i\in W\cap V_t}\alpha_i^{t}dr_i\\
&=\sum\limits_{i\in W\setminus V^b}4\pi dr_i+\sum\limits_{i\in W\cap V^b}2\pi dr_i-\sum_{i\in W}\sum_{t\in T_i}\alpha_i^{t}dr_i\\
&=\sum_{i\in W}K_idr_i,
\end{aligned}
\end{equation}
where $T_i$ is the set of all tetrahedra with the vertex $i$ and $K_i$ is the combinatorial scalar curvature at the vertex $i\in W$ defined on $\Omega_W$, namely
\begin{equation}\label{curt}
K_i:=\begin{cases}
 4\pi-\sum_{t\in T_i}\alpha_i^{t}, & \text{ if }i~ \text{is an interior vertex},  \\
  2\pi-\sum_{t\in T_i}\alpha_i^{t},& \text{ if }i~ \text{is a boundary vertex}.
\end{cases}
\end{equation}
Moreover, by (\ref{acala}), we have
\begin{equation}\label{proj3}
\nabla\mathcal{F}_{W}=K_W=(K_i)_{i\in W},~  \text{Hess}~\mathcal{F}_{W}=-\left(\frac{\partial\alpha_i}{\partial r_j}\right)_{i,j\in W},
\end{equation}
where $\alpha_i=\sum_{t\in T_i}\alpha_i^{t}$ for all $i\in W$.

By Proposition \ref{negac}, the functional $\mathcal{S}_{W\cap V_t}$ is strictly concave for all $t\in T_W$, which implies that the functional $\mathcal{F}_{W}$ is strictly convex on $\Omega_{W}$ for all non-empty subset $W\subset V$. Moreover, by (\ref{acala}), the functional $\mathcal{F}_{W}$ can be rewritten as
\begin{equation}\label{k2}
 \mathcal{F}_{W}(r_{W})= \mathcal{F}_{W}(\mathbf{1}_{W})+\int_{\mathbf{1}_{W}}^{r_{W}}\sum_{i\in W}K_idr_i,~r_{W}=(r_i)_{i\in W}\in\Omega_{W},    
\end{equation}
where $\mathbf{1}_{W}=(1)_{i\in W}\in\Omega_{W}$.

In Section \ref{boundary}, we have defined the extended solid angle $\widetilde{\alpha}_i^{t}$, which is a continuous extension of the solid angle $\alpha_i^{t}$. Moreover, by (\ref{curt}), we can define the extended combinatorial scalar curvature $\widetilde{K}_i$ at the vertex $i$ on $\mathbb{R}_+^W$, namely 
\begin{equation}
\widetilde{K}_i:=\begin{cases}
 4\pi-\sum_{t\in T_i}\widetilde{\alpha}_i^{t}, & \text{ if }i~ \text{is an interior vertex},  \\
  2\pi-\sum_{t\in T_i}\widetilde{\alpha}_i^{t},& \text{ if }i~ \text{is a boundary vertex}.
\end{cases}
\end{equation}

It is not difficult to check that $\widetilde{K}_i$ is a continuous extension of the combinatorial scalar curvature $K_i$. Moreover, we can define the following extended functional $\widetilde{\mathcal{F}}_W$ on $\mathbb{R}_+^W$, namely
\begin{equation}\label{k1}
 \widetilde{\mathcal{F}}_{W}(r_{W}):= \mathcal{F}_{W}(\mathbf{1}_{W})+\int_{\mathbf{1}_{W}}^{r_{W}}\sum_{i\in W}\widetilde{K}_idr_i,~r_{W}=(r_i)_{i\in W}\in\mathbb{R}_+^W.    
\end{equation}
It is not difficult to check that the extended functional $\widetilde{\mathcal{F}}_{W}$ is well-defined and we have
\begin{equation}
\nabla\widetilde{\mathcal{F}}_W=\widetilde{K}_W=(\widetilde{K}_i)_{i\in W}.
\end{equation}
Moreover, by Lemma \ref{exten}, (\ref{curt}), (\ref{k2}) and (\ref{k1}), $\widetilde{\mathcal{F}}_{W}$ is $C^1$-smooth convex on $\mathbb{R}_+^{W}$ and is a $C^1$-extension of the functional $\mathcal{F}_{W}$ for all non-empty subset $W\subset V$.

\begin{lemma}\label{o1}
The gradient $\nabla\mathcal{F}_W$ of the functional $\mathcal{F}_W$ is an injective map on $\Omega_W$ for all non-empty subset $W\subset V$.
\end{lemma}

\begin{proof}
For any non-empty subset $W\subset V$, if $r_W^1,r_W^2\in\Omega_W$ and $\nabla\mathcal{F}_W(r_W^1)=\nabla\mathcal{F}_W(r_W^2)$, then suppose that 
\begin{equation}\label{p1}
\nabla\mathcal{F}_W(r_W^1)=\nabla\mathcal{F}_W(r_W^2)=K_W=(K_i)_{i\in W}
\end{equation}
and we define the following functional $\mathcal{E}_W$ on $\mathbb{R}_+^W$, namely 
\begin{equation}\label{p3}
\mathcal{E}_W(r_W):=\widetilde{\mathcal{F}}_W(r_W)-\sum_{i\in W}K_ir_i,~r_W=(r_i)_{i\in W}\in \mathbb{R}_+^W,
\end{equation}
then we have
\begin{equation}\label{p2}
\nabla\mathcal{E}_W=\nabla\widetilde{\mathcal{F}}_W-K_W.
\end{equation}

Since $\widetilde{K}_i$ and $\widetilde{\mathcal{F}}_W$ are the extensions of the combinatorial scalar curvature $K_i$ and the functional $\mathcal{F}_W$, respectively, then by (\ref{p2}), $r_W^1,r_W^2$ are the critical points of $\mathcal{E}_W$. Moreover, by (\ref{p3}), the functional $\mathcal{E}_W$ is convex on $\mathbb{R}^W_+$ and strictly convex on $\Omega_W$, which implies that $r_W^1=r_W^2$ and $\nabla\mathcal{F}_W$ is an injective map on $\Omega_W$.    
\end{proof}

Now we give the proofs of Theorem \ref{theo1} and Corollary \ref{copo1}.

\pf[Proof of Theorem \ref{theo1}]
If $W=\varnothing$, then the proof is trivial since $\Omega^{\varnothing}=\{\boldsymbol{\infty}\}$.

For any non-empty subset $W\subset V$, define the following projection
\begin{equation}\label{project1}
		\begin{aligned}
			p_1^{W}: \Omega^{W} & \longrightarrow\Omega_W  \\
			r=(r_i)_{i\in V} & \longmapsto r_{W}=(r_i)_{i\in W},
		\end{aligned}
\end{equation}
the projection $p_1^W$ is a bijective map from $\Omega^{W}$ to $\Omega_W$. Moreover, define the following projection
\begin{equation}\label{project2}
		\begin{aligned}
			p_2^{W}: \mathbb{R}^V & \longrightarrow\mathbb{R}^W  \\
			r=(r_i)_{i\in V} & \longmapsto r_{W}=(r_i)_{i\in W}.
		\end{aligned}
\end{equation}

By (\ref{proj3}), (\ref{project1}) and (\ref{project2}), we have the following commutative diagram:
\begin{equation}
\begin{tikzcd}\label{tiz1}
  \Omega^W \arrow[r, "\mathbb{K}"]\arrow[d,"p_1^W"] & \mathbb{R}^V\arrow[d,"p_2^W"] \\
  \Omega_W\arrow[r, "\nabla\mathcal{F}_W"]& \mathbb{R}^W.
\end{tikzcd}
\end{equation}

By Lemma \ref{o1} and (\ref{tiz1}), the composite map $ 
p_2^W\circ\mathbb{K}=\nabla\mathcal{F}_W\circ p_1^W$ is an injective map, which implies that $\mathbb{K}$ is an injective map on the cell $\Omega^W$. This completes the proof.
\pfd

\pf[Proof of Corollary \ref{copo1}] By Theorem \ref{theo1}, $\mathbb{K}$ is an injective map on the cell $\Omega^W$ for all subset $W\subset V$. Moreover, by (\ref{aa4}), $\mathbb{K}$ is an injective map on $\widehat{\Omega}^{\mathbb{H}}$. This completes the proof. 
\pfd

\section{Characterization of conformal n-simplices in constant curvature spaces}\label{s9}

In this section, we prove Theorem \ref{highpolyhedron}, which gives the characterization of the conformal hyperbolic (Euclidean, spherical, ideal hyperbolic, resp.) $n$-simplices. 

The following lemma is an important observation on non-degenerate $n$-simplices in constant curvature spaces.

\begin{lemma}[\cite{Taylor05}, Proposition 2.4.1]\label{mat}
Given a dimension $n\ge 3$, then the hyperbolic $(\text{Euclidean},$ $\text{spherical, resp}.)$ $n$-simplex $\Delta^n\subset\mathbb{H}^n(\mathbb{E}^n, \mathbb{S}^n, \text{resp}.)$ with edge length $\ell_{ij}$ for all geodesic edges $e_{ij}$ of $\Delta^n$ is non-degenerate if and only if the matrix $L:=(L_{ij})_{ij}$ satisfies
\begin{align*}
&\text{all principal minors are negative and $|L|<0$}\\
(&\text{all principal minors are positive and $|L|>0$,}\\
&\text{all principal minors are positive and $|L|>0$, resp.}),
\end{align*}
where 
$$
L_{ij}=\begin{cases}
-1(\ell_{in+1}^2,1,\text{resp}.),  & \text{if}~i=j, \\
 -\cosh\ell_{ij}(\dfrac{\ell_{in+1}^2+\ell_{jn+1}^2-\ell_{ij}^2}{2},\cos\ell_{ij},\text{resp}.), & \text{if} ~i\ne j.
\end{cases}
$$
\end{lemma}

Now we give the proof of Theorem \ref{highpolyhedron}.

\pf[Proof of Theorem \ref{highpolyhedron}] (Hyperbolic case.) ($\Rightarrow$) Since $\ell_{ij}=r_i+r_j$ for all geodesic edge $e_{ij}$ of $\Delta^n$, then by (\ref{c3}), we have
\begin{equation}\label{highhyperbolic1}
L_{ij}=
\begin{cases}
-1=-s_i^2(c_i^2-1),&1\le i=j\le n+1\\
-\cosh(r_i+r_j)=-s_is_j(c_ic_j+1),&1\le i\ne j\le n+1,\\
\end{cases}
\end{equation}
where $s_i=\sinh r_i,c_i=\coth r_i$ for all $1\le i\le n+1$.

We denote the column vector formed by $c_1,\cdots,c_{n+1}$ by $\delta$, namely $\delta=(c_1,\cdots,c_{n+1})^T$, by (\ref{highhyperbolic1}), the determinant  
\begin{align}\label{highhyperbolic2}
|L|=(-1)^{n+1}s_1^2\cdots s_{n+1}^2|A|,
\end{align}
where the matrix
$$
A=(\alpha_1,\cdots,\alpha_{n+1}),~\alpha_i=c_i\delta+\beta_i
$$ 
and $\beta_i=(1,\cdots,1,-1(\text{i-th}),1,\cdots,1)^T$ for all $1\le i\le n+1$, note that $\beta_1,\cdots,\beta_{n+1}$ form a basis of $\mathbb{R}^{n+1}$. Moreover, we have
\begin{equation}\label{highhyperbolic3}
\begin{aligned}
|A|&=|\alpha_1,\cdots,\alpha_{n+1}|\\
&=|c_1\delta+\beta_1,\cdots,c_{n+1}\delta+\beta_{n+1}|\\
&=\sum\limits_{i=1}^{n+1}c_i|B_{i}|+|B|,
\end{aligned}
\end{equation}
where $B=(\beta_1,\cdots,\beta_{n+1})$ and $B_{i}=(\beta_1,
\cdots,\beta_{i-1},\delta,\beta_{i+1},\cdots,\beta_{n+1})$ for all $1\le i\le n+1$, note that $|B|=(-1)^n(n-1)2^n$.

Choose the column vectors $\gamma_i=(1,\cdots,1,2-n(\text{i-th}),1,\cdots,1)^T$($1\le i\le n+1$), then we have
\begin{align}\label{hyperbolicorthogonal}
\langle\beta_i,\gamma_j\rangle=0,~1\le i\ne j\le n+1.
\end{align}
By (\ref{hyperbolicorthogonal}), $\beta_1,\cdots,\beta_{i-1},\gamma_i,\beta_{i+1},\cdots,\beta_{n+1}$ form a basis of $\mathbb{R}^{n+1}$ for all $1\le i\le n+1$ and there exist $\lambda_1^{(i)},\cdots,\lambda_{n+1}^{(i)}, \mu_1^{(i)},\cdots,$ $\mu_{n+1}^{(i)}\in\mathbb{R}$ such that 
\begin{equation}\label{ver}
\delta=\sum_{j\ne i}\lambda_j^{(i)}\beta_j+\lambda_i^{(i)}\gamma_i,~\beta_i=\sum_{j\ne i}\mu_j^{(i)}\beta_j+\mu_i^{(i)}\gamma_i,~1\le i\le n+1.
\end{equation}
Note that $\mu_i^{(i)}\ne 0$ for all $1\le i\le n+1$, since $\beta_1,\cdots,\beta_{n+1}$ are linearly independent. Moreover, by (\ref{ver}), we have  
\begin{align}
&\delta=\dfrac{\lambda_i^{(i)}}{\mu_i^{(i)}}\beta_i+(\text{linear\ combinations\ of}\ \beta_1,\cdots,\beta_{i-1},\beta_{i+1},\cdots,\beta_{n+1}),\nonumber\\
&|B_{i}|=|\beta_1,
\cdots,\beta_{i-1},\delta,\beta_{i+1},\cdots,\beta_{n+1}|=\dfrac{\lambda_i^{(i)}}{\mu_i^{(i)}}|B|,~1\le i\le n+1.\label{highhyperbolic5}
\end{align}

By (\ref{hyperbolicorthogonal}) and (\ref{ver}), we have
\begin{equation}\label{c}
 \dfrac{\lambda_i^{(i)}}{\mu_i^{(i)}}=\dfrac{\langle \delta,\gamma_i\rangle}{\langle \beta_i,\gamma_i\rangle}=\dfrac{\sum_{j=1}^{n+1}c_j-(n-1)c_i}{2(n-1)},~1\le i\le n+1.  
\end{equation}
Moreover, by (\ref{highhyperbolic5}) and (\ref{c}), we have  
\begin{equation}\label{f}
  |B_{i}|=(-1)^n2^{n-1}[\sum_{j=1}^{n+1}c_j-(n-1)c_i],~1\le i\le n+1. 
\end{equation}

By (\ref{highhyperbolic2}), (\ref{highhyperbolic3}) and (\ref{f}), we have
$$
|A|=(-1)^n2^{n-1}Q^{\mathbb{H}}_{n+1},~|L|=-2^{n-1}s_1^2\cdots s_{n+1}^2\mathcal{Q}^{\mathbb{H}}_{n+1},
$$
where $$
\mathcal{Q}^{\mathbb{H}}_{n+1}=(\sum\limits_{i=1}^{n+1}c_i)^2-(n-1)\sum\limits_{i=1}^{n+1}c_i^2+2(n-1)=Q^{\mathbb{H}}_n(r).
$$

By Lemma \ref{mat}, we have $$|L|=-2^{n-1}s_1^2\cdots s_{n+1}^2\mathcal{Q}^{\mathbb{H}}_{n+1}<0,$$ which implies that $Q^{\mathbb{H}}_n(r)=\mathcal{Q}^{\mathbb{H}}_{n+1}>0$.

$(\Leftarrow)$ For any $1\le k\le n+1$, define the (sub)matrix $L_k:=(L_{ij})_{1\le i,j\le k}$, namely
$$
L=\begin{pmatrix}
  L_{k}& *\\
  *& *
\end{pmatrix}.
$$

Similar to the proof of ``$\Leftarrow$'', we have
\begin{equation}\label{ab1}
|L_k|=-2^{k-2}s_1^2\cdots s_{k}^2\mathcal{Q}^{\mathbb{H}}_k,
\end{equation}
where  
$$
\mathcal{Q}^{\mathbb{H}}_k=(\sum\limits_{i=1}^{k}c_i)^2-(k-2)\sum\limits_{i=1}^{k}c_i^2+2(k-2).
$$
Note that 
\begin{equation}\label{ab2}
|L_1|=-1<0,~|L_2|=-\sinh^2\ell_{12}<0,~ |L_{n+1}|=|L|.
\end{equation}
Moreover, we have
\begin{equation}\label{ab}
\mathcal{Q}^{\mathbb{H}}_k=\frac{k-2}{k-1}\mathcal{Q}^{\mathbb{H}}_{k+1}+(k-1)(c_{k+1}-\frac{\sum_{i=1}^{k+1}c_i}{k-1})^2,~3\le k\le n.
\end{equation}

Since $Q^{\mathbb{H}}_n(r)=\mathcal{Q}^{\mathbb{H}}_{n+1}>0$, by (\ref{ab}), we have $\mathcal{Q}^{\mathbb{H}}_k>0$ for all $3\le k\le n$. Then by (\ref{ab1}) and (\ref{ab2}), we have
$$
 |L_k|=-2^{k-2}s_1^2\cdots s_{k}^2\mathcal{Q}^{\mathbb{H}}_k<0,~1\le k\le n+1.
$$

For any $k\times k(1\le k\le n+1)$ principal (sub)matrix $\widetilde{L}_{k}$ of the matrix $L$, there exists a non-singular matrix $P^{(k)}$ such that 
$$
\widetilde{L}:=P^{(k)T}LP^{(k)}=\begin{pmatrix}
  \widetilde{L}_{k}& *\\
  *& *
\end{pmatrix}.
$$

By the same argument as above, we have $|\widetilde{L}_{k}|<0$ for all $1\le k\le n+1$, which implies that all principal minors of the matrix $L$ are negative and $|L|<0$. Then by Lemma \ref{mat}, we complete the proof. 

(Euclidean case.) The proof is quite similar to the proof of hyperbolic case and we may abuse some notations. 

($\Rightarrow$) Since $\ell_{ij}=r_i+r_j$ for all geodesic edge $e_{ij}$ of $\Delta^n$, then we have
\begin{equation}\label{kf}
L_{ij}=
\begin{cases}
(r_i+r_{n+1})^2=r_i^2r_{n+1}^2(c_i+c_{n+1})^2,&1\le i=j\le n\\
r_ir_jr_{n+1}^2[(c_i+c_{n+1})(c_j+c_{n+1})-2c_{n+1}^2],&1\le i\ne j\le n,\\
\end{cases}
\end{equation}
where $c_i=1/r_i$ for all $1\le i\le n+1$.

We denote the column vector formed by $c_1+c_{n+1},\cdots,c_n+c_{n+1}$ by $\delta$, namely $\delta=(c_1+c_{n+1},\cdots,c_n+c_{n+1})^T$, by (\ref{kf}), the determinant  
\begin{align}\label{k12}
|L|=r_1^2\cdots r_n^2r_{n+1}^{2n}|A|,
\end{align}
where the matrix
$$
A=(\alpha_1,\cdots,\alpha_{n}),~\alpha_i=(c_i+c_{n+1})\delta-2c_{n+1}^2\beta_i 
$$
and $\beta_i=(1,\cdots,1,0(\text{i-th}),1,\cdots,1)^T$ for all $1\le i\le n$, note that $\beta_1,\cdots,\beta_{n}$ form a basis of $\mathbb{R}^{n}$. Moreover, we have
\begin{equation}\label{b1}
\begin{aligned}
|A|&=|\alpha_1,\cdots,\alpha_{n+}|\\
&=|(c_1+c_{n+1})\delta-2c_{n+1}^2\beta_1,\cdots,(c_n+c_{n+1})\delta-2c_{n+1}^2\beta_{n}|\\
&=(-2)^{n-1}c_{n+1}^{2(n-1)}(\sum\limits_{i=1}^{n}(c_i+c_{n+1})|B_{i}|-2c_{n+1}^2|B|),
\end{aligned}
\end{equation}
where $B=(\beta_1,\cdots,\beta_{n})$ and $B_{i}=(\beta_1,
\cdots,\beta_{i-1},\delta,\beta_{i+1},\cdots,\beta_{n})$ for all $1\le i\le n$, note that $|B|=(-1)^{n-1}(n-1)$.

Choose the column vectors $\gamma_i=(1,\cdots,1,2-n(\text{i-th}),1,\cdots,1)^T$($1\le i\le n$), then we have
\begin{align}\label{q1}
\langle\beta_i,\gamma_j\rangle=0,~1\le i\ne j\le n.
\end{align}
By (\ref{q1}), $\beta_1,\cdots,\beta_{i-1},\gamma_i,\beta_{i+1},\cdots,\beta_{n}$ form a basis of $\mathbb{R}^{n}$ for all $1\le i\le n$ and there exist $\lambda_1^{(i)},\cdots,\lambda_{n}^{(i)}, \mu_1^{(i)},\cdots, \mu_{n}^{(i)}\in\mathbb{R}$ such that 
\begin{equation}\label{m1}
\delta=\sum_{j\ne i}\lambda_j^{(i)}\beta_j+\lambda_i^{(i)}\gamma_i,~\beta_i=\sum_{j\ne i}\mu_j^{(i)}\beta_j+\mu_i^{(i)}\gamma_i,~1\le i\le n.
\end{equation}
Note that $\mu_i^{(i)}\ne 0$ for all $1\le i\le n$, since $\beta_1,\cdots,\beta_{n}$ are linearly independent. Moreover, by (\ref{m1}), we have  
\begin{align}
&\delta=\dfrac{\lambda_i^{(i)}}{\mu_i^{(i)}}\beta_i+(\text{linear\ combinations\ of}\ \beta_1,\cdots,\beta_{i-1},\beta_{i+1},\cdots,\beta_{n}),\nonumber\\
&|B_{i}|=|\beta_1,
\cdots,\beta_{i-1},\delta,\beta_{i+1},\cdots,\beta_{n}|=\dfrac{\lambda_i^{(i)}}{\mu_i^{(i)}}|B|,~1\le i\le n.\label{jh1}
\end{align}

By (\ref{q1}) and (\ref{m1}), we have
\begin{equation}\label{bn1}
 \dfrac{\lambda_i^{(i)}}{\mu_i^{(i)}}=\dfrac{\langle \delta,\gamma_i\rangle}{\langle \beta_i,\gamma_i\rangle}=\dfrac{\sum_{j=1}^{n}(c_j+c_{n+1})-(n-1)(c_i+c_{n+1})}{n-1},~1\le i\le n.  
\end{equation}
Moreover, by (\ref{jh1}) and (\ref{bn1}), we have  
\begin{equation}\label{pl1}
  |B_{i}|=(-1)^{n-1}[\sum_{j=1}^{n}(c_j+c_{n+1})-(n-1)(c_i+c_{n+1})],~1\le i\le n. 
\end{equation}

By (\ref{k12}), (\ref{b1}) and (\ref{pl1}), we have
$$
|A|=2^{n-1}c_{n+1}^{2(n-1)}\mathcal{Q}^{\mathbb{E}}_{n+1},~|L|=2^{n-1}r_1^2\cdots r_{n+1}^{2}\mathcal{Q}^{\mathbb{E}}_{n+1},
$$
where
$$
\mathcal{Q}^{\mathbb{E}}_{n+1}=(\sum_{i=1}^{n+1}c_i)^2-(n-1)\sum_{i=1}^{n+1}c_i^2=Q^{\mathbb{E}}_n(r).
$$

By Lemma \ref{mat}, we have $$|L|=2^{n-1}r_1^2\cdots r_{n+1}^{2}\mathcal{Q}^{\mathbb{E}}_{n+1}>0,$$ which implies that $Q^{\mathbb{E}}_n(r)=\mathcal{Q}^{\mathbb{E}}_{n+1}>0$.

$(\Leftarrow)$ For any $1\le k\le n+1$, define the (sub)matrix $L_k:=(L_{ij})_{1\le i,j\le k}$, namely
$$
L=\begin{pmatrix}
  L_{k}& *\\
  *& *
\end{pmatrix}.
$$

Similar to the proof of ``$\Leftarrow$'', we have
\begin{equation}\label{ab3}
|L_k|=2^{k-2}r_1^2\cdots r_{k}^2\mathcal{Q}^{\mathbb{E}}_k,
\end{equation}
where  
$$
\mathcal{Q}^{\mathbb{E}}_{k}=(\sum_{i=1}^{k}c_i)^2-(k-2)\sum_{i=1}^{k}c_i^2.
$$
Note that 
\begin{equation}\label{ab4}
|L_1|=\ell_{1n+1}^2>0,~|L_2|=\ell_{1n+1}^2\ell_{2n+1}^2\sin^2\gamma_{12}>0,~ |L_{n+1}|=|L|.
\end{equation}
Moreover, we have
\begin{equation}\label{ab5}
\mathcal{Q}^{\mathbb{E}}_k=\frac{k-2}{k-1}\mathcal{Q}^{\mathbb{E}}_{k+1}+(k-1)(c_{k+1}-\frac{\sum_{i=1}^{k+1}c_i}{k-1})^2,~3\le k\le n.
\end{equation}

Since $Q^{\mathbb{E}}_n(r)=\mathcal{Q}^{\mathbb{E}}_{n+1}>0$, by (\ref{ab5}), $\mathcal{Q}^{\mathbb{E}}_k>0$ for all $3\le k\le n$. Then by (\ref{ab3}) and (\ref{ab4}), we have
$$
 |L_k|=2^{k-2}r_1^2\cdots r_{k}^2\mathcal{Q}^{\mathbb{E}}_k>0,~1\le k\le n+1,
$$
which implies that all principal minors of the matrix $L$ are positive and $|L|>0$. Then by Lemma \ref{mat}, we complete the proof.

(Spherical case.) The proof is quite similar to that of hyperbolic case and we may abuse some notations. 

Since $\ell_{ij}=r_i+r_j$ for all $1\le i\ne j\le n+1$, then by (\ref{c3}), we have
\begin{equation}\label{k}
L_{ij}=
\begin{cases}
1=s_i^2(c_i^2+1),&1\le i=j\le n+1\\
\cos(r_i+r_j)=s_is_j(c_ic_j-1),&1\le i\ne j\le n+1,\\
\end{cases}
\end{equation}
where $s_i=\sin r_i,c_i=\cot r_i$ for all $1\le i\le n+1$.

We denote the column vector formed by $c_1,\cdots,c_{n+1}$ by $\delta$, namely $\delta=(c_1,\cdots,c_{n+1})^T$, by (\ref{k}), the determinant  
\begin{align}\label{ui}
|L|=s_1^2\cdots s_{n+1}^2|A|,
\end{align}
where the matrix
$$
A=(\alpha_1,\cdots,\alpha_{n+1}),~\alpha_i=c_i\delta+\beta_i
$$ 
and $\beta_i=(-1,\cdots,-1,1(\text{i-th}),-1,$ $\cdots,-1)^T$ for all $1\le i\le n+1$, note that $\beta_1,\cdots,\beta_{n+1}$ form a basis of $\mathbb{R}^{n+1}$. Moreover, we have
\begin{equation}\label{o}
\begin{aligned}
|A|&=|\alpha_1,\cdots,\alpha_{n+1}|\\
&=|c_1\delta+\beta_1,\cdots,c_{n+1}\delta+\beta_{n+1}|\\
&=\sum\limits_{i=1}^{n+1}c_i|B_{i}|+|B|,
\end{aligned}
\end{equation}
where $B=(\beta_1,\cdots,\beta_{n+1})$ and $B_{i}=(\beta_1,
\cdots,\beta_{i-1},\delta,\beta_{i+1},\cdots,\beta_{n+1})$ for all $1\le i\le n+1$, note that $|B|=-(n-1)2^n$.

Choose the column vectors $\gamma_i=(-1,\cdots,-1,n-2(\text{i-th}),-1,\cdots,-1)^T$($1\le i\le n+1$), then we have
\begin{align}\label{q}
\langle\beta_i,\gamma_j\rangle=0,~1\le i\ne j\le n+1.
\end{align}
By (\ref{q}), $\beta_1,\cdots,\beta_{i-1},\gamma_i,\beta_{i+1},\cdots,\beta_{n+1}$ form a basis of $\mathbb{R}^{n+1}$ for all $1\le i\le n+1$ and there exist $\lambda_1^{(i)},\cdots,\lambda_{n+1}^{(i)}, \mu_1^{(i)},\cdots,$ $\mu_{n+1}^{(i)}\in\mathbb{R}$ such that 
\begin{equation}\label{m}
\delta=\sum_{j\ne i}\lambda_j^{(i)}\beta_j+\lambda_i^{(i)}\gamma_i,~\beta_i=\sum_{j\ne i}\mu_j^{(i)}\beta_j+\mu_i^{(i)}\gamma_i,~1\le i\le n+1.
\end{equation}
Note that $\mu_i^{(i)}\ne 0$ for all $1\le i\le n+1$, since $\beta_1,\cdots,\beta_{n+1}$ are linearly independent. Moreover, by (\ref{m}), we have  
\begin{align}
&\delta=\dfrac{\lambda_i^{(i)}}{\mu_i^{(i)}}\beta_i+(\text{linear\ combinations\ of}\ \beta_1,\cdots,\beta_{i-1},\beta_{i+1},\cdots,\beta_{n+1}),\nonumber\\
&|B_{i}|=|\beta_1,
\cdots,\beta_{i-1},\delta,\beta_{i+1},\cdots,\beta_{n+1}|=\dfrac{\lambda_i^{(i)}}{\mu_i^{(i)}}|B|,~1\le i\le n+1.\label{jh}
\end{align}

By (\ref{q}) and (\ref{m}), we have
\begin{equation}\label{bn}
 \dfrac{\lambda_i^{(i)}}{\mu_i^{(i)}}=\dfrac{\langle \delta,\gamma_i\rangle}{\langle \beta_i,\gamma_i\rangle}=-\dfrac{\sum_{j=1}^{n+1}c_j-(n-1)c_i}{2(n-1)},~1\le i\le n+1.  
\end{equation}
Moreover, by (\ref{jh}) and (\ref{bn}), we have  
\begin{equation}\label{pl}
  |B_{i}|=2^{n-1}[\sum_{j=1}^{n+1}c_j-(n-1)c_i],~1\le i\le n+1. 
\end{equation}

By (\ref{ui}), (\ref{o}) and (\ref{pl}), we have
\begin{equation}\label{as1}
|A|=2^{n-1}\mathcal{Q}^{\mathbb{S}}_{n+1},~|L|=2^{n-1}s_1^2\cdots s_{n+1}^2\mathcal{Q}^{\mathbb{S}}_{n+1},
\end{equation}
where
$$
\mathcal{Q}^{\mathbb{S}}_{n+1}=(\sum_{i=1}^{n+1}c_i)^2-(n-1)\sum_{i=1}^{n+1}c_i^2-2(n-1)=Q^{\mathbb{S}}_n(r).
$$

By Lemma \ref{mat} and (\ref{as1}), we have $$|L|=2^{n-1}s_1^2\cdots s_{n+1}^2\mathcal{Q}^{\mathbb{S}}_{n+1}>0,$$ which implies that $Q^{\mathbb{S}}_n(r)=\mathcal{Q}^{\mathbb{S}}_{n+1}>0$.

$(\Leftarrow)$ For any $1\le k\le n+1$, define the (sub)matrix $L_k:=(L_{ij})_{1\le i,j\le k}$, namely
$$
L=\begin{pmatrix}
  L_{k}& *\\
  *& *
\end{pmatrix}.
$$

Similar to the proof of ``$\Leftarrow$'', we have
\begin{equation}\label{ab6}
|L_k|=2^{k-2}s_1^2\cdots s_{k}^2\mathcal{Q}^{\mathbb{S}}_{k},
\end{equation}
where  
$$
\mathcal{Q}^{\mathbb{S}}_{k}=(\sum_{i=1}^{k}c_i)^2-(k-2)\sum_{i=1}^{k}c_i^2-2(k-2).
$$
Note that 
\begin{equation}\label{ab7}
|L_1|=1>0,~|L_2|=\sin^2\ell_{12}>0,~ |L_{n+1}|=|L|.
\end{equation}
Moreover, we have
\begin{equation}\label{ab8}
\mathcal{Q}^{\mathbb{S}}_k=\frac{k-2}{k-1}\mathcal{Q}^{\mathbb{S}}_{k+1}+(k-1)(c_{k+1}-\frac{\sum_{i=1}^{k+1}c_i}{k-1})^2,~3\le k\le n.
\end{equation}

Since $Q^{\mathbb{S}}_n(r)=\mathcal{Q}^{\mathbb{S}}_{n+1}>0$, by (\ref{ab8}), $\mathcal{Q}^{\mathbb{S}}_k>0$ for all $3\le k\le n$. Then by (\ref{ab6}) and (\ref{ab7}), we have
$$
 |L_k|=2^{k-2}s_1^2\cdots s_{k}^2\mathcal{Q}^{\mathbb{S}}_{k}>0,~1\le k\le n+1,
$$
which implies that all principal minors of the matrix $L$ are positive and $|L|>0$. Then by Lemma \ref{mat}, we complete the proof. 

(Ideal hyperbolic case.) ($\Rightarrow$) Since the ideal hyperbolic $n$-simplex $\Delta^n$ is conformal, then by definition, there exists a sequence of conformal hyperbolic $n$-simplices $\{\Delta^n_{m}\}_{m=1}^{\infty}\subset\mathbb{H}^n$ such that $\Delta^n_{m}\to\Delta^n(m\rightarrow\infty)$. 

By the proof of hyperbolic case, the conformal hyperbolic $n$-simplex $\Delta^n_{m}(m\ge1)$ can be parameterized by a $r^{(m)}=(r_i^{(m)})_{1\le i\le n+1}\in\mathbb{R}^{n+1}_+$, which satisfies $Q_{n}^{\mathbb{H}}(r^{(m)})>0$ for all $m\ge1$. Moreover, $\Delta^n_{m}(m\ge1)$ corresponds to a sphere configuration $\mathcal{P}^{(m)}=\{S_i^{(m)}\}_{1\le i\le n+1}$ parameterized by the $r^{(m)}$, which satisfies
\begin{enumerate}[a.]
    \item  $S_i^{(m)}\subset\mathbb{H}^n$ is a hyperbolic $(n-1)$-sphere with radius $r_i^{(m)}$ for all $1\le i\le n+1$; 
    \item The $(n-1)$-spheres $\{S_i^{(m)}\}_{1\le i\le n+1}$ are mutually externally tangent;
     \item The centers of the $(n-1)$-spheres $\{S_i^{(m)}\}_{1\le i\le n+1}$ are not all on the same hyperplane.
\end{enumerate}    

Since $\Delta^n_{m}\to\Delta^n(m\rightarrow\infty)$, then $\mathcal{P}^{(m)}\to\mathcal{P},~r^{(m)}\to r(m\to\infty)$, where $\mathcal{P}=\{S_i\}_{1\le i\le n+1}$ is the sphere configuration parameterized by the $r$, which corresponds to the conformal ideal hyperbolic $n$-simplex $\Delta^n$ and satisfies     
\begin{enumerate}[I.]
    \item  $S_i\subset\mathbb{H}^n$ is a hyperbolic $(n-1)$-sphere with radius $r_i$ if $r_i\in\mathbb{R}_+$; 
    \item $S_i\subset\overline{\mathbb{H}^n}$ is a $(n-1)$-horosphere if $r_i=\infty$;
    \item The $(n-1)$-(horo)spheres $\{S_i\}_{1\le i\le n+1}$ are mutually externally tangent;
     \item The centers of the $(n-1)$-(horo)spheres $\{S_i\}_{1\le i\le n+1}$ are not all on the same hyperplane.
\end{enumerate}    

Since $r^{(m)}\to r(m\to\infty)$ and $Q_{n}^{\mathbb{H}}(r^{(m)})>0$ for all $m\ge1$, then we have $Q_{n}^{\mathbb{H}}(r)\ge0$. If $Q_{n}^{\mathbb{H}}(r)=0$, then by higher-dimensional Descartes circle Theorem (see Remark \ref{aa2}), the centers of $\{S_i\}_{1\le i\le n+1}$ are all on the same hyperplane, which contradicts IV. Then we obtain $Q_{n}^{\mathbb{H}}(r)>0$. 

($\Leftarrow$) If $Q_{n}^{\mathbb{H}}(r)>0$, then by higher-dimensional Descartes circle Theorem, the ideal hyperbolic $n$-simplex $\Delta^n$ corresponds to a sphere configuration $\mathcal{P}$ parameterized by the $r$, which satisfies I, II, III, IV. Moreover, by IV, $\Delta^n$ is non-degenerate.  

For the sphere configuration $\mathcal{P}$, there exists a sequence $\{r^{(m)}\}_{m=1}^{\infty}\subset\mathbb{R}^{n+1}_+$ and a sequence of sphere configurations $\{\mathcal{P}^{(m)}\}_{m=1}^{\infty}$ such that $r^{(m)}\to r$, $\mathcal{P}^{(m)}\to\mathcal{P}(m\to\infty)$, where $\mathcal{P}^{(m)}$ is parameterized by the $r^{(m)}$ and satisfies a, b, c for all $m\ge1$. Moreover, by c and higher-dimensional Descartes circle Theorem, we have $Q(r^{(m)})>0$ for all $m\ge1$.    

Since $\mathcal{P}^{(m)}\to\mathcal{P}(m\to\infty)$, then we have $\Delta_m^n\to\Delta^n(m\to\infty)$, where $\Delta_m^n$ is the hyperbolic $n$-simplex parameterized by the $r^{(m)}$ for all $m\ge1$. Since $Q(r^{(m)})>0(m\ge1)$, then by the proof of hyperbolic case, $\Delta_m^n$ is conformal for all $m\ge1$, which implies that the ideal hyperbolic $n$-simplex $\Delta^n$ is conformal.    
\pfd

\textbf{Acknowledgements.} 
We are deeply grateful to Professor Ze Zhou for his valuable suggestions and comments, such as the statement of Corollary \ref{rigidityoftetrahedron}. We also thank Dr. Xiang Zhu for reminding us that the Euclidean case of Corollary \ref{sumangle} was proved by Katsuura \cite{Katsuura2021Dihedral} after finishing the draft of this paper. The third author would like to thank Professor Jean-Marc Schlenker for several conversations concerning de sitter geometry during the International Congress of Mathematicians (ICM) 2026 breaks.    

Zunwu He is supported by NSF of China (Grant No.12301094), Guangzhou Basic and Applied Basic Research Foundation
(Grant
No.2024A04J3483) and the General Program of Guangdong Basic and Applied Basic Research
Foundation (Grant No.2025A1515010502).
Guangming Hu is supported by NSF of China (Grant No.12671093) and by Natural Science Research Start-up Foundation of Recruiting Talents of Nanjing University of Posts and Telecommunications (Grant No.NY224040).
Ziping Lei is partially supported by NSF of China (Grant No.12122119).


\bibliography{py1(1)}
\bibliographystyle{alpha}

\end{document}